\documentclass{amsart}
\newif\ifshort
\shortfalse

\ifshort
\usepackage[top=3.7cm, bottom=3.7cm, left=3.7cm, right=3.7cm]{geometry}
\fi

\usepackage{sty-en} 

\usepackage[pdftex]{hyperref}
\hypersetup{colorlinks,citecolor=blue,linktocpage,hyperindex=true,backref=true}

\usepackage{tikz-cd, bm}

\usepackage{todonotes}

\title[Compact moduli of Enriques surfaces]{Compact moduli 
  of Enriques surfaces with a numerical polarization of degree 2}
\date{December 27, 2023} 

\author{Valery Alexeev}
\ifshort
  \else
\address{Department of Mathematics, University of Georgia, Athens GA 30602, USA}
\email{valery@uga.edu}
\fi

\author{Philip Engel} 
\ifshort
  \else
\address{Department of Mathematics, University of Bonn, Mathematisches Institut, Bonn, Germany}
\email{engel@math.uni-bonn.de}
\fi

\author{D. Zack Garza} 
\ifshort
  \else
\address{Department of Mathematics, University of Georgia, Athens GA 30602, USA}
\email{zack@uga.edu}
\fi

\author{Luca Schaffler}
\ifshort
  \else
\address{Dipartimento di Matematica e Fisica, Universit\`a degli Studi Roma Tre, Largo San Leonardo Murialdo 1, 00146, Roma, Italy}
\email{luca.schaffler@uniroma3.it}
\fi

\begin{document}
\begin{abstract}
  We describe a geometric, stable pair compactification of the
  moduli space of Enriques surfaces with a numerical polarization of
  degree~$2$, and identify it with a semitoroidal compactification
  of the period space.
\end{abstract}
 
\maketitle
\setcounter{tocdepth}{1}
\tableofcontents

\section{Introduction} 
\label{sec:intro}

Enriques surfaces are quotients of K3 surfaces by basepoint free
involutions. They satisfy $2K\sim 0$ and $q=0$ and occupy a place
somewhere in between rational and K3 surfaces. Unlike K3 surfaces,
there are only finitely many moduli spaces of polarized Enriques surfaces, see
\cite{gritsenko2016moduli-polarized}. Each of them parameterizes the
same surfaces, with some finite data attached.

In this paper we consider the moduli space $\fen$ of pairs
$(Z,[\cL_Z])$, where $Z$ is an Enriques surface with $ADE$
singularities and $[\cL_Z]\in \Pic Z / \bZ_2$ is an ample numerical
polarization of degree~$2$.  Equivalently, this is the moduli space of
$ADE$ Enriques surfaces $(Z,\cM)$ with a $2$-divisible polarization
$\cM = \cL_Z^{\otimes 2} \in \Pic Z$ of degree~$8$.
The Baily-Borel compactification $\ofen\ubb$ of this space was
described by Sterk \cite{sterk1991compactifications-enriques1}. 

By the classification of big and nef linear system on Enriques
surfaces \cite{cossec1983projective-models}, see also
\cite{cossec1989enriques-surfaces, cossec2024enriques-surfaces1}, the
linear system $|\cM|$ is basepoint free and defines a double cover
$\rho\colon Z\to W$ to a quartic del Pezzo surface with $4A_1$ or
$A_3+2A_1$ singularities. The ramification divisor $R_Z\in |\cM|$ of
$\rho$ is ample and the pair $(Z,\epsilon R_Z)$ is log canonical for
any $0<\epsilon \ll 1$. Thus, the moduli space $\fen$ admits a
geometric, modular compactification $\ofen$ in the KSBA moduli space
of pairs $(Z,\epsilon R_Z)$ with semi log canonical singularities,
$K_Z\equiv 0$ and ample $\bQ$-Cartier divisor $R_Z$. See
\cite{kollar2023families-of-varieties} for their general theory.  Our
main result, in Section \ref{sec:semitoroidal}, is

\begin{theorem} The normalization of $\ofen$ is a semitoroidal
  compactification $\ofen^\fF$
  corresponding to a collection
  $\fF = \{\fF^k\}_{k=1,2,3,4,5}$ of explicit semifans, one
  for each $0$-cusp of $\fen$, and it is dominated by a toroidal
  compactification $\ofen\ucox$ for a collection
  $\fF\cox = \{\fF^k\cox\}_{k=1,2,3,4,5}$ of Coxeter fans.
\end{theorem}

In Sections~\ref{sec:abcde}, \ref{sec:stable-models} we describe all
stable pairs parametrized by the boundary of $\ofen$. The irreducible
components of these pairs turn out to be surfaces that naturally
correspond to the ABCDE Dynkin diagrams. They generalize the ADE
surfaces of \cite{alexeev17ade-surfaces} that appear on the boundary
of K3 moduli spaces.

In Section~\ref{sec:ias} we also provide a detailed description of
some nice models of Enriques degenerations, which are slightly more
singular than simple normal crossing. Instead, they are dlt
(divisorially log terminal, see
\cite[Def. 2.37]{kollar1998birational-geometry}).

Just as weighted graphs encode semistable degenerations of curves,
$K$-trivial semistable (i.e.~Kulikov) degenerations of K3 surfaces $\cX\to (C,0)$ 
are encoded by integral-affine structures on the $2$-sphere, or $\ias$
\cite{gross2015mirror-symmetry-for-log, 
engel2018looijenga, engel2021smoothings}.
An $\ias$ is a collection
of local embeddings of $S^2$ minus a finite set into the flat plane, 
which differ on overlaps by $\SL_2(\bZ)\ltimes \bZ^2$. Complete 
triangulations of $\ias$ which take their vertices in $\bZ^2$, under the local
embedding, describe the dual complexes $\Gamma(\cX_0)$ of 
Kulikov degenerations. In this paper, we realize Diagram (\ref{eq:basic})
on the integral-affine level, by constructing $\Gamma(\cX_0)$ together with
two commuting involutions~$\iota_{\enr, \,{\rm IA}}$ and $\iota_{\dP,{\rm IA}}$
of the $\ias$.

The quotient of $\Gamma(\cX_0)$ by $\iota_{\enr, \,{\rm IA}}$ is either an integral-affine
structure on a disk $\bD^2$ or a real-projective plane $\bR\bP^2$. These
 ${\rm IA}\bD^2$ and ${\rm IA}\bR\bP^2$ are the dual complexes $\Gamma(\cZ_0)$
 of particularly nice dlt models of Enriques surface degenerations $\cZ\to (C,0)$.
 From these dlt models, one can extract a completely explicit description
 of the stable limit of any degeneration in $\ofen$ from Hodge-theoretic data.

The validity of our description of these Kulikov, dlt, and stable models relies
on the general theory of compactifications of moduli spaces of K3 surfaces developed
by the first two authors
\cite{alexeev2023compact,alexeev2022mirror-symmetric}. Most relevant
to the situation at hand,
\cite{alexeev2022mirror-symmetric} considers the
$75$ moduli spaces of K3 surfaces with a non-symplectic
involution. For $50$ of them, the ramification divisor $R$ contains a
component $C$ of genus $g(C)\ge2$ providing a polarization. For these,
\cite{alexeev2022mirror-symmetric} describes the Kulikov models
and KSBA compactification for the pairs $(X,\epsilon C)$.

Crucial for us here is the
compactified moduli space $\oF_{(2,2,0)}$ of K3 surfaces of degree~$4$
with a del Pezzo involution $\idp$ corresponding, generically, to double covers
of $\bP^1\times \bP^1$ branched over a curve of bidegree $(4,4)$.
By immersing the moduli space $\ofen\to \oF_{(2,2,0)}$
and understanding the Enriques involution on the fibers, 
we give a description of $\ofen$ and its universal family.

\medskip

The plan of the paper is as follows. In Section~\ref{sec:setup} we
discuss the model of Enriques surfaces which we use in this paper.
We also recall the description of the moduli space $\fen$ after Sterk
\cite{sterk1991compactifications-enriques1}.
Then we describe morphisms between $\fen$,
$\fhyp$ and the moduli $\fken$ of unpolarized Enriques surfaces.
Finally, we briefly recall the theory of KSBA
stable pairs and their compact moduli.

In Section~\ref{sec:cusps} we recall the cusp diagrams of $\fhyp$,
$\fken$, $\fen$ and determine how they map to each other. Next, we
describe Coxeter diagrams associated to each of the five $0$-cusps and
the $9$ $1$-cusps of $\fen$. The diagrams are the same as in Sterk
\cite{sterk1991compactifications-enriques1} but we find them
in a different way, ``folding by involutions'' the Coxeter diagrams of
the lattices $U\oplus E_8^2$ and $U(2)\oplus E_8^2$ corresponding
to the two $0$-cusps of $\fhyp$. This is the combinatorial heart of
the paper.  An idea from \cite{alexeev2023stable-pair,
  alexeev2022mirror-symmetric} employed here is that one can read off
degenerations of K3 surfaces directly from Coxeter
diagrams. Consequently, we are able to read off degenerations of
Enriques surfaces from the Coxeter diagrams of $\fen$.

Using the above description, in Section~\ref{sec:ias} we find
integral-affine spheres with two commuting
involutions~$\iota_{\enr, \,{\rm IA}}$ and $\iota_{\dP,{\rm IA}}$ and
the corresponding Kulikov models of K3 surfaces with involutions.
We then describe how to construct the dlt models for degenerations
of Enriques surfaces, and give some detailed examples.

In Section~\ref{sec:ksba}, as an application of general theory and in a
similar way to \cite{alexeev2022compactifications-moduli,
  alexeev2022mirror-symmetric} we describe the KSBA compactification
$\ofen$ and the stable pairs appearing on the boundary. For K3
surfaces, the irreducible components of degenerations are ADE
surfaces of \cite{alexeev17ade-surfaces}. In the case of Enriques
surfaces, additional B and C
surfaces appear, corresponding to $B$ and $C$ Dynkin diagrams
resulting from folding ADE diagrams.
We describe them in Sections~\ref{sec:abcde} and~\ref{sec:stable-models}.

\smallskip
We work over the complex numbers, although most of the results can be
generalized to any field of characteristic $\ne2$. 

\begin{acknowledgements}
  The first author was partially supported by the NSF under
  DMS-2201222. The second author was partially supported by the NSF
  under DMS-2201221.  The fourth author is a member of the INdAM group
  GNSAGA and was partially supported by the projects ``Programma per
  Giovani Ricercatori Rita Levi Montalcini'', PRIN2020KKWT53 and PRIN
  2022 -- CUP E53D23005790006.

  We thank Igor Dolgachev for helpful comments.
\end{acknowledgements}

\section{General setup}
\label{sec:setup}

\subsection{The main diagram,  general case}
\label{sec:main-diagram}

Consider a surface $Y=\bP^1\times\bP^1$ with an involution
$\barien\colon (x,y)\to (-x,-y)$ and quotient $W=Y/\tau$.  Let
$B\in |-2K_Y|$ be a $\barien$-invariant divisor with at worst $ADE$
singularities not passing through the four points with
$x,y\in\{0,\infty\}$ fixed by $\barien$. The double cover
$\pi\colon X\to Y$ branched in $B$ is a K3 surface with a \emph{del
  Pezzo involution} $\idp$ so that $Y=X/\idp$. The involution
$\barien$ on $Y$ lifts to a basepoint-free \emph{Enriques involution}
$\ien$ on $X$ commuting with $\idp$ and the quotient $Z=X/\ien$ is an
Enriques surface.  The second lift of $\tau$ is a Nikulin involution
and the quotient $Z'=X/\inik$ is a K3 surface with eight $A_1$
singularities and possibly more.  This gives the following commutative
diagram:
\begin{equation}\label{eq:basic}
\begin{tikzcd}
  X \arrow{d}[swap]{\pi} \arrow[swap]{r}{\psi}
  \arrow[bend left=20]{rr}{\psi'}
  & Z \arrow{d}[swap]{\rho}
  & Z' \arrow{dl}{\rho'} 
  \\
  Y\arrow[swap]{r}{\varphi} 
  & W 
\end{tikzcd}    
\end{equation}

The surface $Y=\bP^1\times\bP^1$ is toric and the line bundle
$\cO(4,4)$ has, as its polytope, a square $Q$ with sides of lattice
length $4$, shown in the left panel of Figure~\ref{fig-square-triangle}. 
The surface $W$ is toric as well, for the same polytope but for the even
sublattice $\bZ^2_\even = \{(a,b) \mid a+b \in 2\bZ \}$ shown by gray
dots.
It is a quartic del Pezzo surface with four $A_1$ singularities.
Vectors in $\bZ^2_\even$ are in bijection with monomials
$x^ay^b$ invariant under the involution $\barien\colon (x,y)\to (-x,-y)$.
Here, we freely identify monomials $1,x,\dotsc x^4$ with
$x_0^4,x_0^3x_1, \dotsc, x_1^4$, and $1,y,\dotsc,y^4$ with
$y_0^4,y_0^3y_1,\dotsc, y_1^4$.

\begin{figure}[htbp]
  \includegraphics[width=200pt]{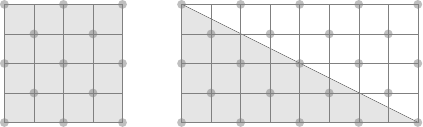}
  \smallskip
  \caption{Polytopes $Q$ and lattices for the toric surfaces $Y$ and $W$}
  \label{fig-square-triangle}
\end{figure}

Let $f(x,y)$ be a $\tau$-invariant polynomial of bidegree $(4,4)$ in which 
$1$, $x^4$, $y^4$, $x^4y^4$ have nonzero coefficients, so
that the hypersurface $\{f(x,y)=0\}\subset Y=V_Q$ does not contain
any of the
four torus-fixed points.  Let $P\subset\bR^3$ be the pyramid over
$Q\subset 0\times\bR^2$ with the vertex $(2,2,2)$, to which we
associate the monomial $z^2$. Then the K3 surface $X$ is a
hypersurface defined by the equation $z^2 + f(x,y)$ in the projective
toric variety $V_P$ associated with $P$. 
The polynomials $f(x,y)$
 invariant under $\barien$ are linear combinations of
$13$ monomials marked by gray dots. Thus, $f$ defines a point in an
open subset $U\subset\bP^{12}$ and in its quotient
$U/D_4\ltimes (\bC^*)^2$, of dimension~$10$.
There are three commuting involutions on $X$:
\begin{eqnarray*}
  &\text{del Pezzo} &\idp\colon (x,y,z) \to (x,y,-z)\\
  &\text{Enriques}  &\ien\colon (x,y,z) \to (-x,-y,-z)\\
  &\text{Nikulin}   &\inik\colon (x,y,z) \to (-x,-y,z)
\end{eqnarray*}
which together with the identity form a Klein-four group.
Both $\ien$ and $\inik$ are lifts of $\barien$.
On an affine subset of $X$ a nonvanishing
$2$-form is given by
\begin{displaymath}
\omega=\operatorname{Res}_X\frac{dx\wedge dy\wedge dz}{z^2 + f}.  
\end{displaymath}
One has $\idp^*\omega = \ien^*\omega=-\omega$ and
$\inik^*\omega=\omega$. So $\idp$ and $\ien$ are nonsymplectic and $\inik$
is symplectic.  The Enriques surface $Z$ is then a hypersurface in the
toric variety for the polytope $P$ but for the even sublattice $\bZ^3_\even
= \{(a,b,c) \mid a+b+c \in 2\bZ \}$. It is defined by the same
polynomial $z^2+f(x,y)$ whose monomials lie in $\bZ^3_\even$. 

\smallskip

Let $R$ be the ramification divisor of $\pi$.
The involution $\idp$ on $X$ descends to an involution $\baridp$ on
$Z$, and $W=Z/\baridp$.
Let $R_Z$ and $B_W$  be the ramification and branch divisors of~$\rho$.
Then $R=\psi^*(R_Z)$ and $R_Z = \frac12\psi_*(R)$. 
Since $R=\frac12 \pi^* (B)$ is an ample divisor, $R_Z$
is ample as well. One has $\cO(R_Z)=\cL_Z^{\otimes 2}\in\Pic Z$.

Horikawa \cite{horikawa1978periods-enriques2} analyzed in some detail
the sets of possible equations $f(x,y)$ and the maps from various opens
subsets of $\bP^{12}$ to the period domain $\bD/\Gamma$ and its
Baily-Borel compactification,  introduced in the next section.  In
particular, he showed that certain mildly singular $f(x,y)$ vanishing at
a torus-fixed point correspond to Coble surfaces, which are
$S_2$-quotients of nodal K3 surfaces.

The GIT compactification $\bP^{12}/\!/D_4\ltimes(\bC^*)^2$ was
described by Shah \cite{shah1981degenerations-of-k3}, who gave normal
forms for polystable orbits. As usual for the moduli of K3 surfaces
with a projective construction, the relation between the GIT and the
Baily-Borel compactifications is not straightforward,
cf. \cite{looijenga1986new-compactifications} for K3 surfaces of
degree~$2$ and \cite{laza2021git-versus} for degree $4$ K3 surfaces
which are double covers of $\bP^1\times\bP^1$.

\subsection{The main diagram, special case}
\label{sec:special-case}

The previous section describes the general case, when the K3
cover $X$ is non-unigonal. The special case corresponds to a Heegner
divisor in $\fen$ for which $(Y, L) = (\bP(1,1,2), \cO(4))$ is a
singular quadric.
The toric surfaces $Y$ and $W$ correspond to the same
polytope~$Q$ shown in the right panel of
Figure~\ref{fig-square-triangle} but for different lattices: $\bZ^2$
and $\bZ^2_\even$. The surface $W$ is a quartic del Pezzo surface with
$A_3+2A_1$ singularities.

There are still $13$ even monomials giving a family over
$U\subset\bP^{12}$. However, in this case $\Aut(W)$, the
centralizer of $\tau$ in $\Aut(Y)$, is three-dimensional, equal to
$S_2\ltimes (\bC^*{}^2\ltimes\bC)$. So there is only a
$9$-dimensional family of non-isomorphic surfaces.

\begin{remark}
  To our knowledge, Diagram~\eqref{eq:basic} minus $Z'$ first appeared
  in \cite[6.3.1]{cossec1983projective-models}, see also
  \cite[Cor. 4.7.2]{cossec1989enriques-surfaces}.  Horikawa model
  \cite{horikawa1978periods-enriques1} is a birationally isomorphic
  version of this diagram. 
  Ultimately, it can
  be traced back to the Enriques' double plane model
  \cite{enriques1906sopra-le-superficie}.  We refer the reader to
  \cite{cossec2024enriques-surfaces1} for a detailed historical
  account and many other projective models of Enriques surfaces.
\end{remark}

\begin{remark}\label{rem:two-dual-covers}
  The entire Diagram~\eqref{eq:basic} is intrinsic to the pair
  $(Z, [\cL_Z])$, in both the general and special
  cases. Indeed, $W= \varphi_{|\cL_Z^{\otimes 2}|}(Z)$,
  $Y = \Spec \cO_W \otimes \cO_W(A)$ where $A$ is the generator of
  $\Tors\operatorname{Cl}(W)=\bZ_2$, and $X = Y\times_W Z$.
  One has
  \begin{displaymath}
    Z' = \Spec \cO_W \oplus \cO_W(K_W)
    \quad\text{and}\quad
    Z = \Spec \cO_W \oplus \cO_W(K_W+A)
  \end{displaymath}
  with the multiplications defined by the divisor
  $B_W \in |-2K_W| = |-2(K_W+A)|$.
\end{remark}

\subsection{Period domains} 
\label{sec:periods}

We follow \cite{sterk1991compactifications-enriques1} for the moduli
space of Enriques surfaces with a numerical degree~$2$ polarization,
and \cite{alexeev2022mirror-symmetric} for the moduli space of K3
surfaces of degree~$4$ with a nonsymplectic involution.

Let $L=I\!I_{3,19} = U^3\oplus E_8^2 \simeq H^2(X,\bZ)$ be the K3
lattice. It is even, unimodular and has signature $(3,19)$. Here,
$U=I\!I_{1,1}$ and our $E_8=I\!I_{0,8}$ is a negative-definite unimodular
lattice. Let us write $L$ in block form:
\begin{displaymath} L = U \oplus U \oplus U \oplus E_8 \oplus E_8 = \{
(v,u,u',e,e') \mid v,u,u'\in U,\ e,e'\in E_8 \} 
\end{displaymath}  

\begin{definition}\label{def:involutions}
  Define three involutions $I_\dP$, $I_\enr$ and
  $I_\nik=I_\dP \circ I_\enr$ on $L$ corresponding to the Enriques,
  del Pezzo, and Nikulin involutions on K3 surfaces of degree~$4$ as
  follows (cf. \cite{peters2020on-k3-double} for the nodal case):
\begin{eqnarray*}
  &I_\dP\colon &(v,u,u',e,e') \to (-v,u',u, -e,-e') \\
  &I_\enr\colon &(v,u,u',e,e') \to (-v,u',u,e',e) \\
  &I_\nik\colon &(v,u,u',e,e') \to (v,u,u',-e', -e) 
\end{eqnarray*}
\end{definition}
Their $(\pm1)$-eigenspaces are
\begin{align*}
  &S_\dP:=L_\dP^+ = \Delta(U)
  \quad
  &&T_\dP:=L_\dP^- = U\oplus\Delta^-(U)\oplus E_8^2
  \\
  &S_\enr:=L_\enr^+ = \Delta(U)\oplus \Delta(E_8)
  \quad
  &&T_\enr:=L_\enr^- = U\oplus\Delta^-(U)\oplus \Delta^-(E_8)
  \\
  &L_\nik^+ = U^3 \oplus \Delta^-(E_8)
  &&L_\nik^- =  \Delta(E_8)
 \end{align*}
Here, $\Delta$ and $\Delta^-$ denote the diagonals and anti-diagonals
in $U^2$ and $E_8^2$. As lattices, they are isomorphic to 
\begin{align*}
  &S_\dP = U(2) = (2,2,0)_1
  \quad
  &&\tdp = U\oplus U(2)\oplus E_8^2 = (20,2,0)_2
  \\
  &S_\enr = U(2)\oplus E_8(2) = (10,10,0)_1
  \quad
  &&\ten = U\oplus U(2)\oplus E_8(2) = (12,10,0)_2
  \\
  &L_\nik^+ = U^3\oplus E_8(2) = (14,8,0)_3
  &&L_\nik^- = E_8(2) = (8,8,0)_0
\end{align*}
All these lattices are even and $2$-elementary, i.e.~with the
discriminant group $\Lambda^*/\Lambda \simeq \bZ_2^a$ for some
$a$. Recall, see e.g.  \cite{nikulin1979integer-symmetric}, that an
indefinite even $2$-elementary lattice is uniquely determined by its signature and
a triple $(r,a,\delta)$, where $r=\rk_\bZ(\Lambda)$,
$a=\rk_{\bZ_2}(\Lambda^*/\Lambda)$ and $\delta\in \{0,1\}$ is the coparity:
$\delta=0$ if the discriminant form
$q_\Lambda\colon \Lambda^*/\Lambda\to \frac12\bZ/2\bZ$,
$q_\Lambda(x) = x^2{\ \rm mod\ } 2\bZ$ 
is $\bZ$-valued
and $\delta=1$ otherwise. In our notation, $(r,a,\delta)_{n_+}$
denotes such a lattice of signature $(n_+,n_-)$.

\begin{lemma}\label{lem:unique-lattices}
  The sequence $U\oplus U(2)\oplus E_8(2) \to
  U\oplus U(2) \oplus E_8^2\to L$ of primitive embeddings is unique up to an isometry in $O(L)$.
\end{lemma}
\begin{proof}
  By taking the orthogonals, this is equivalent to the condition that
  the sequence of primitive embeddings
  $U(2)\to U(2)\oplus E_8(2) \to L$ is unique. The second embedding is
  unique by \cite[1.14.4]{nikulin1979integer-symmetric}.  The first
  embedding is equivalent to the embedding
  $U\to \Lambda = U\oplus E_8$ and it is well known that it is
  unique. Indeed, $\Lambda=U\oplus U^\perp$ and $U^\perp\simeq E_8$.
  The homomorphisms $O(L)\to O(\tdp)$ and $O(L)\to O(\ten)$ are
  surjective by \cite[1.5.2, 3.6.3]{nikulin1979integer-symmetric}.
\end{proof}

\begin{definition}
We have type IV period domains $\bD(\ten)$ and $\bD(\tdp)$, where for a
lattice $\Lambda$ of signature $(2,n)$ the corresponding period domain  $\bD(\Lambda)$
is a connected component of
\begin{displaymath}
  \{ [x]\in \bP(\Lambda\otimes\bC) \mid
  x\cdot x=0, \ x\cdot \bar x > 0 \}.
\end{displaymath}
\end{definition}

Since, $\ten\subset \tdp$, one has $\bD(\ten) \subset \bD(\tdp)$. The
polarizations we consider in both cases are defined by the vector $h=e+f \in
U(2)=S_\dP \subset S_\enr$. Here, $\{e,f\}$ is the basis of $U(2)$ with
$e^2=f^2=0$, $e\cdot f=2$.

\begin{definition}
  Define the arithmetic group $\gen$ as the image in $O(\ten)$ of
  \begin{displaymath}
    \{g\in O(L) \mid g\circ I_\enr = I_\enr\circ g
    \text{ and } g(h)=h \}.
  \end{displaymath}
  Additionally, define $\Gamma_\enr= O(T_\enr)$ and $\Gamma_\dP = O(T_\dP)$.  
 We have $\gen = \Gamma_\enr \cap \Gamma_\dP.$
\end{definition}

Since $O(L)\to O(\ten)$ and $O(L)\to O(\tdp)$ are surjective by
\cite[1.5.2, 3.6.3]{nikulin1979integer-symmetric}, the homomorphisms
from the centralizer groups
$Z_{O(L)}(I_{\enr})\to \Gamma_\enr$ and
$Z_{O(L)}(I_{\dP})\to \Gamma_\dP$ are surjective.

\begin{definition}
  Define three quotients of period domains:
  \begin{displaymath}
    \fen = \bD(\ten)/\gen, \quad
    \fken = \bD(\ten)/\Gamma_\enr, \quad
    \fhyp = \bD(\tdp)/\Gamma_\dP.
  \end{displaymath}
\end{definition}

By \cite{alexeev2021nonsymplectic,alexeev2022mirror-symmetric},
$\fhyp$ is 
the coarse moduli spaces of K3 surfaces 
with ADE singularities and 
a nonsymplectic involution for which the 
$(+1)$-eigenspace $(\Pic X)^+$ is $(2,2,0)_1$.

By \cite[2.13]{namikawa1985periods-of-enriques} there is a unique
$(-2)$-vector $\alpha\in T_{\enr}$ modulo $\Gamma_\enr$. The
discriminant divisor $\Delta=\alpha^\perp \subset \bD(\ten)/\Gamma_\enr$
parameterizes quotients of nodal K3 surfaces by an involution fixing a
node. These are rational Coble surfaces with a
$\frac{(1,1)}4$-singularity.  It is well known that the points of
$(\bD(\ten)\setminus\Delta)/ \Gamma_\enr\subset\fken$ are in a
bijection with the isomorphism classes of Enriques surfaces.
%



By \cite[2.15]{namikawa1985periods-of-enriques} there are two
$\Gamma_\enr$-orbits of $(-4)$-vectors $\beta$ in $\ten$. The divisor
corresponding to the vector with
$\beta^\perp \simeq \la 4\ra \oplus U\oplus E_8(2)$
parameterizes nodal Enriques surfaces, whose
desingularizations contain a $(-2)$-curve.
The other $(-4)$-vector corresponds to the unigonal Enriques surfaces
which are double covers of $\bP(1,1,2)$.

By \cite{sterk1991compactifications-enriques1} the
complement of the discriminant divisor in $\fen = \bD(\ten)/\gen$ is
the coarse moduli space of Enriques surfaces with a numerical polarization of
degree~$2$. 

\begin{lemma}\label{lem:fen-closed-in-fhyp}
  $\fen$ is the normalization of a closed subvariety of $\fhyp$. 
\end{lemma}
\begin{proof}
  One has $\bD(\ten) \subset \bD(\tdp)$.
  The isometry group $O(\tdp)$ coincides with the image of the group
  \begin{displaymath}
    \{g\in O(L) \mid g(\tdp)=\tdp, \ g(h)=h \}.
  \end{displaymath}
Indeed, any element of $O(T_\dP^*/T_\dP) \simeq O(S_\dP^*/S_\dP)$ can be extended
to an automorphism of $S_\dP$ that fixes $h$ because this group of
order $2$ preserves $\frac12 h\in S_\dP^*/S_\dP \simeq U(2)^*/U(2)\simeq
\bZ_2^2$.
Thus, the stabilizer of $\ten$ in $\Gamma_\dP$ is $\gen$ and so the
stabilizer of $\bD(\ten)$ in $\bD(\tdp)$ is~$\gen$. Thus the finite map
$\bD(\ten)/\gen \to \bD(\tdp)/\Gamma_\dP$ is generically injective. 
\end{proof}

Since $\gen\subset \Gamma_\enr$ is a finite index subgroup, there is also an obvious morphism
$\fen\to\fken$. It has degree $2^7\cdot 17\cdot 31$, see
\cite[Rem.~2.12]{sterk1991compactifications-enriques1}. 

\begin{definition} For a type IV arithmetic quotient $F=\bD(\Lambda)/\Gamma$, denote by $\oF\ubb$ 
its Baily-Borel compactification \cite{baily1966compactification-of-arithmetic}. \end{definition}

The boundary components of $\oF\ubb$ are points and modular curves, corresponding respectively
to primitive isotropic lines and planes in $\Lambda$. We call these boundary components $0$-cusps
and $1$-cusps respectively.




\subsection{KSBA stable pairs and their moduli spaces}
\label{sec:ksba-recall}

The idea behind KSBA spaces is very simple: they are a close
generalization of Deligne-Mumford-Knudsen's moduli spaces $\oM_{g,n}$
of pointed stable curves. For a one-parameter degeneration of of K3 surfaces
with a distinguished ample divisor, there are infinitely many Kulikov models that 
differ by flops, but there is a canonical KSBA-stable limit. 

In brief, a KSBA stable pair $(X,B=\sum_{i=1}^n b_iB_i)$ consists of a
projective variety $X$ which is deminormal: seminormal with only
double crossings in codimension~$1$, $B_i$ are effective Weil divisors
not containing any components of the double locus of~$X$, $0<b_i\le1$
are rational numbers, all satisfying two conditions:
\begin{enumerate}
\item (on singularities) the pair $(X,B)$ has semi log canonical (slc)
  singularities, the generalization of the log canonical singularities
  appearing in the MMP to the nonnormal case, and
\item (numerical) the divisor $K_X+B$ is an ample $\bQ$-Cartier divisor.
\end{enumerate}
The main result is that in characteristic zero for the fixed dimension
$d=\dim X$, number $n$, coefficient vector $(b_1,\dotsc, b_n)$ and degree
$(K_X+B)^d$ there is a (carefully defined) moduli functor for families
of KSBA stable pairs, the moduli stack is Deligne-Mumford, and the
coarse moduli space is projective.  We refer the reader to
\cite{kollar2023families-of-varieties} for complete details.

We  need  a version  of  this  definition  when $K_X$  is  numerically
trivial,  $R$   is  an   ample  Cartier  divisor   and  the   pair  is
$(X,\epsilon  R)$  with  $0<\epsilon\ll1$ allowed  to  be  arbitrarily
small. By \cite{kollar2019moduli-of-polarized,birkar2023geometry-polarized} in
any dimension  $d$ for fixed degree $R^d$ there exists $\epsilon_0>0$ such
that the moduli space for any $0<\epsilon<\epsilon_0$ is the
same. We only need this result for K3 surfaces, in which case the
construction and the proof were given in \cite{alexeev2023stable-pair} and
\cite{alexeev2022compactifications-moduli}. The Enriques case then
immediately follows.

In \cite{alexeev2023stable-pair, alexeev2022mirror-symmetric} this
general construction was applied to describe a geometric
compactification for the pairs $(X,\epsilon R)$ where $X$ is a K3
surface with
a non-symplectic
involution and 
ADE singularities, and $R$ is a connected component of
genus $g\ge2$ of the ramification divisor for the induced double
cover. 

In this paper we apply it to the pairs $(Z,\epsilon R_Z)$, where $Z$
is an Enriques surface with ADE singularities and with a numerical
degree~$2$ polarization, an $S_2$-quotient of a K3 surface
$X\in F_{(2,2,0)}$ with ADE singularities, and $R_Z$ is the
ramification divisor of the induced involution $\baridp$ as in the
introduction. For the KSBA-stable limits, $R_Z$ will be the divisorial
part of the ramification divisor of $Z\to W=Z/\baridp$ that is not contained
in the double locus of $Z$. 

\begin{definition}\label{def:ksba-comp} The compactification $\fen\hookrightarrow \ofen$ 
is the closure of the space of pairs $\{(Z,\epsilon R_Z)\,\big{|}\,Z\in \fen\}$ 
in the moduli space of KSBA stable pairs.
 \end{definition} 

 Our main goal is to describe the normalization of $\ofen$ and the
 surfaces appearing on the boundary.

\section{Cusps and Coxeter diagrams}
\label{sec:cusps}

\subsection{Cusp diagram of $\fhyp$}
\label{sec:cusps-Fhyp}

Figure~\ref{fig-k3-allcusps} reproduces the cusp diagram of $\ofhyp\ubb$ 
given in the last section
of \cite{alexeev2022mirror-symmetric}.
An equivalent diagram is found in \cite{laza2021git-versus}. 

\begin{figure}[htbp]
  \includegraphics[width=330pt]{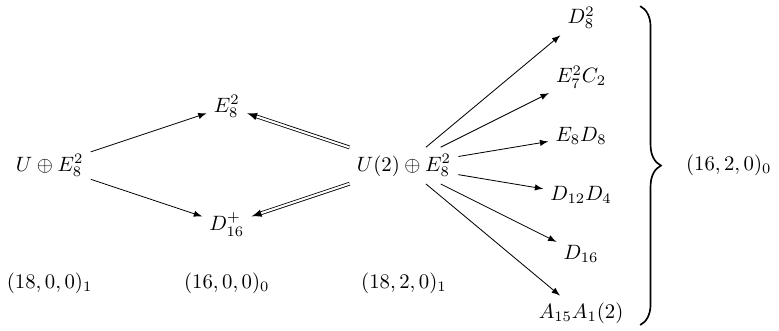}
  \smallskip
  \caption{Cusp diagram of $\fhyp$, for $\tdp=U\oplus U(2)\oplus E_8^2$}
  \label{fig-k3-allcusps}
\end{figure}

There are two $0$-cusps which are in bijection with the primitive isotropic
lines $\bZ e\in \tdp \mod {\Gamma_\dP}$, distinguished by the divisibility
$\di(e)\in \{1,2\}$ of $e$ in the dual lattice $T_\dP^*$. 
For a primitive vector
in a $2$-elementary lattice one must have $\di(e)\in \{1,2\}$.

The lattices $e^\perp/e$ are hyperbolic and $2$-elementary, and here are of the form $U\oplus E_8^2 = (18,0,0)_1$ and $U(2)\oplus E_8^2 = (18,2,0)_1$ depending on whether $\di(e)=2$ or~$1$ respectively.
Similarly, the eight $1$-cusps are in
bijection with the primitive isotropic planes $\Pi\subset \tdp
\mod{\Gamma_\dP}$. For each of them there is a negative-definite
lattice $\Pi^\perp/\Pi$ which is $2$-elementary but is no longer
uniquely determined by the triple $(r,a,\delta)$. The label
denotes the root sublattice of $\Pi^\perp/\Pi$.
Here, $D_{16}$ is a root lattice with determinant $4$ and $D_{16}^+$
is its unique even unimodular
extension.

The bipartite diagram in Figure \ref{fig-k3-allcusps} depicts all $0$- and $1$-cusps
added to compactify $F_{(2,2,0)}$. An arrow indicates that a $0$-cusp
lies in the closure of a $1$-cusp. Equivalently, there is, up to the group action,
an inclusion $\bZ e\subset \Pi$ of the corresponding isotropic subspaces.

\subsection{Cusp diagram of $\fken$}
\label{sec:cusps-Fken}

The cusp diagram for $\ofken\ubb$ is
well known. It can also be easily found by 
\cite[Sec. 5]{alexeev2022mirror-symmetric}.  We give it in
Figure~\ref{fig-en-allcusps}, keeping the same notation as
above. There are two $0$-cusps distinguished by the divisibility $\di(e)=1$
or $2$.

\begin{figure}[htbp]
  \includegraphics[width=280pt]{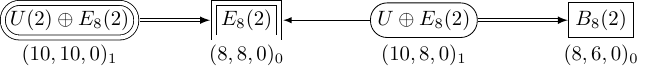} 
  \smallskip
  \caption{Cusp diagram of $\fken$, for $\ten=U\oplus U(2)\oplus E_8(2)$}
  \label{fig-en-allcusps}
\end{figure}

A geometric interpretation of these cusps is as follows. Let
$\cX\to (C,0)$ be a Kulikov model and consider the completed
period mapping $$(C,0)\to\ofken\ubb.$$
Suppose that $\iota_{\enr, t}$ on the generic fiber extends to an involution
$\iota_{\enr, 0}$ on the central fiber. If $0\in C$ is sent to the double-circled cusp
$(10,10,0)_1$ or $(8,8,0)_0$ then $\iota_{\enr, 0}$ is basepoint
free. Otherwise, $\iota_{\enr, 0}$ has a nonempty finite set of fixed
points.

Furthermore, the dual complex $\Gamma(\cX_0)$ is a $2$-sphere
and the induced action of $\iota_{\enr, 0}$ on $\Gamma(\cX_0)$
in the $(10,10,0)_1$ case is an antipodal involution, while in the 
$(10,8,0)_1$ case it is a hemispherical involution, see 
\cite[Sec.~8F]{alexeev2022mirror-symmetric}. So the quotients
of $\Gamma(\cX_0)$ by the Enriques involution are, respectively,
the real projective plane $\bR\bP^2$ and a disk $\bD^2$.
In Type II, $\Gamma(\cX_0)$ is a segment.
In the $(8,8,0)_0$ case, the action of $\iota_{\enr, 0}$ flips
the segment, whereas in the $(8,6,0)_0$ case it fixes
the segment.


\subsection{Cusp diagram of $\fen$}
\label{sec:cusps-Fent}

Sterk \cite{sterk1991compactifications-enriques1} computed the cusp
diagram for $\ofen\ubb$. There are
five $0$-cusps for which we use Sterk's numbering $1,2,3,4,5$.  There
are also $9$ distinct $1$-cusps.

\begin{notation}\label{1-cusp-notation}
We denote a $1$-cusp by $i_1\dots i_k$ if its closure contains the $0$-cusps
$i_1,\dotsc,i_k$. Here, $1\leq k\leq 5$.
\end{notation}

\begin{lemma}\label{lem:bb-extensions}
  The morphisms $\fen \to \fhyp$ and $\fen\to \fken$ extend to the
  Baily-Borel compactifications, mapping $0$-cusps to $0$-cusps and
  $1$-cusps to $1$-cusps in the manner shown in Figure~\ref{fig-sterk-allcusps}.
\end{lemma}

\begin{figure}[htbp]
  \includegraphics[width=300pt]{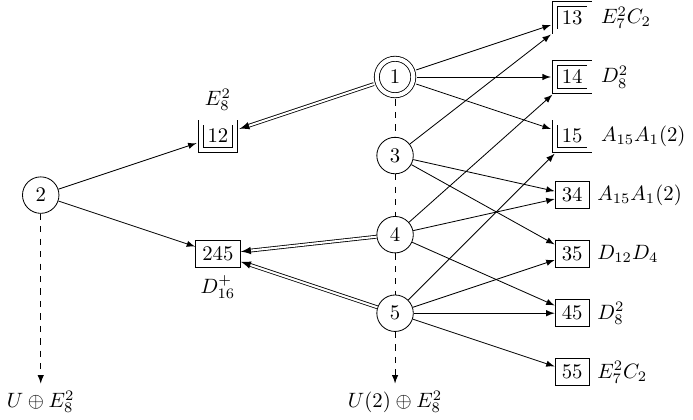} 
  \smallskip
  \caption{Cusps of $\ofen\ubb$ with images in $\ofken\ubb$ and
    $\ofhyp\ubb$}
  \label{fig-sterk-allcusps}
\end{figure}

The images in $\ofhyp\ubb$ are shown by labels
from Figure \ref{fig-sterk-allcusps}, and in
$\ofken\ubb$ by the corresponding border shapes (oval, double oval, rectangle,
double rectangle).

\begin{proof}
  The extension property holds by
  \cite{kiernan1972satake-compactification}. The maps on $0$-cusps
  are easy to see by looking at the divisibilities of Sterk's isotropic
  vectors $e_1,\dotsc, e_5$ considered separately as vectors in the
  lattices $\ten$ and $\tdp$.  The maps on $1$-cusps are then recovered
  by considering incidences between $0$- and $1$-cusps and the images
  of the $1$-cusps in the Baily-Borel compactification $\oF_4\ubb \supset \ofhyp\ubb$ of
  the moduli space of K3 surfaces of degree~$4$, computed at the end of
  \cite{sterk1991compactifications-enriques1}.
\end{proof}

\subsection{Vinberg's theory and Coxeter diagrams}
\label{sec:vinberg}

We refer to \cite{vinberg1973some-arithmetic,
  vinberg1972-groups-of-units} for Vinberg's theory of reflection
groups of hyperbolic lattices, saying just enough to fix the
notations. 

Let $\Lambda$ be a hyperbolic lattice. Let $\cC$ be the component of the set
$\{ v\in \Lambda_\bR \mid v^2>0\}$, containing a fixed class $h$ with
$h^2>0$. Let $\cH=\bP\cC$ be the corresponding hyperbolic space.  A
vector $v\ne0$ with $v^2=0$ in the closure of $\cC$ defines a point on the
sphere at infinity of $\cH$. Let $\ocC$ denote the closure of $\cC$.

A \emph{reflection} in a vector
$\alpha\in \Lambda$ is the isometry
$$w_\alpha(v) = v - \frac{2(\alpha\cdot v)}{\alpha\cdot \alpha} \alpha.$$  A
\emph{root} is a vector $\alpha$ with $\alpha^2<0$ such that
$w_\alpha(\Lambda)=\Lambda$, equivalently such that
$2\di(\alpha)\in(\alpha\cdot \alpha) \bZ$.
For a group of isometries $\Gamma\subset O(\Lambda)$ 
we denote by $W(\Gamma)$ its subgroup generated by reflections. 

We denote by $\ch$ the fundamental chamber for
$W(\Gamma)$. Equivalently, one can treat it as the (possibly infinite)
polyhedron $P=\bP\ch \subset \cH$. One has
\begin{eqnarray}\label{eq:chamber}
  &\ch = \left\{ v\in\cC \mid
    \alpha_i\cdot  v \ge 0 
    \text{ for simple roots } \alpha_i \right\}\\
  &O(\Lambda) = S\ltimes W(\Gamma) 
    \quad\text{for some } S\subset\Sym(P).
\end{eqnarray}
The fundamental chamber is encoded in a Coxeter diagram.
The vertices correspond to the simple roots $\alpha_i$ and the
edges show the angles between them as follows. Let
$g_{ij} = (\alpha_i\cdot \alpha_j) / \sqrt{(\alpha_i\cdot \alpha_i)
  (\alpha_j\cdot \alpha_j)}$. One connects $i$ and $j$ by
\begin{itemize}
\item an $m$-tuple line if $g_{ij} = \cos\frac{\pi}{m+2}$. The
  hyperplanes $\alpha_i^\perp$, $\alpha_j^\perp$ intersect in $\cH$.
\item a thick line if $g_{ij}=1$. $\alpha_i^\perp$,
  $\alpha_j^\perp$ are parallel, meet at an infinite point of $\cH$.
\item a dotted line if $g_{ij} > 1$. $\alpha_i^\perp$,
  $\alpha_j^\perp$ do not meet in $\cH$ or its closure.
\end{itemize}

The lattices $L_\enr^\pm$ and $L_\dP^\pm$ are even $2$-elementary. For such
lattices the roots are the $(-2)$-vectors and the $(-4)$-vectors with
$\di(\alpha)=2$.  We denote the roots with $\alpha^2=-2$ by white
vertices and those with $\alpha^2=-4$ by black vertices.

\subsection{Coxeter diagrams for the $0$-cusps of $\fhyp$}
\label{sec:coxeter-Fhyp} 

The Coxeter diagrams for the lattices $(18,0,0)_1 = U\oplus E_8^2$ and
$(18,2,0)_1 = U(2)\oplus E_8^2$,
cf.~\cite{alexeev2022mirror-symmetric}, are shown in
Figure~\ref{fig-k3-cusps}. To describe Kulikov models and KSBA stable
models, it is important to keep track of the even and odd nodes on the
boundaries of these diagrams. We assign even numbers to the even nodes;
in Figure~~\ref{fig-k3-cusps} they are shown as double-circled
nodes. For typographical reasons, in the diagrams that follow we skip
these double circles. The corners are always even.
%

\begin{figure}
  \centering
  \includegraphics[width=280pt]{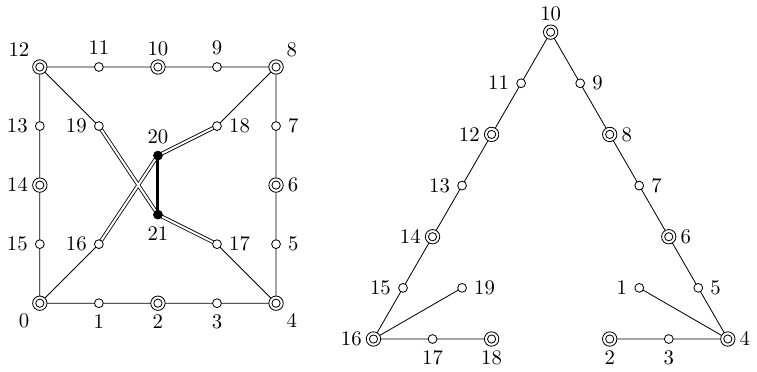} 
  \caption{Coxeter diagrams for $(18,2,0)_1$ and $(18,0,0)_1$}
  \label{fig-k3-cusps}
\end{figure}

The lattice $U\oplus E_8^2$ is generated by $19$ roots
$\alpha_1,\dotsc,\alpha_{19}$ with a single relation
\begin{align}\label{close-up}\begin{aligned}
  v &=3\alpha_1 + 2\alpha_2 + 4\alpha_3 + 6\alpha_4 + 5\alpha_5 +
        4\alpha_6 + 3\alpha_7 + 2\alpha_8 + \alpha_9  \\
  &=
  3\alpha_{19} + 2\alpha_{18} + 4\alpha_{17} + 6\alpha_{16} + 5\alpha_{15} +
  4\alpha_{14} + 3\alpha_{13} + 2\alpha_{12} + \alpha_{11} 
\end{aligned}\end{align}

The lattice $U(2)\oplus E_8^2$ is generated by $22$ roots
$\alpha_0,\dotsc,\alpha_{21}$. The relations come from maximal
parabolic subdiagrams with more than one connected component. Maximal
parabolic subdiagrams correspond to parabolic sublattices with a
unique isotropic line; the generator of this line is a linear
combination of roots in each connected component, which gives a linear
relation. For example, the
following relation results from $\wE_7^2\wC_2$:
\begin{equation}\label{eq:UE82}
  \alpha_{16}+\alpha_{20}+\alpha_{18} =
  \alpha_1 + 2\alpha_2 + 3\alpha_3 + 4\alpha_4
  +3\alpha_5 + 2\alpha_6 + \alpha_7 + 2\alpha_{17}.
\end{equation}

\subsection{Coxeter diagrams for the $0$-cusps of $\fken$}
\label{sec:coxeter-Fen}

The Coxeter diagrams for the lattices
$(10,10,0)_1 = U(2)\oplus E_8(2)$ and $(10,8,0)_1 = U\oplus E_8(2)$
are well-known.  They are shown in Figure~\ref{fig-enriques0}.

\begin{figure}
  \centering
  \includegraphics[width=330pt]{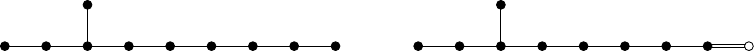}
  \caption{Coxeter diagrams for $(10,10,0)_1$ and $(10,8,0)_1$}
  \label{fig-enriques0}
\end{figure}

\subsection{Folding Coxeter diagrams by involutions}
\label{sec:folding}

\begin{definition}
  Let $\Lambda$ be a lattice with an involution and let
  $\alpha\in\Lambda$ be a vector. We call the following vector
  $\alpha_I\in\Lambda^{I=1}$ a \emph{folded} vector:
  \begin{displaymath}
    \alpha_I = 
    \begin{cases}
      \alpha &\text{if } I(\alpha)=\alpha\\
      \alpha + I(\alpha) &\text{if } I(\alpha)\ne\alpha\\
    \end{cases}
  \end{displaymath}
\end{definition}

\begin{lemma}\label{lem:folded-roots}
  Consider the lattice $\tdp$ with the involution
  $ I := -I_\enr=I_\dP\circ I_\enr = I_\nik$, so that
  $\tdp^{I=1} = \ten$.  Let $\alpha$ be a root of $\tdp$ and assume
  that $\alpha_I^2<0$.  Then $\alpha_I$ is a root in $\ten$ and one
  of the following holds:
  \begin{enumerate}
  \item $\alpha^2=-2$, $\alpha\in\ten$, so $\alpha=\alpha_I$ 
    is a root of both $\tdp$ and $\ten$.
  \item $\alpha^2=-4$, $\alpha\in\ten$, so $\alpha=\alpha_I$ 
    is a root of both $\tdp$ and $\ten$.
  \item $\alpha^2=-2$, $\alpha\cdot I(\alpha)=0$, $\alpha_I^2=-4$,
    and $\alpha_I$ is a root of $\ten$ but not of $\tdp$.
  \end{enumerate}
  Vice versa, all roots of $\ten$ are of these three types.
\end{lemma}
\begin{proof}
  If $\alpha\in\ten$, the claim is clear. Now suppose that
  $\alpha\not\in\ten$, so $\alpha_I = \alpha+I(\alpha)$.  
  Write $\alpha\in \tdp$ in the block form as in
  Definition~\ref{def:involutions}. Then
  \begin{eqnarray*}
    &\alpha = (v,u,-u,e,e') \quad\text{and}\quad
      I(\alpha) = (v,u,-u,-e',-e)\\
    &I(\alpha) = \alpha - (0,0,0,e+e',e+e'), \quad\\
    &\alpha\cdot I(\alpha) = \alpha^2 - (e+e')^2,\quad
      \alpha_I^2 = 2\alpha^2 + 2\alpha\cdot I(\alpha)
  \end{eqnarray*}
  Since, $\alpha\ne I(\alpha)$, $e+e'$ is a nonzero vector in $E_8$.
  Therefore, $\alpha\cdot I(\alpha) > \alpha^2$.
  
  For $\alpha^2=-2$ this leaves the only possibility $\alpha\cdot
  I(\alpha)=0$ and $\alpha_I^2=-4$. Clearly, $\di(\alpha_I)=2$ in $\ten$
  so $\alpha_I$ is a root of $\ten$. But $\di(\alpha_I)\ne 2$ in $\tdp$.
  Otherwise, $e-e'\in 2E_8$, which implies that
  $(e+e')^2\equiv 0\pmod 4$, $\alpha\cdot I(\alpha)\ge 2$ and
  $\alpha_I^2\ge0$. 

  Now let $\alpha$ be a $(-4)$-root in $\tdp$. Since the
  divisibility of $\alpha$ is $2$, one must have $e,e'\in 2E_8$, so
  also $e+e'\in 2E_8$. But then $-(e+e')^2\ge 8$, $\alpha\cdot
  I(\alpha)\ge 4$ and $\alpha_I^2\ge0$, a contradiction.
  This completes the forward direction.
  
  The converse follows from \cite[2.13 and
  2.15]{namikawa1985periods-of-enriques}: up to $\Gamma_\enr=O(\ten)$
  acting on $\ten$ there is only one type of $(-2)$-vector and two
  types of $(-4)$-vectors. 
\end{proof}

\begin{definition}
  \label{def:J}
Consider a primitive vector $e\in \ten$ with $e^2=0$. We get
two hyperbolic lattices
\begin{displaymath}
  \oT_\enr = e^\perp(\text{in }\ten)/e, \quad
  \oT_\dP = e^\perp(\text{in }\tdp)/e,  \quad \textrm{with }\oT_\enr\subset\oT_\dP.
\end{displaymath}
There are induced involutions
$\oI_\enr$ and $\oI_\dP$ on these hyperbolic lattices.  We denote
$J = \oI_\enr\circ\oI_\dP = -\oI_\enr$, which is an involution
on $\oT_\dP$, for which the $(+1)$-eigenspace of $J$ in $\oT_\dP$
is $\oT_\enr$.
\end{definition}

\begin{definition}
  Let $e\in T_\enr$ be primitive isotropic. The stabilizer
  $\Gamma_{\enr, 2, e}$ of $e$ in $\gen$ fits into an exact sequence
  $$0\to U_e\to \Stab_{\gen}(e)\to \Gamma_{\enr, 2, e}\to 0$$
  where $U_e$ is the unipotent subgroup, which acts trivially on
  $e^\perp/e = \oT_\enr$. We define $\Gamma_{\dP,e}$ and $\Gamma_{\enr, e}$
  similarly.

  Denote by $W(\Gamma_{\dP,e})$ the reflection subgroup of
  $\Gamma_{\dP,e}$. Its Coxeter diagram $G(\oT_{\dP})$ is one of the two
  Coxeter diagrams in Figure~\ref{fig-k3-cusps}. Denote by
  $W(\Gamma_{\enr, 2, e})$ the reflection subgroup of
  $\Gamma_{\enr, 2, e}$; it is generated by reflections in the roots
  $\alpha\in\ten$ with $\alpha\cdot e=0$.
\end{definition}

\begin{definition}
  Let $\Lambda$ be an elliptic, parabolic, or hyperbolic lattice
  with an involution $J$, and let $G$ be its Coxeter diagram. We
  define the folded diagram $G^J$ to be the diagram with the vectors
  $\alpha_J$ for the roots $\alpha$ in $G$ for which
  the folded vectors $\alpha_J$ happen to be roots of $\Lambda^{J=1}$.
\end{definition}

\begin{lemma}\label{lem:all-diagrams-are-folded}
  Let $\ch$
  be a chamber for the action of $W(\Gamma_{\dP,e})$ on the positive cone
  $\cC(\oT_\dP)$ whose intersection with $\oT_{\enr}$ has
  maximal dimension. Then the cone $\ch^J:=\ch\cap \oT_\enr\otimes\bR$ is a
  fundamental chamber for $W(\Gamma_{\enr, 2, e})$ and its Coxeter
  diagram is the folded diagram $G(\oT_\dP)^J$. 
\end{lemma}
\begin{proof}
  Let $\alpha$ be one of the simple roots in
  equation~\eqref{eq:chamber}, so that $\alpha^\perp$ is a wall of~$\ch$.
  The intersection of the positive cone $\cC(\oT_\enr)$ with
  $\alpha^\perp$ is the same as with $\alpha_J^\perp$. If it is
  nonempty then $\alpha_J^2<0$. But then $\alpha_J$ is a root in
  $\oT_\enr$ by Lemma~\ref{lem:folded-roots}. So the walls of $\ch^J$
  are $\alpha_J^\perp$ for the folded roots in $G(\oT_\dP)^J$ and $\ch^J$ is the
  fundamental chamber for the reflection group with Coxeter diagram
  $G(\oT_\dP)^J$. 
\end{proof}

\subsection{Coxeter diagrams for the $0$-cusps of $\fen$ by folding}
\label{sec:coxeter-Fent}

We now find five involutions of the lattices $U\oplus E_8^2$ and
$U(2)\oplus E_8^2$ and compute folded diagrams for them. We prove that
they are the Coxeter diagrams for the groups $\Gamma_{\enr, 2, e}$ for some
isotropic vectors $e\in\ten$. These turn out to be the same as the
Coxeter diagrams computed in
\cite{sterk1991compactifications-enriques1} by Vinberg's method
\cite{vinberg1973some-arithmetic}.  We 
keep Sterk's numbering for the $0$-cusps. In the order of appearance,
they are 2, 1, 3, 4, 5.

\begin{lemma}\label{lem:cusp2}
  On the Coxeter diagram for the lattice $\oT_\dP = (18,0,0)_1$,
  consider the reflection $ J\colon \alpha_k \to \alpha_{20-k}$
  about the vertical line. Then $\oT_\dP^{J=1}\simeq U\oplus E_8(2)$
  and the folded diagram is shown in Figure~\ref{fig-cusp2}.
\end{lemma}
\begin{figure}
  \centering
  \includegraphics{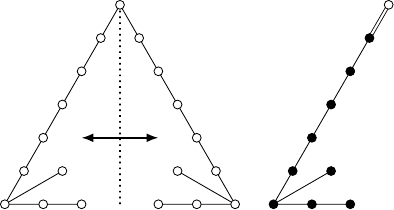}
  \caption{Folded diagram for cusp 2}
  \label{fig-cusp2}
\end{figure}

\begin{proof}
  The sublattice $\oT_\dP^{J=1}$ is generated by the vectors
  $\alpha_k + \alpha_{20-k}$, $1\le k\le 8$ spanning $E_8(2)$ and two
  vectors spanning an orthogonal $U$: $\alpha_{10}$ along with the vector $v$ in
  the relation~\eqref{close-up}. The computation of the folded Coxeter
  diagram is immediate.
\end{proof}

\begin{lemma}\label{lem:cusps1345}
  Consider the following involutions $J$ on the lattice $\oT_\dP =
  (18,2,0)_1$:
  \begin{enumerate}
  \item rotation of the diagram by $180$ degrees. 
    \setcounter{enumi}{2}
  \item reflection of the diagram about the diagonal, followed by a
    lattice reflection in the root $\alpha_{20}$.
  \item reflection of the diagram about a horizontal line. 
  \item the composition of $8$ commuting reflections in the roots
    $\alpha_1, \alpha_3, \dotsc, \alpha_{15}$. 
  \end{enumerate}
  The fixed sublattice $\oT_\dP^{J=1}$ is isomorphic to
  $U(2)\oplus E_8(2)$ in case (1) and to $U\oplus E_8(2)$ in cases
  (3,4,5). The folded diagrams are shown in Figure~\ref{fig-cusps1345}.
\end{lemma} 

\begin{figure}[h]
  \centering
  \includegraphics[width=360pt]{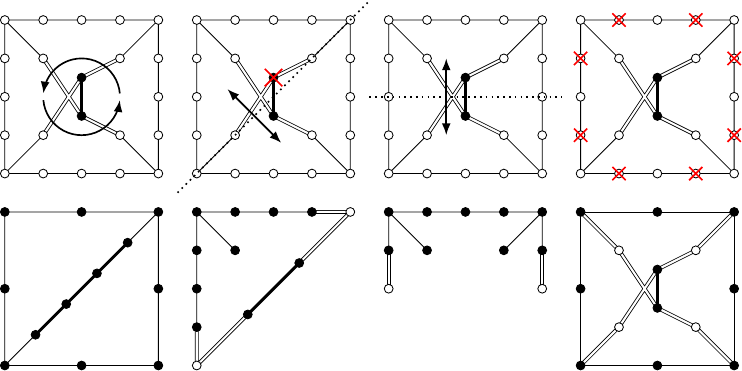}
  \caption{Folded diagrams for cusps 1, 3, 4, 5} 
  \label{fig-cusps1345}
\end{figure}

\begin{proof}
  The computation of the folded Coxeter diagrams is immediate. The
  fixed sublattices are computed as follows. In all cases the roots
  generate an index-$2$ sublattice of $\oT_\dP^{J=1}$. 

  The Coxeter diagram for cusp 1 contains a copy of $\wE_8(2)$,
  cf. diagram 12 in Figure~\ref{fig-pars}, and so contains a copy of $E_8(2)$.
  Half of the isotropic vector of $\wE_8(2)$ is integral, i.e.~lies in
  $\oT_\dP^{J=1}$. Together with the root disjoint from $E_8(2)$, these
  two elements form an orthogonal copy of $U(2)$,
  and together they span $\oT_\dP^{J=1}$.
  This gives $\oT_\dP^{J=1} = U(2)\oplus E_8(2)$.

  For cusp 3 we observe from diagram 31 in Figure~\ref{fig-pars} that the
  Coxeter diagram contains a copy of $(\wE_7\wA_1)(2)$,
  i.e. $\wE_7\wA_1$ with the doubled bilinear form. Inside it is a
  copy of $(E_7A_1)(2)$ which is an index-$2$ sublattice of
  $E_8(2)$. One checks that this $E_8(2)$ is indeed a sublattice of
  $\oT_\dP^{J=1}$. Half of the isotropic vector of $\wA_1(2)$
  together with the root disjoint from $(E_7A_1)(2)$ form an
  orthogonal copy of $U$.
  The computations for cusps 4 and 5 are
  similar, starting with the subdiagrams $\wD_8(2)$ and
  $(\wA_7\wA_1)(2)$, for cusps 41 and 51.
  We also made a check with sagemath.
\end{proof}

\begin{lemma}\label{lem:lift}
  The involution $J$ on lattice $\oT_\dP$
  of Lemmas~\ref{lem:cusp2},
  \ref{lem:cusps1345} can be lifted to an involution $ {I}$ on $\tdp$
  with the fixed sublattice $\ten$. Taking $I_\enr=- {I}$ gives $T_\dP^{I_\enr=-1}
  = T_\enr$. 
\end{lemma}
\begin{proof}
  For the involution of Lemma~\ref{lem:cusp2} the statement is obvious:
  we simply define $ {I}$ to be the identity on the first summand of
  $\tdp = U(2)\oplus\oT_\dP$. Similarly for cusp~(1) in
  Lemma~\ref{lem:cusps1345} one has $\tdp = U\oplus\oT_\dP$ and
  we extend $ {I}$ to $U$ as the identity. 

  In the cases (3,4,5) we have an exact sequence of abelian groups
  \begin{displaymath}
    0 \to U \to \tdp = U\oplus \oT_\dP\to \oT_\dP \to 0
  \end{displaymath}
  with $\oT_\dP=U(2)\oplus E_8^2$, $\oT_\dP^{J=1}=U\oplus E_8(2)$,
  and the trivial extension does not work.

  Write $U=\la e,f\ra$ using the standard basis with $e^2=f^2=0$,
  $e\cdot f=1$. A section
  $s\colon \oT_\dP \to \bZ e\oplus \oT_\dP \subset U\oplus \oT_\dP$ is
  the same as an element $a\in \oT_\dP^*$, so that $x\mapsto x+a(x)$. The
  orthogonal complement of $\oT_\dP$ is $\la e, f-a\ra\simeq U$ if
  $a\in\oT_\dP$, and $\la e, 2f-2a\ra$ if $a\notin\oT_\dP$. One has
  $(2f-2a)^2 = 4a^2$. From this, we see that the last lattice is
  isomorphic to $U(2)$ if the discriminant form of $\oT_\dP$ satisfies
  $q_{\oT_\dP}(a)\in\bZ$, and it is isomorphic to
  $I_{1,1}(2)=\la 2\ra\oplus\la -2\ra$ otherwise.
  
  The discriminant
  form of $\oT_\dP$ is the same as for
  $U(2) = \la e', f'\ra\subset\oT_\dP$. We choose $a=\frac12 e'$ and
  define the involution $I$ on $\tdp$ to be ${J}$ on
  $s(\oT_\dP)$ and the identity on its orthogonal complement $U(2)$. Then
  $$\tdp^{I=1} \simeq U(2)\oplus \oT_\dP^{J=1} \simeq \ten.$$ We
  complete the proof by Lemma~\ref{lem:unique-lattices}.
\end{proof}

\begin{corollary}\label{cor:folded-diagrams}
  The above five folded diagrams $G(\oT_\dP)^J$ are precisely the
  Coxeter diagrams for the reflection groups
  $W(\Gamma_{{\rm En},2,e})$ for the isotropic vectors
  $e\in\ten \pmod{\gen}$.
\end{corollary}
\begin{proof}
  By Lemmas~\ref{lem:all-diagrams-are-folded}, \ref{lem:cusp2}, \ref{lem:cusps1345},
  \ref{lem:lift} the five diagrams we have found are Coxeter diagrams
  for the reflection groups $W(\Gamma_{{\rm En},2,e})$ for some isotropic
  vectors $e\in\ten$.  By \cite{sterk1991compactifications-enriques1}
  the space $\fen$ has exactly five $0$-cusps, so we have found them
  all.
\end{proof}

Indeed, our Coxeter diagrams, obtained by folding, coincide with the
ones found by Sterk in \cite{sterk1991compactifications-enriques1} who
used the Vinberg algorithm \cite{vinberg1973some-arithmetic} to compute them.

\begin{proof}[Second proof, without using
  \cite{sterk1991compactifications-enriques1}]
  By Lemmas~\ref{lem:all-diagrams-are-folded} and \ref{lem:lift} it is
  sufficient to find all involutions ${J}$ on hyperbolic lattices
  $\oT_\dP = U\oplus E_8^2$ and $\oT_\dP = U(2)\oplus E_8^2$ for which
  the sublattice $\oT_\dP^{J=1}$ is isomorphic to
  $U(2)\oplus E_8(2)$ or $U\oplus E_8(2)$ and such that the folded
  root vectors define a chamber $\ch^J$ lying inside a chamber $\ch$
  for the Coxeter diagram $G(\oT_\dP)$.  Any such involution is a
  product of an involution of the diagram $G(\oT_\dP)$, which may be
  the identity, composed with a commuting involution in the Weyl
  group. It is well known that an involution in a Coxeter group is a
  composition of commuting reflections. 
  
  Under the condition
  $\rk\,\oT_\dP^{J=1}=10$, this reduces the check to the
  following possibilities, in addition to the ones in
  Lemma~\ref{lem:cusp2} and cases (1,4) of Lemma~\ref{lem:cusps1345}:

  \begin{enumerate}\renewcommand{\theenumi}{\alph{enumi}}
  \item a composition of reflections in $8$ orthogonal roots of $G(U\oplus E_8^2)$.
  \item the diagonal involution of $G(U(2)\oplus E_8^2)$ composed with
    a single reflection in $\alpha_0$, $\alpha_8$, $\alpha_{16}$,
    $\alpha_{18}$, $\alpha_{20}$ or $\alpha_{21}$. 
    \item a composition of reflections in $8$ orthogonal roots of
    $G(U(2)\oplus E_8^2)$.
  \end{enumerate}
  The first case does not occur. We confirmed with sagemath that
  $\oT_\dP^{J=1}$ is never isomorphic to $U(2)\oplus E_8(2)$,
  and that it is isomorphic to $U\oplus E_8(2)$ only in the second
  case for $\alpha_{20}$, and in the last case for
  $\{\alpha_1,\alpha_3,\dotsc,\alpha_{15}\}$.
\end{proof}

\subsection{$1$-cusps of $\fen$ by folding}
\label{sec:Fen-1cusps} 

\begin{lemma}
  The $1$-cusps of $\fen$ correspond to the maximal parabolic
  subdiagrams of the Coxeter diagrams of $U\oplus E_8^2$ and
  $U(2)\oplus E_8^2$ which are symmetric with respect to one of the
  five involutions of Lemma~\ref{lem:cusps1345}. For cusps 3 and 5
  this means that the subdiagram has to contain the roots
  $\alpha_{20}$, resp. $\alpha_1,\alpha_3,\dots,\alpha_{15}$ in which
  the reflections are made. 
\end{lemma}

Indeed, both correspond to the isotropic planes contained in the
sublattice of $T$ fixed by the involution.
We list these $1$-cusps in
Figures~\ref{fig-par2} and \ref{fig-pars}. They agree with Sterk's
computations in \cite{sterk1991compactifications-enriques1}, and the
entire cusp diagram agrees with Figure~\ref{fig-sterk-allcusps}. 

\begin{figure}[htpb]
  \centering
  \includegraphics[width=330pt]{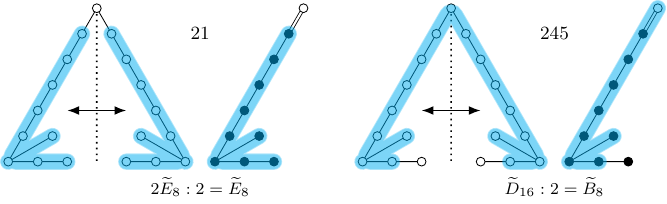}
  \caption{$1$-cusps of $\fen$ passing through $0$-cusp 2}
  \label{fig-par2}
\end{figure}

\begin{figure}[htpb]
  \centering
  \includegraphics[width=360pt]{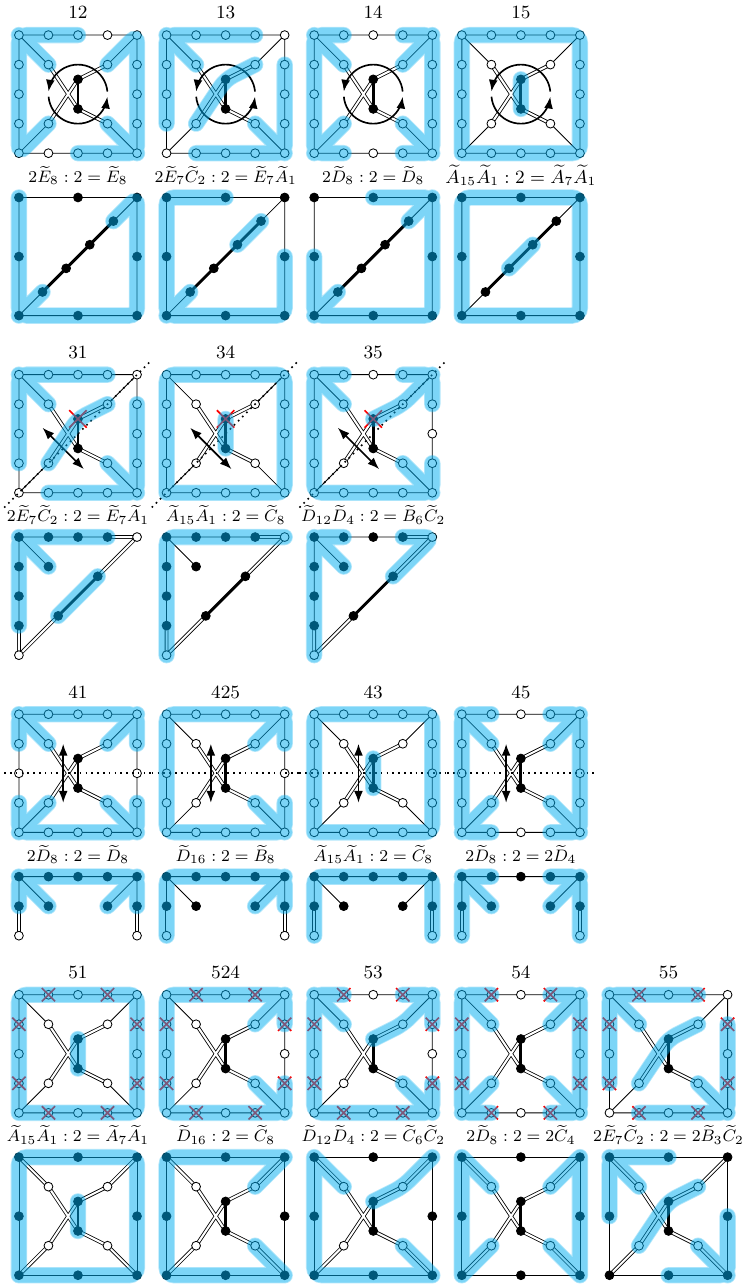} 
  \caption{$1$-cusps of $\fen$ passing through $0$-cusps 1, 3,
    4, 5}
  \label{fig-pars}
\end{figure}

Figures \ref{fig-par2}, \ref{fig-pars} contain the information of all
cusp incidences 
of $\ofen\ubb$. The figures are read as follows: the first numeral indicates
one of the five folding symmetries $1,2,3,4,5$ of the relevant hyperbolic lattice $\oT_\dP$,
and this symmetry is also depicted on the Coxeter diagram for $\oT_\dP$, with an $\times$
indicating that we reflect in the corresponding root. These correspond to the five $0$-cusps
added to $\fen$.

In blue is highlighted a 
maximal parabolic subdiagram invariant under the given folding symmetry. Necessarily,
all $\times$-ed vertices are contained in this diagram, since only these
diagrams can be invariant under the corresponding composition of root reflections.
Such blue diagrams are in bijection with the $1$-cusps incident upon the corresponding $0$-cusp.
The collection of all numerals, including the first label, indicate the corresponding $1$-cusp,
see Notation \ref{1-cusp-notation}.

Finally, adjacent to each maximal parabolic diagram for $\oT_\dP$ is the corresponding 
maximal parabolic subdiagram of the folded lattice $\oT_\enr$.

\begin{remark}
  The $1$-cusps $E_8D_8$ and $D_{16}$ of $\fhyp$ do not appear as
  images of $1$-cusps of $\fen$. The reason is now clear: these are
  exactly the two of the eight $1$-cusps of $\fhyp$ for which the
  parabolic subdiagrams, that can be found e.g. in
  \cite{alexeev2022mirror-symmetric}, are not symmetric with
  respect to any of the four involutions in
  Lemma~\ref{lem:cusps1345}.
\end{remark}

\begin{remark}
  The idea that folded diagrams may be relevant to compactifying
  $\fen$ implicitly appears in
  \cite{sterk1991compactifications-enriques1}, e.g. there is a folded
  $A_{15}$ diagram in Fig.~16.  We found that once the K3 case
  understood, the folding, when applied to the correct space---which
  is $\fhyp$ and not $F_4$---completely solves the Enriques case.
  Note that the moduli space $F_4$ of quartic K3 surfaces has a unique
  $0$-cusp with a non-reflective hyperbolic lattice, but the moduli
  space $\fhyp$ of hyperelliptic K3 surfaces of degree~$4$ has two
  $0$-cusps with reflective hyperbolic lattices; see
  Figure~\ref{fig-k3-cusps}.
\end{remark}

\section{Dlt models via integral-affine structures on the disk and $\bR\bP^2$}
\label{sec:ias}

\subsection{General theory}
\label{sec:IAS-general}

The general theory of integral affine spheres, $\ias$ for short, in
the form that we need it here is detailed in
\cite{engel2018looijenga, engel2021smoothings, alexeev2022compactifications-moduli,
alexeev2023stable-pair,
  alexeev2022mirror-symmetric}. We refer
the reader to the above papers for the necessary background, and give a broad summary now.

A {\it Kulikov model} is a $K$-trivial semistable model $\cX\to (C,0)$
of a degeneration of K3 surfaces over a pointed curve \cite{kulikov1977degenerations-of-k3-surfaces},
\cite{persson1981degeneration-of-surfaces}. For {\it Type III degenerations},
the dual complex $\Gamma(\cX_0)$ of the central fiber
is the $2$-sphere $S^2$, and for {\it Type II degenerations} $\Gamma(\cX_0)$ is a segment.
By \cite[Rem.~1.1v1]{gross2015mirror-symmetry-for-log}, \cite[Prop.~3.10]{engel2018looijenga}
there is a natural integral-affine structure on $\Gamma(\cX_0)$, with singularities.
The correct notion of singularities is detailed in \cite[Sec.~5]{alexeev2022compactifications-moduli}.

Fixing one Kulikov model $\cX\to (C,0)$, we get Kulikov models for all other degenerations
with the same Picard-Lefschetz transform, of the same
combinatorial type \cite[Lem.~5.6]{friedman1986type-III}, \cite[Def.~7.14]{alexeev2023compact}
by deforming the gluings and moduli of components.
We can extract the KSBA
stable limit of a degeneration $(\cX^*,\epsilon \cR^*)$ of K3 pairs, if
we can describe the {\it integral-affine polarization} $R_{\rm IA}\subset \Gamma(\cX_0)$,
a certain weighted balanced graph \cite[Def.~5.17]{alexeev2022compactifications-moduli}.
This weighted graph encodes the line bundle $\cO_{\cX_0}(\cR_0)$ on a {\it divisor model} $(\cX,\cR)$:
a Kulikov model  which admits a nef extension of $\cR_t$, $t\in C\setminus 0$,
containing no singular strata of $\cX_0$ \cite[Thm.~3.12]{alexeev2023stable-pair},
\cite[Thm.~2.11]{laza2016ksba-compactification}.

By \cite[Thm.~3.24]{alexeev2021nonsymplectic}, our chosen divisor $R$, as the fixed
locus of an automorphism $\iota_\dP$ on a general Enriques K3 surface, is {\it recognizable},
see \cite[Def.~6.2]{alexeev2023compact}. By the main theorem on recognizable divisors 
\cite[Thm.~1]{alexeev2023compact}, there is a unique semifan $\fF_R$ whose corresponding 
semitoroidal compactification  \cite{looijenga2003compactifications-defined2}, \cite[Sec.~5C]{alexeev2023compact} 
normalizes the KSBA compactification of $\fen$.
By \cite[Thm.~8.11(5)]{alexeev2023compact}, 
$(\Gamma(\cX_0),R_{\rm IA})$
can be chosen to be
 the same for all degenerations with fixed Picard-Lefschetz transform.

In turn, the combinatorial data of $(\Gamma(\cX_0),R_{\rm IA})$
determines the combinatorial type of the KSBA stable limit of the degeneration $(\overline{\cX}_0 ,\epsilon \overline{\cR}_0)$
by \cite[Cor.~8.13]{alexeev2023compact}. Then \cite[Thm.~9.3]{alexeev2023compact} gives an algorithm 
to determine $\fF_R$: Its cones are given by collections of Picard-Lefschetz transformations
for which $(\Gamma(\cX_0),R_{\rm IA})$ determines a KSBA-stable pair of a fixed combinatorial type.
This is the natural notion of combinatorial constancy of such pairs.

The possible Picard-Lefschetz transformations of Kulikov degenerations in $\ofen$
are encoded by a vector $\lambda\in \ch^J$ called the {\it monodromy invariant}. It is valued
in the fundamental chamber $\ch^J$ for one of the five folded diagrams $G^J=G(\oT_{\dP})^J$
of Figures \ref{fig-cusp2}, \ref{fig-cusps1345}
as in Lemma~\ref{lem:all-diagrams-are-folded}
because $\lambda$ must
be invariant under the involution $J$ on $\oT_\dP$.
An algorithm (albeit a complicated one), is provided
in \cite[Thm.~8.3]{alexeev2022mirror-symmetric} to build 
$(\Gamma(\cX_0), R_{\rm IA})$ for all monodromy invariants $\lambda\in \ch$
in the fundamental chamber for the Weyl
group action, for either hyperbolic lattice $\oT_\dP=(18,2,0)_1$ or $\oT_\dP = (18,0,0)_1$
corresponding to a $0$-cusp of $F_{(2,2,0)}$.

Restricting $\ias$ for $F_{(2,2,0)}$ to the involution-invariant 
sublattice $\oT=\oT_{\dP}^{J=1}$ exhibits a polarized $\ias$ for any
Type III degeneration in $\ofen$. Then, one can hope (and it is indeed the case,
as shown below), that on these subloci, the corresponding divisor models $(\cX,\cR)$
admit a second involution identified with the limit of the Enriques involution.
Thus, these polarized $\ias$ will provide a method to compute the 
Kulikov and KSBA-stable models of all degenerations of both the Enriques surfaces
and their corresponding double covers, the Enriques K3 surfaces.

\begin{definition} Let $\lambda\in \ch$, for the one of the two $0$-cusps of $F_{(2,2,0)}$. 
We define $$\ell := (\ell_i)_{i\in G} = (\lambda\cdot \alpha_i)_{i\in G}$$ where $\alpha_i$
are the roots of either diagram in Figure \ref{fig-k3-cusps}. Thus, $\ell\in (\bZ_{\geq 0})^{22}$
for the cusp $(18,2,0)_1$ and $\ell\in (\bZ_{\geq 0})^{19}$ for the cusp $(18,0,0)_1$.
 \end{definition}

\subsection{$\ias$ for $\fhyp$}
\label{sec:ias-hyp}

We now identify the polarized $\ias$ for degenerations in $F_{(2,2,0)}$ following the instructions
of \cite[Thms.~7.4,~8.3]{alexeev2022mirror-symmetric}. We treat each of the two $0$-cusps individually.

\medskip

{\bf Cusp $(18,2,0)_1$:} We are to first take a K3 surface $\hX$ in the mirror moduli space
for this $0$-cusp---these are $U(2)\oplus E_8^{\oplus 2}$-polarized K3 surfaces. Then we are to consider the
anticanonical pair quotient $$(\hY,\hD) := \hX/\,\widehat{\iota}_\dP$$ by the mirror involution and, 
for each $\hL$ in the nef cone of $\hY$, we must build a Symington polytope $P(\ell)$ for the line bundle $\hL\to (\hY,\hD)$
corresponding to $\ell$, see \cite{symington2003four-dimensions}, \cite[Construction 6.16]{alexeev2022mirror-symmetric}.
We build a sphere $B(\ell) = P(\ell)\cup P(\ell)^{\rm opp}$ 
by identifying two copies of this integral-affine disk
along their common boundary, to form the equator of the sphere. Then $B(\ell)=\Gamma(\cX_0)$ 
for a monodromy-invariant $\lambda\leftrightarrow \ell\leftrightarrow \hL$ and the integral-affine polarization 
$R_{\rm IA}$ corresponding to the flat limit $\cR_0\subset \cX_0$ is the equator of the sphere,
with weights alternating $2$ and $1$.

The anticanonical pair $(\hY,\hD)$ is a rational elliptic surface with an anticanonical cycle of $16$ curves,
of alternating self-intersections $-1$ and $-4$, which result from blowing up the corners of an $I_8$ Kodaira fiber.
This pair admits a toric model $$(\hY,\hD)\to (\overline{\hY},\overline{\hD})$$ whose fan is depicted
on the left-hand side of Figure \ref{fig-maxdegs}. The rays going to the four corners correspond to components of the 
toric model receiving an {\it internal blow-up}, i.e.~a blow-up at a smooth point
of the anticanonical boundary.

A moment polytope $\oP(\ell)$ for the line bundle
$\overline{\hL}\to  (\overline{\hY},\overline{\hD})$ is depicted on the left of Figure \ref{ias-for-220} and a 
Symington polytope $P(\ell)$ for the line bundle $\hL\to (\hY,\hD)$ corresponding to $\ell$ is depicted on
the right of Figure \ref{ias-for-220}. The right hand-side also serves as a visualization of
each hemisphere of the integral-affine sphere $B(\ell)=\Gamma(\cX_0)$, with the equator 
in blue and integral-affine singularities in red. The quantities $\ell_{20}$ and $\ell_{21}$
are, respectively, twice the lattice length between the singularities introduced by Symington
surgeries on opposite sides of the figure.

\begin{figure}
\centering
  \includegraphics[width=360pt]{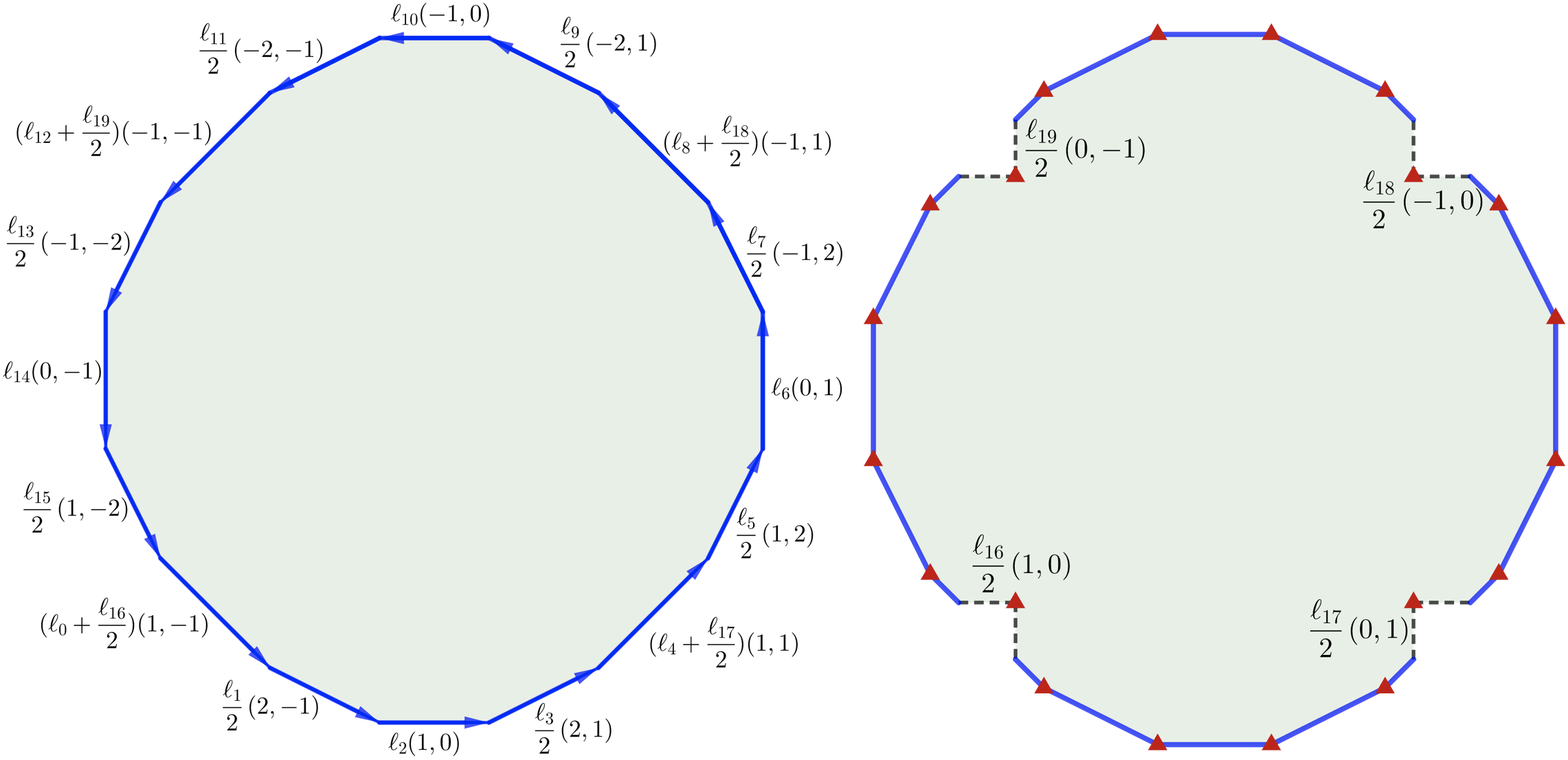} 
  \caption{Moment and Symington polytopes for cusp $(18,2,0)_1$}
  \label{ias-for-220}
  
 \bigskip 
  
  \includegraphics[width=260pt]{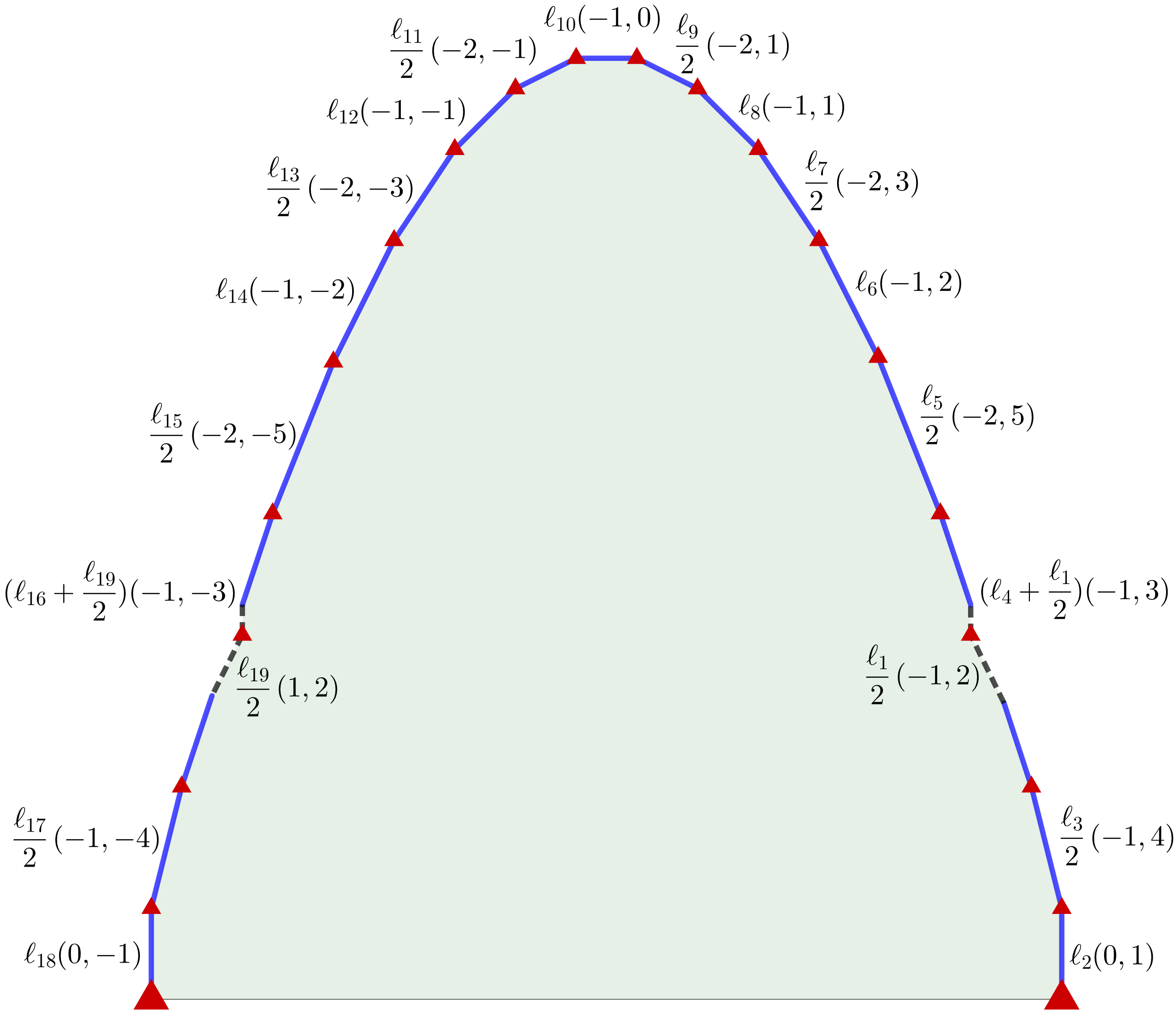} 
  \caption{Symington polytope for cusp $(18,0,0)_1$}
  \label{ias-for-200}
\end{figure}

\medskip

{\bf Cusp $(18,0,0)_1$:} The procedure for constructing polarized $\ias$ at this cusp
is essentially the same as the above, instead taking the mirror moduli space to be $U\oplus E_8^{\oplus 2}$-polarized.
The fan of a toric model of the mirror is provided by the right hand side of Figure \ref{fig-maxdegs}.
The integral-affine structures are similar to those depicted in \cite[Fig.~4]{alexeev2022compactifications-moduli}, with an important difference: 
The cusp $(18,0,0)_1$ corresponds to a non-simple mirror
of $F_{(2,2,0)}$. This means that the integral-affine polarization 
$R_{\rm IA}\subset B(\ell)$ has no support on 
the bottom edge of the Symington polytope $P$, and $\cR_0$
is empty on the corresponding components of $\cX_0$. 
See the discussion of a ``B-move" in \cite[Sec.~8D]{alexeev2022mirror-symmetric}
for further details.

The $\ias$ we need is the result of taking the $\ias$
of \cite[Fig.~4]{alexeev2022compactifications-moduli}, splitting
the $I_2$ singularity at the bottom into two $I_1$ singularities traveling
in opposite directions, and colliding
each one with a corner.
This produces singularities in the bottom left and right corners
of charge $2$, depicted with a larger red triangle, see Figure \ref{ias-for-200}.

\medskip

\begin{remark} Note that in both cases, certain coordinates of $\ell$ must be divisible by $2$ to build
the polarized $\ias$. This does not present
an issue, since we only need divisor models for all sufficiently divisible $\lambda$. \end{remark}

\begin{remark} For the cusp $(18,0,0)_1$, the polygon $\oP(\ell)$ in Figure \ref{ias-for-200} 
can be closed by a horizontal base exactly because of relation (\ref{close-up}). The same holds
for cusp $(18,2,0)_1$ with relation (\ref{eq:UE82}) and other similar relations. \end{remark}

To summarize, by \cite[Thm.~8.3]{alexeev2022mirror-symmetric}, we have:

 \begin{theorem}\label{dP-kulikovs} Let $(B(\ell),\,R_{\rm IA})$ be the polarized
 $\ias$ built above, from $\ell\in (\bZ_{\geq 0})^{22}$ or $(\bZ_{\geq 0})^{19}$. Then, upon triangulation
 into lattice simplices, $$(B(\ell),R_{\rm IA})=
 (\Gamma(\cX_0),\Gamma(\cR_0))$$ is the dual complex of the central fiber $(\cX_0,\cR_0)$
 of a divisor model $(\cX,\cR)\to (C,0)$, whose monodromy invariant $\lambda\in \ch$ satisfies
 $\ell = (\lambda\cdot \alpha_i)_{i\in G}$. \end{theorem}


\subsection{${\rm IA}\bD^2$ and ${\rm IA}\bR\bP^2$ for $\fen$}
\label{sec:ias-ent}

Now suppose that $(\cX,\cR)\to (C,0)$ is a Type III divisor model as in Theorem \ref{dP-kulikovs}, whose period
map $C^*\to F_{(2,2,0)}$ factors through $\fen$. Then, the general fiber $\cX_t$ is an Enriques K3 surface
with degree $4$ polarization, and $(\cX,\cR)\to (C,0)$ is a divisor model
for the degeneration. 
The quotient $$\cX^*/\iota_{\enr}^* = \cZ^*\to C^*$$ of the general fiber
by the Enriques involution is a degenerating family of Enriques surfaces.
The monodromy invariant $\lambda\in \ch^J$ then necessarily lies in the fundamental chamber
for one of the five $0$-cusps of $\ofen$. Equivalently, $\ell$ must be 
invariant under one of the five folding symmetries.

\begin{proposition}\label{kulikov-prop} Let $\lambda\in \ch^J$, $\ell = (\lambda\cdot \alpha_i)_{i\in G}$.
The folding symmetry $J$ on $\oT_\dP$ induces an isomorphism 
 $\iota_{\enr, {\rm IA}}$ of the polarized $\ias$ $(B(\ell),\,R_{\rm IA})$ of Theorem \ref{dP-kulikovs}.
 The dual complexes $\Gamma(\cX_0,\cR_0)$ of divisor models for Enriques
 K3 surface degenerations in $\fen$ are exactly those admitting the additional symmetry $\iota_{\enr, {\rm IA}}$
 (appropriately interpreted for $\times$-ed nodes in cusps 3, 5).
 \end{proposition}
 
 \begin{proof} For each $0$-cusp, we directly inspect the $\ias$ for the parameters
 $\ell$ corresponding to $\lambda\in \ch^J$ and see that there is an additional symmetry
of $B(\ell)$.
 
 \smallskip
{\bf Cusp 2:} We have $\lambda\in \ch^J$ if and only if $\ell_i=\ell_{20-i}$ for all $i=1,\dots,9$.
The $\ias$ in Figure \ref{ias-for-200} then has a visible symmetry, which is to act on the
both the hemisphere $P$, and its opposite hemisphere $P^{\rm opp}$, by a flip across the vertical line
bisecting the bottom and top edges.

 \smallskip
{\bf Cusp 1:} We have $\lambda\in \ch^J$ if and only if $\ell_i=\ell_{8+i}$ for all $i=0,\dots, 7$,
$\ell_{16}=\ell_{18}$ and $\ell_{17}=\ell_{19}$. Then the corresponding involution  $\iota_{\enr, {\rm IA}}$
of the $\ias$ is to rotate each hemisphere, shown in Figure \ref{ias-for-220},
by $180$ degrees, and then flip the two hemispheres $P$
and $P^{\rm opp}$.

 \smallskip
{\bf Cusp 3:}  We have $\lambda\in \ch^J$ if and only if $\ell_i=\ell_{16-i}$ for all $i=1,\dots, 7$,
$\ell_{17}=\ell_{19}$, and
$\ell_{20}=0$. This is because the folding symmetry also reflects in the root $\alpha_{20}$. 
So if $w_{\alpha_{20}}(\lambda)=\lambda$, then $\ell_{20}=\lambda\cdot \alpha_{20}=0$.

Recall that $\ell_{20}$ is the lattice distance between the singularities introduced by the Symington surgeries
resting on the edges parallel to $(1,-1)$ and $(-1,1)$, 
on the right-hand side of Figure \ref{ias-for-220}. We construct
$B(\ell)$ in such a way that the two singularities introduced by these Symington surgeries coincide. 
The involution $\iota_{\enr,{\rm IA}}$
of the $\ias$ acts by flipping each hemisphere $P$, $P^{\rm opp}$ diagonally.

 \smallskip
{\bf Cusp 4:} Similar to Cusp 2, we have $\lambda\in \ch^J$ if and only if each hemisphere
of $B(\ell)$ is symmetric with respect to flipping along a horizontal line 
bisecting the edges $\ell_6(0,1)$ and $\ell_{14}(0,-1)$.

 \smallskip
{\bf Cusp 5:} We have $\lambda\in \ch^J$ if and only if $\ell_{2i+1}=0$ for $i=0,\dots,7$. We declare
that $\iota_{\enr,{\rm IA}}$ act in the same manner as the extension of $\iota_{\dP}$ to $\cX_0$: It
flips the two hemispheres $P$ and $P^{\rm opp}$. The eight $\times$-ed nodes 
correspond to eight collisions of pairs of $I_1$ singularities along the equator.

\smallskip

By \cite[Prop.~6.17]{alexeev2022mirror-symmetric}, the mirror K3 surface $\hX$ admits a symplectic
form $\omega$ and Lagrangian torus fibration $$\mu\colon (\hX,\omega) \to B(\ell),$$ 
for generic $\ell\in \ch^J$. Note that while some of the $24$ $I_1$-singularities collide on $B(\ell)$
for Cusps 3 and 5, we only ever get, for generic $\ell$,
a collision of two $I_1$-singularities with parallel $\SL_2(\bZ)$-monodromies. So the 
fibration $\mu$ still exists, but has $I_2$ fibers over these collisions.

The involution $\iota_{\enr, {\rm IA}}$ acting on $B(\ell)$ induces an involution of the Lagrangian
torus fibration $(\hX,\mu)$ 
and in turn on ${\rm Pic}(\hX)\simeq \oT_\dP=(18,2,0)_1$ or $(18,0,0)_1$ which is generated by classes
of visible curves, cf.~\cite[Sec.~6G]{alexeev2022mirror-symmetric}. In the current setting, the visible
curves (which correspond to the roots $\alpha_i$)
are all of the following simple form: A path connecting two $I_1$-singularities with parallel 
$\SL_2(\bZ)$-monodromies. For Cusps 1, 2, 4, the involution $\iota_{\enr, {\rm IA}}$ 
acts on the classes of visible curves by the Enriques involution on $\oT_\dP$ and thus, 
by the Mirror/Monodromy Theorem \cite[Prop.~3.14]{engel2021smoothings},
\cite[Thms.~6.19, 7.6]{alexeev2022mirror-symmetric}, $B(\ell)$ is the dual complex
of a degeneration with a monodromy invariant in $\oT_\dP^{J=1}$. 

Some additional care must be taken for Cusps 3 and 5, where $B(\ell)$ is a limit
of $\ias$ with $24$ distinct $I_1$-singularities. For each $\times$-ed node, the
involution $J$ acts on $\oT_\dP$ by reflecting along $\alpha_i$ and so the
class $[\omega]$ of the symplectic form should satisfy $[\omega]\cdot \alpha_i=0$.
Equivalently, there should be a nodal slide, see \cite[Sec.~6E]{alexeev2022mirror-symmetric},
which collides the two $I_1$ singularities of the visible curve corresponding to $\alpha_i$
into an $I_2$ singularity. This is indeed the case for the $\ias$ described above. To summarize,
invariance under reflection of an $\times$-ed node $\alpha_i$ corresponds, on the $\ias$, 
to colliding the two $I_1$ singularities bounding the corresponding visible curve.

We conclude that an $\iota_{\enr, {\rm IA}}$-invariant polarized $\ias$, 
which has a coalescence to an $I_2$-singularity for each $\times$-ed node, is the dual 
complex of a divisor model $(\cX,\cR)\to (C,0)$ for degree $4$ Enriques K3 surfaces
whose monodromy invariant $\lambda\in \ch^J$ is generic. The passage from
the result for generic $\lambda\in \ch^J$ to all $\lambda\in \ch^J$ is a standard trick involving a limit
procedure on the corresponding $\ias$, examining $B(\ell)$ as some $\ell_i\to 0$,
see \cite[Thm.~6.29]{alexeev2023stable-pair}, \cite[Sec.~6G]{alexeev2022mirror-symmetric}.
\end{proof}
 
\begin{theorem}\label{enr-kul} For all $\lambda\in \ch^J$, the general divisor model $(\cX,\cR)\to (C,0)$
with monodromy invariant $\lambda$ admits a second involution $\iota_\enr\colon \cX\to \cX$
extending the Enriques involution on the general fiber, and satisfying $\iota_\enr(\cR)=\cR$. \end{theorem}

\begin{proof} The proof is essentially the same as \cite[Thm.~8.3]{alexeev2022mirror-symmetric}.
The key point is that the Kulikov models $\cX_0$ which arise as limits of Enriques K3s are those
whose period point $\varphi_{\cX_0}\in {\rm Hom}(\lambda^\perp(\oT_\dP),\bC^*)$ is anti-invariant under $\iota_{\enr, {\rm IA}}$---we require anti-invariance because $\iota_{\enr, {\rm IA}}$
 acts in an orientation reversing manner on $\Gamma(\cX_0)$.

The anti-invariant periods are those $\varphi_{X_0}$ for which
$(S_\dP)^\perp_{S_\enr}= E_8(2)\subset \ker(\varphi_{\cX_0})$,
and the smoothings keeping these
classes Cartier are exactly those admitting an $S_\enr$-polarization
(and hence admitting an Enriques involution). 
Finally, the Kulikov surfaces $\cX_0$ with an anti-invariant period are also identified with those
admitting an additional involution $\iota_{\enr,0}$ because the $\bC^*$-moduli of components and their gluings
are made invariantly with respect to the action of $\iota_{\enr, {\rm IA}}$ on the gluing complex 
\cite[Def.~5.10]{alexeev2023compact}. Furthermore,
the deformations keeping the involution $\iota_{\enr,0}$ 
are then identified with those keeping the $S_\enr$-polarization.

Since the divisor model is generic, \cite[Thm.~8.3]{alexeev2022mirror-symmetric}
implies that $\cR\subset \cX$ is the divisorial component of the fixed locus
of an involution $\iota_{\dP}$ on the threefold $\cX$ extending the del Pezzo
involution on the general fiber. Then $\iota_\dP$ and $\iota_{\enr}$ commute
on the general fiber and hence commute on all of $\cX$. So $\iota_\enr$
preserves $\cR$.
\end{proof}

More generally, every degeneration of Enriques surfaces admits a divisor model $(\cX,\cR)\to (C,0)$
for which $\iota_{\enr}$ defines a birational involution, and for which the union of the fixed locus
and the locus of indeterminacy contains $\cR$.

\begin{definition}\label{half-divisor} Let $\cX\to (C,0)$, $(\cX,\cR)\to (C,0)$ be a Kulikov, resp.~divisor, 
model of Enriques K3 surfaces for which $\iota_{\enr}$
defines a regular involution on $\cX$, resp.~preserving $\cR$. 
We define the {\it dlt model}, resp.~the {\it half-divisor model}, to be 
the quotient by the Enriques involution:
\begin{displaymath}
  \cZ:= \cX/\iota_{\enr},
  \quad\textrm{resp.~}\quad
  (\cZ,\cR_\cZ) := (\cX,\cR)/\iota_{\enr}.
\end{displaymath}

\end{definition}

\begin{proposition}\label{half-kulikov-prop} Let $(\cZ,\cR_\cZ)\to (C,0)$ be a half-divisor model for $\fen$ for the divisor
models constructed in Proposition \ref{kulikov-prop}. Then, the fibers of $\cZ$ have slc singularities,
$K_\cZ+\epsilon \cR_\cZ$ is relatively big and nef over $C$, and $\cR_\cZ$ contains no log canonical centers.
 In Type III, for cusp number
\begin{enumerate}
\item[(1)] we have $\Gamma(\cZ_0) = \bR\bP^2$. Each component $V_i\subset \cZ_0$ is isomorphic, 
up to normalization, with either of the two connected components of its inverse image in $\cX_0$.
\smallskip
\item[(2--5)] we have $\Gamma(\cZ_0) = \bD^2$. If the component $V_i\subset \cZ_0$ is covered by two irreducible
components of $\cX_0$ then up to normalization, $V_i$ is isomorphic to either of these components. If 
$V_i\subset \cZ_0$ is covered by one irreducible component of $\wV_i\subset \cX_0$ then $\iota_{\enr,0}$
acts on $\wV_i$ with exactly four fixed points, two pairs of points on appropriately chosen double curves $\wD_{ij}$
and $\wD_{ik}\subset \wV_i$. 
\end{enumerate} 
In Type II, $\Gamma(\cZ_0)$ is a segment.
For cases in Fig.~\ref{fig-sterk-allcusps} with a double rectangle, $\iota_{\enr,0}$ acts
by flipping $\Gamma(\cX_0)$ and fixing no points of $\cX_0$. Assuming that $\cX_0$ contains a double curve $E$
preserved by $\iota_{\enr,0}$, the action of the involution on $E$ is a nontrivial $2$-torsion
translation. For cases in Fig.~\ref{fig-sterk-allcusps} with a single rectangle, $\iota_{\enr,0}$ preserves
every component of $\cX_0$. On any double curve, the action is by an elliptic involution fixing exactly
four points. There are no other fixed points on $\cX_0$.
\end{proposition}

\begin{proof} In Type III, the homeomorphism type of $\Gamma(\cZ_0)$ follows directly from the description
of the action of $\iota_{\enr, {\rm IA}}$ in Proposition \ref{kulikov-prop}. In Type II, we can construct divisor models
$(\cX,\cR)\to (C,0)$ by taking limits of $B(\ell)$ as it collapses to a segment, or equivalently as $\lambda$ approaches
a rational isotropic ray at the boundary of $\ocC$.

The resulting central fiber $\cX_0$
is a chain of surfaces and by arguments similar to Theorem \ref{enr-kul}, a general degeneration
$(\cX,\cR)$ to the given Type II boundary divisor admits an additional involution $\iota_\enr$ which acts
on $\Gamma(\cX_0)$ by the limit of the action of $\iota_{\enr,{\rm IA}}$ on the segment $B(\ell)$.
This gives the claimed action on $\Gamma(\cX_0)$, by direct examination of the limiting dual segment
$B(\ell)$ for all entries of Figures \ref{fig-par2}, \ref{fig-pars}.

If $\iota_{\enr,0}$ permutes two irreducible components, it is clear that the normalization of the quotient
agrees with the normalization of either component. 


So suppose $\iota_{\enr,0}$ preserves the pair $(\wV_i,\wD_i)\subset \cX_0$ (necessarily we
are at a Cusp 2--5).
Possibly assuming further divisibility of $\lambda$, and choosing our triangulation of 
$B(\ell)\simeq S^2$ appropriately, we may assume that the fixed locus of $\iota_{\enr,{\rm IA}}$
is a circle $S^1\subset B(\ell)$ formed from a collection of vertices $\widetilde{v}_i$
and edges $\widetilde{e}_{ij}$ of $\Gamma(\cX_0)$. 

Denote the corresponding collection of components 
$\widetilde{V}_i\subset \cX_0$ the {\it Enriques equator}. For Cusps 2, 3, 4 the Enriques equator
is distinct from the {\it del Pezzo equator}, which corresponds to the common glued boundary 
of $P$ or $P^{\rm opp}$ and supports $\Gamma(\cR_0)\subset \Gamma(\cX_0)$. 
For Cusp 5, the Enriques and
del Pezzo equators coincide since $\iota_{\dP,{\rm IA}}= \iota_{\enr,{\rm IA}}$.

The logarithmic $2$-form on $\cX_0$ is of the form
 $dx\wedge dy$, $\frac{dx}{x}\wedge dy$, or $\frac{dx}{x}\wedge \frac{dy}{y}$
depending, respectively, on whether $(x,y)$ are local coordinates (in a component) at a
smooth point, a point in a double curve, or a triple point of $\cX_0$.
Since $\iota_{\enr,0}$ has no divisorial fixed locus and is non-symplectic, it fixes at most
a finite subset of $\cX_0$ contained in the double locus, where  
$(x,y)\mapsto (-x,-y)$ is non-symplectic. 

So in Type III, the only fixed
points of $\iota_{\enr,0}$ are points in some $\wD_{ij}\subset \cX_0$ along the Enriques
equator. Being an involution of $\wD_{ij}\simeq \bP^1$, there must be exactly
 $2$ such fixed points. We recover then a similar phenomenon as for the Kulikov models 
 of Enriques degenerations which were described in \cite[Sec.~8F]{alexeev2022mirror-symmetric}.

 In Type II, the analysis is similar: If $\iota_{\enr,0}$ preserves a component $\wV_i$,
 then the involution on a preserved double curve $\wD_{ij}\simeq E$ is locally given by 
 negation on $E$. So the induced action on this (and in turn any) 
 double curve is an elliptic involution. On the other hand, suppose $\iota_{\enr,0}$ permutes
 the two components containing $E$. Then, since the residues ${\rm Res}_E\omega_{\wV_i}=
 - {\rm Res}_E\omega_{\wV_j}$ from the two components
 of the logarithmic two-form are negatives of each other, we must have by non-symplecticness
 that $\iota_{\enr, 0}$ preserves a holomorphic one-form on $E$. So it acts by a $2$-torsion translation, 
 nontrivial because the fixed locus is finite.

Having analyzed the action of $\iota_{\enr, 0}$ on $\cX_0$, we see the quotient $\cZ_0$ 
has SNC singularities at all points, except the images of the fixed points along 
the double curves of the Enriques equator. Here the local equation of the quotient is 
\begin{equation} \label{double-sing}\{(x,y,z)\in \bC^3\,\big{|}\, xz=0\}/(x,y,z)\mapsto (-x,-y,-z) \end{equation}
 which is slc,
with the only log canonical center being the image of the double locus $x=z=0$. So $\cZ_0$
has slc singularities.

Since the fixed locus is finite, $K_\cZ$ is numerically trivial.
Furthermore, $\cR_\cZ$ inherits the property of containing
no log canonical centers, and being relatively big and nef, from $\cR$---it is important
here that no log canonical center was introduced at $(0,0,0)$ in the above quotient (\ref{double-sing}).
The proposition follows.
\end{proof}

\begin{corollary}\label{hk-cor} The KSBA-stable limit of a degeneration of $(\cZ^*,\epsilon \cR_\cZ^*)\to C^*$ 
can be computed from the half-divisor model $(\cZ,\cR_\cZ)\to (C,0)$ of Proposition \ref{half-divisor} as 
$${\rm Proj}_C\, \bigoplus_{n\geq 0} H^0(\cZ, n\cR_\cZ).$$
\end{corollary}

This also furnishes a somewhat inexplicit description of the components of 
the KSBA-stable limit $(\overline{\cZ}_0,\epsilon \overline{\cR}_{\overline{\cZ},0})$: First, we take a component 
$(\wV_i,\wD_i,\wR_i)\subset (\cX_0,\cR_0)$ of the divisor model for the degeneration in $F_{(2,2,0)}$. This
is, up to corner blow-ups, the minimal resolution of an ADE surface of \cite{alexeev17ade-surfaces}.
Then, we impose the condition that the periods and dual complex of 
$(\cX_0,\cR_0)$ are involution invariant. Now, the Torelli theorem for anticanonical pairs
\cite[Thm.~1.8]{gross2015moduli-of-surfaces}, \cite[Sec.~8]{friedman2015on-the-geometry} 
implies that $(\wV_i,\wD_i,\wR_i)$ admits an involution $\iota_{\enr, i}$ which acts
in an orientation-reversing manner on the cycle $\wD_i$. Then the quotient 
$$(V_i,D_i+\epsilon R_i):=(\wV_i,\wD_i+\epsilon\wR_i)/\iota_{\enr, i}$$ is a log Calabi-Yau pair of index $\leq 2$
with $R_i$ big and nef. The stable component 
$$(\oV_i,\oD_i+\epsilon \oR_i)\subset (\overline{\cZ}_0,\epsilon\overline{\cR}_{\overline{\cZ}_0})$$
is (up to normalization) the result of contracting all curves which intersect $R_i$ to be zero.
Alternatively, we can reverse the order, taking first the stable model of $(\wV_i,\wD_i+\epsilon\wR_i)$
to get an ADE surface of \cite{alexeev17ade-surfaces} forming a component of 
$(\overline{\cX}_0,\epsilon\overline{\cR}_0)$ and then taking the quotient by the induced 
involution $\overline{\iota}_{\enr, i}$. These stable surfaces and their
quotients are described further in Section \ref{sec:abcde}. 

\begin{remark} The quotient $\Gamma(\cZ_0)=\Gamma(\cX_0)/\iota_{\enr, {\rm IA}}$
of the dual complex inherits naturally an integral-affine structure (with boundary in the case
of a $\bD^2$ quotient) from $\Gamma(\cX_0)$. For the $\bD^2$ case,
the components forming the boundary of $\Gamma(\cZ_0)$ 
are exactly the image of the Enriques equator, and they are the only singular components
of $\cZ_0$, each component having $4$ total $A_1$ singularities.
\end{remark}

\begin{remark}\label{dlt} We have only proven that a
half-divisor model $(\cZ,\cR_\cZ)\to (C,0)$ exists for generic degenerations
with a given Picard-Lefschetz transform $\lambda$, since $\iota_{\enr}$
will in general be a birational involution. This issue arises even for Type I degenerations, when
$\cX_0$ acquires a $(-2)$-curve. 
If one contracts the ADE configurations in components 
of $\cX_0\subset \cX$ forming the loci of indeterminacy of $\iota_{\enr}$,
this issue does not arise and $\iota_\enr$ defines a morphism.
In general, the pair $(\cZ,\cZ_0)$ will only be dlt.
This lends further weight to the notion that dlt models are the correct
analogue of Kulikov models in the more general setting of
$K$-trivial degenerations, see \cite{defernex2017dual-complex, kollar2018remarks}.
\end{remark}

\begin{remark} In \cite{morrison1981semistable}, Morrison gave a description of semistable degenerations
of Enriques degenerations, in the style of the Kulikov-Persson-Pinkham
theorem
\cite{kulikov1977degenerations-of-k3-surfaces, persson1981degeneration-of-surfaces}.
The description of irreducible components
and how they are glued is quite intricate (and floral), 
involving {\it flowers}, {\it pots}, {\it stalk assemblies}, and {\it corbels}.

On the other hand, by Proposition \ref{half-kulikov-prop} and Remark \ref{dlt}, we have, 
in all cases, a relatively simple dlt model, whose singularities are SNC, except for some copies
of the singularity with equation \eqref{double-sing} on double loci, and some klt singularities which are $S_2$-quotients
of ADE singularities in the interiors of components.

The corbels of {\it loc.cit.}~correspond to the singularity (\ref{double-sing})
while the flowers and stalk assemblies
are the semistable resolutions of the $S_2$-quotients of the ADE singularities, in the total space of the smoothing.
Finally, the pots are the components of our dlt models along which the flower and stalk assembly are attached.

The cases (i a), (i b), (ii a), (ii b), (iii a), (iii b) of
\cite[Cor.~6.2]{morrison1981semistable} correspond, respectively, to Type I degenerations, Type I degenerations
with a klt singularity, Type II degenerations with Enriques involution flipping the segment, Type II degenerations with
Enriques involution fixing the segment,
Type III degenerations with $\Gamma(\cZ_0) = \bR\bP^2$ (Cusp 1), and finally Type III degenerations with
 $\Gamma(\cZ_0)=\bD^2$ (Cusps 2--5).
\end{remark}

\subsection{Examples}
\label{sec:examples} We give some examples of divisor and half-divisor models. To distinguish notationally
between different $0$-cusps, we write $B_k(\ell)$, $k\in \{1,\dotsc, 5\}$
for the folding-symmetric polarized $\ias$ at Cusp $k$, from Proposition \ref{kulikov-prop}.

\begin{example}\label{ex1} 
  $B_3(2,0^{15},2,4,6,4,0,4)$:
  Consider Cusp 3, with the diagonal folding
symmetry, and set $\ell = (2,0^{15},2,4,6,4,0,4)\in (\bZ_{\geq 0})^{22}$. Then from Section~\ref{sec:ias-hyp}, the moment polygon $\oP(\ell)$ for the toric model of the mirror 
is the sequence of vectors $(3,-3)$, $(2,2)$, $(-3,3)$, $(-2,-2)$ put successively end-to-end.

We perform Symington surgeries of size $1, 2, 3, 2$ along the four
 edges $(3,-3)$, $(2,2)$, $(-3,3)$, $(-2,-2)$ respectively, because
$(\ell_{16},\ell_{17},\ell_{18},\ell_{19})=(2,4,6,4)$. The result is the Symington
polytope $P(\ell)$. Glue $P(\ell)$ and $P(\ell)^{\rm opp}$ to produce
$B_3(\ell)$, which is depicted on the left of Figure \ref{halfKex1} (of course, only a fundamental
domain of the sphere $S^2$ can be depicted on flat paper).
Five red triangles depict the integral-affine singularities, with their {\it charge}
\cite[Def.~5.3]{alexeev2022compactifications-moduli} shown in red.

\begin{figure}
\centering
  \includegraphics[width=360pt]{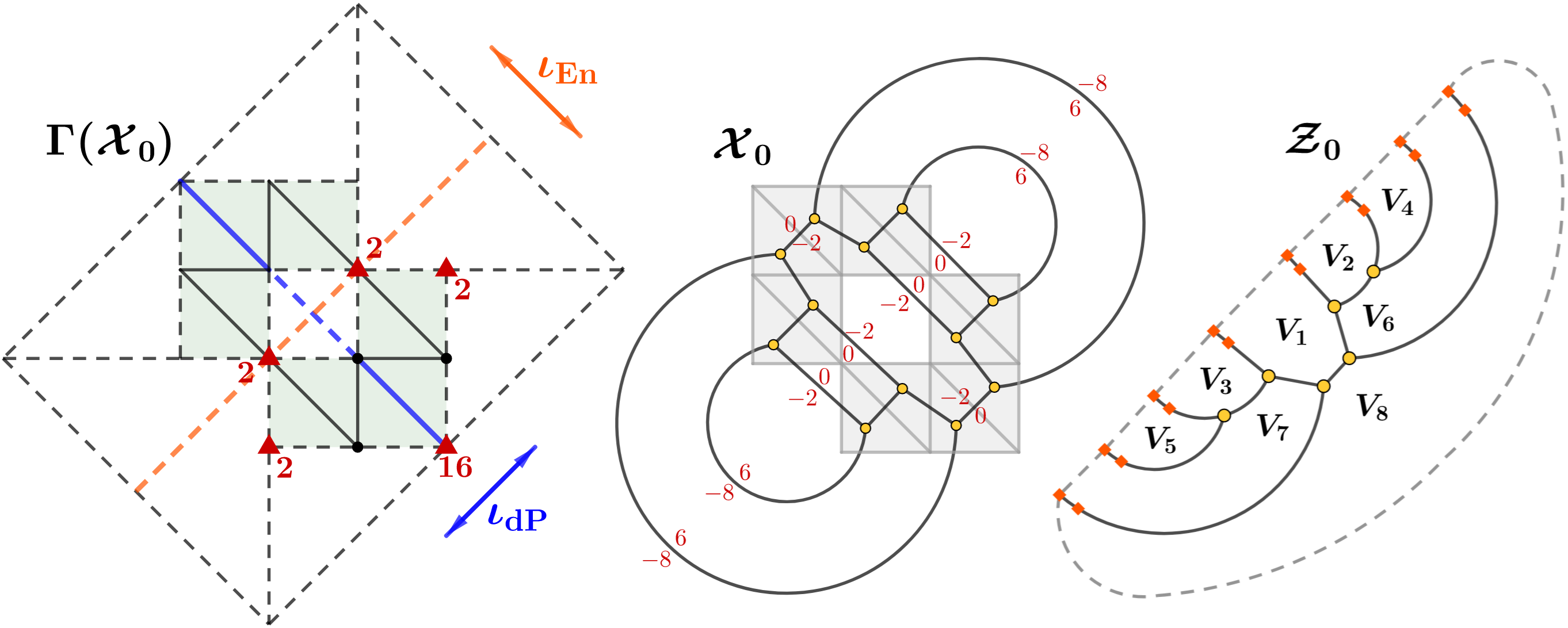} 
  \caption{$B_3(\ell)$ and central fibers for $\ell =(2,0^{15},2,4,6,4,0,4)$.}
  \label{halfKex1}
\end{figure}

The $\ias$ is then admits two involutions, Enriques and del Pezzo,
whose actions are shown in orange and blue, respectively. The corresponding
Enriques and del Pezzo equators are shown in the respective colors. A triangulation
into (green) lattice triangles is chosen, subordinate to both equators.
The blue del Pezzo equator, with integer weight $2$,
 forms the integral-affine polarization $R_{\rm IA}$.

The middle image of Figure \ref{halfKex1} depicts the corresponding Kulikov model $\cX_0$
of Enriques K3 degeneration. Triple points $\wT_{ijk}= \wV_i\cap \wV_j\cap \wV_k$ 
are depicted in yellow, double curves $\wD_{ij} = \wV_i\cap \wV_j$ are depicted in black.
The self-intersection numbers $$\wD_{ij}\big{|}_{\wV_i}^2+\wD_{ij}\big{|}_{\wV_j}^2=-2$$
are written in red (suppressed when both equal $-1$). The faces, including an outer face, represent the components $\wV_i$ with their anticanonical cycles $\wD_i = \sum_j \wD_{ij}$.

The righthand of Figure \ref{halfKex1} depicts the dlt model. It consists
of eight components $V_i$, $i=1,\dotsc,8$. Double loci and triple points
are still depicted in black and yellow. Successive components along the image
of the Enriques equator are $$V_1\cup V_2\cup  V_4 \cup V_6\cup V_8\cup V_7\cup V_5\cup V_3\cup V_1$$ and the double curves between these two components
contain two $A_1$-singularities of either containing surface, depicted by orange diamonds. 

The double covers $(\wV_6,\wD_6) \simeq (\wV_7,\wD_7)$ are toric, isomorphic to 
a two-fold corner blowup of $\bF_6$ as is $(\wV_1,\wD_1)$, which is the blow-up of the four
corners of an anticanonical square in $\bP^1\times \bP^1$.

The double covers
$(\wV_2,\wD_2)\simeq (\wV_3,\wD_3)$ are the internal blow-ups of $\bP^1\times \bP^1$
at two points $p,q$ on opposite components of an anticanonical square. The Enriques
involution acts in the corresponding toric coordinates by $(x,y)\mapsto (x^{-1},-y)$,
and thus for this involution to lift to the internal blow-up, the two blow-up points
must be interchanged: $y(p)=-y(q)$. This corresponds to choosing
the involution anti-invariant periods on $\cX_0$ for the unique $\times$-ed node
at Cusp 3.

The double covers $(\wV_4,\wD_4)\simeq (\wV_5,\wD_5)$ are both isomorphic to $\bF_2$.
Finally, $(\wV_8,\wD_8)$ is a minimal resolution of the $A_{15}$ surface of \cite{alexeev17ade-surfaces}. It is the $16$-fold internal blowup of $\bF_8$ at $16$ points on a section $s$, $s^2=8$. These $16$ points are placed symmetrically with respect to an involution
of $s$ and $\bF_8$, giving rise to an Enriques involution on $(\wV_8,\wD_8)$. 

The divisor $\cR_{\cZ_0}\subset \cZ_0$ is entirely supported on $V_1\cup_{D_{18}} V_8$
and has intersection number $\cR_{\cZ_0}\cdot D_{18} = 1$. We have that $R_1^2=0$
and $R_1\subset V_1$ is the image of two fibers of a toric ruling on $\wV_1$ while 
$R_8\subset V_8$ satisfies $R_8^2=8$ as it is the reduced image of the fixed locus $\wR_8\subset \wV_8$ satisfying $\wR_8^2=16$.

The map to the stable model $(\overline{\cZ}_0,\epsilon \overline{\cR}_{\overline{\cZ}_0})$ contracts
all components except $V_1$ and $V_8$ to points and contracts $V_1$ along a ruling,
leaving the image of $V_8$ as the only component. 
The normalization $$(\oV_8,\oD_8+\epsilon \oR_8)\simeq  
(\overline{\cZ}_0, \epsilon \overline{\cR}_{\overline{\cZ}_0})^\nu$$
has, as anticanonical boundary $\oD_8\simeq \bP^1$,
which is self-glued in $\overline{\cZ}_0$ along an involution fixing $0,\infty\in \oD_8$.
The singularities at $0,\infty\in \oV_8$ are rather complicated.

\end{example}

\begin{figure}[htbp]
\centering
  \includegraphics[width=360pt]{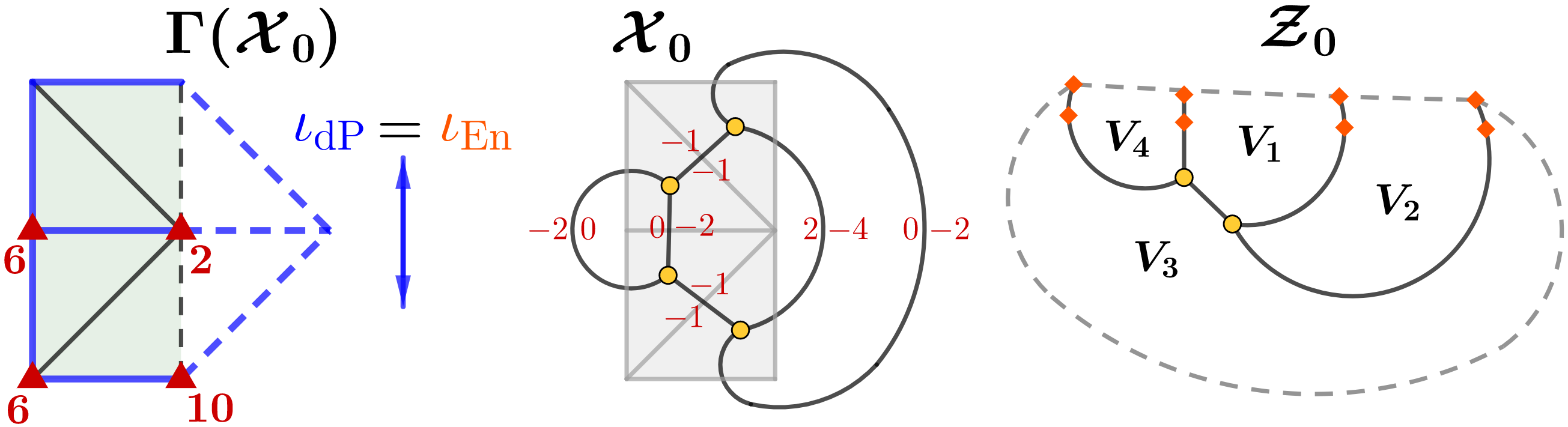} 
  \caption{$B_5(\ell)$ and central fibers for $\ell =(0,0,2,0^7,1,0^3,1,0,0,0,2,0,2,6)$.}
  \label{halfKex2}
\end{figure}

\begin{example}\label{ex2} $B_5(0,0,2,0^7,1,0^3,1,0,0,0,2,0,2,6)$: Consider Cusp 5,
whose folding symmetry is the same as $\iota_\dP$. This value of $\ell$ dictates
that we should put $(2,0)$, $(-1,1)$, $(-1,0)$, $(0,-1)$ end-to-end, then perform
a surgery of size $1$ along the edge $(-1,1)$, to construct $P(\ell)$. The corresponding
sphere $B(\ell)$ is shown in Figure \ref{halfKex2}, together with a Kulikov and 
dlt model, following the conventions of Example \ref{ex1}.

There are four components $(\wV_i,\wD_i)\subset \cX_0$ of the Kulikov model,
all of them preserved by the Enriques involution. Both $(\wV_3,\wD_3)$ and 
$(\wV_4,\wD_4)$ are corner blow-ups of $D_4$ involution pairs. The surface
$(\wV_1,\wD_1)$ is the internal blow-up at points on two opposite
fibers of an anticanonical square in $\bF_2$ and the Enriques involution
interchanges the blow-up points. Finally, $(\wV_2,\wD_2)$ is a corner
blow-up of a $D_8$ involution pair. We have $\wR_1^2=0$, $\wR_2^2=8$,
$\wR_3^2=\wR_4^2=4$.

The components of the stable limit of Enriques surfaces
$(\oV_3,\oD_3+\epsilon \oR_3)\simeq 
(\oV_4,\oD_4+\epsilon \oR_4)$ are denoted $D_{4} : 2 =  \pB_2^- $ 
and $(\oV_2,\oD_2+\epsilon \oR_2)$ is denoted by
$D_{8} : 2 =  \pB_4^- $ in Section \ref{sec:abcde}, where these
surfaces are described further. Only $V_1$ 
is contracted (along a ruling) in the stable limit $\overline{\cZ}_0$.
\end{example}

\begin{figure}[htbp]
  \centering
  \includegraphics[width=335pt]{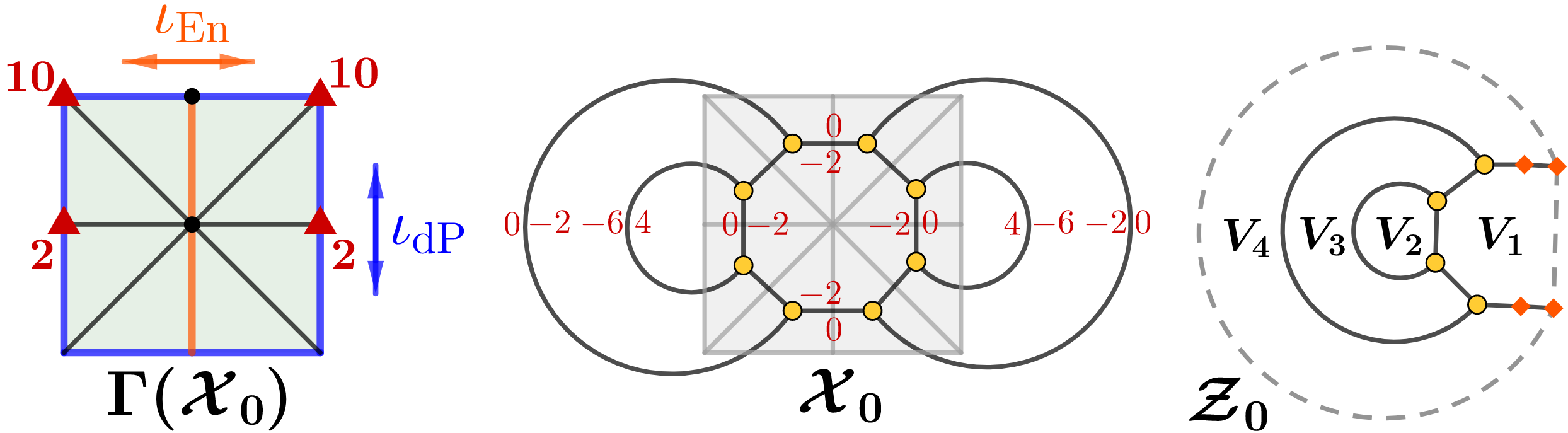} 
  \caption{$B_2(\ell)$ and central fibers for $\ell=(0,1,0^7, 2, 0^7, 1,0)$.}
  \label{halfKex3}
\end{figure}

\begin{example} $B_2(0,1,0^7, 2, 0^7, 1,0)$: Note here that since we are at Cusp 2,
$\ell\in (\bZ_{\geq 0})^{19}$. To form $\oP(\ell)$, we put $(0,1)$,
$(-2,0)$, $(0,-1)$ end-to-end, and then close the base of the polygon by a horizontal
line. No Symington surgeries of positive size are made, so $P(\ell)=\oP(\ell)$.
We glue to get $B(\ell)$ as in Figure \ref{halfKex3}. Even though the central horizontal
segment is fixed by $\iota_{\dP,{\rm IA}}$ it does not form part of the support
of $R_{\rm IA}$, see Section \ref{sec:ias-hyp}.

Then $\wV_2$ and $\wV_3$ are each two disjoint copies of $V_2$ and $V_3$.
$(V_2,D_2) \simeq (\bF_1, \hL+C)$ where $\hL$ is the strict transform of a line
in $\bP^2$ and $C$ is a conic. The surface $(V_3,D_3)$ is, up to two corner blow-ups,
the $D_8$ involution pair as in Example \ref{ex2}, but since $(\wV_3,\wD_3)$ is two disjoint
copies of such, there is no period-theoretic restriction. Finally, $(V_1,D_1)$ and $(V_4,D_4)$
form the image of the 
Enriques equator. They are both quotients of smooth toric surfaces by an involution
$(x,y)\mapsto (x^{-1},-y)$.

Only $V_3$ survives as a component $(\oV_3, \oD_3+\epsilon \oR_3)$. The double locus
$\oD_3\simeq \bP^1\cup \bP^1$ is a banana curve. In the stable model $\overline{\cZ}_0$, 
each $\bP^1\subset \oD_2$ is self-glued by an involution fixing the two nodes in 
$\oD_3$. We have $(\oR_3)^2=8$.
\end{example}

\begin{example} $B_4(0^6, 1, 0^7, 1, 0^5, 2,2) = B_1(0,0,1,0^7,1,0^9,2,2)$. This
is the Type II ray corresponding to the $1$-cusp with label $41$, so it occurs
as a limit of $\ias$ at either Cusp 1 or 4. The dual complex $\Gamma(\cX_0)$
is a segment of length one and the Enriques involution flips the segment 
(this means that, strictly speaking, the Enriques equator is not a 
sub-simplicial complex of $\Gamma(\cX_0)$, as we usually require).
The surface $\cX_0=\wV_1\cup_E \wV_2$ is the union of two copies of the same
$\wD_8$ involution pair, glued with a twist by $2$-torsion
along the elliptic curves $E\in |-K_{\wV_i}|$. 

The quotient $\cZ_0$ is then a non-normal surface with $\cZ_0^\nu\simeq (\wV_1,E)$,
and the normalization map glues $E$ to itself by the $2$-torsion translation.
We have $\overline{\cZ}_0=\cZ_0$. \end{example}

\section{Toroidal, semitoroidal, and KSBA compactifications}
\label{sec:ksba}

\subsection{Toroidal compactification for the Coxeter fans}
\label{sec:toroidal}

In Section~\ref{sec:vinberg} we reviewed the basic results of
\cite{vinberg1972-groups-of-units, vinberg1973some-arithmetic} on
reflection groups acting on hyperbolic lattices. Now we recall
applications of this theory to toroidal compactifications.

Let $\Lambda$ be a hyperbolic lattice of rank $r$ and signature
$(1,r-1)$, and let $\cC$ be the positive cone, one of the two halves
of the set $\{v\in\Lambda_\bR \mid v^2>0\}$.  In the applications to
(semi)toroidal compactifications, instead of the closure $\ocC$ one
operates with the rational closure $\ocC_\bQ$, obtained by adding only
rational vectors at infinity.

Let $W$ be a group acting on $\Lambda$, generated by reflections in a
set of vectors of $\Lambda$. Its fundamental domain is
\begin{displaymath}
  \ch = \{v \in \ocC \mid \alpha_i \cdot v \ge 0 \}
\end{displaymath}
for a set of simple roots $\alpha_i$ with $\alpha_i^2<0$ which is encoded in
a Coxeter diagram $G$. The chamber $\ch$ can be identified with a
polyhedron $P$ in the hyperbolic space $\bP\cC$.  The vectors with
$v^2=0$ are treated as points at infinity of $\bP\cC$.

The subgroup $O^+(\Lambda)$ of the isometry group $O(\Lambda)$ is the
subgroup of index~$2$ that preserves $\cC$. One has
$O^+(\Lambda) = S.W$, where $S$ is a subgroup of symmetries of $P$. 

\begin{definition}
  The \emph{Coxeter semifan} $\fF\cox$ is the semifan with support
  $\ocC_\bQ$ whose maximal cones are chambers of $W$,
  i.e. $\ch$ and its $W$-images.
\end{definition}

It is a fan iff $P$ has finite volume, which is equivalent to $W$
having finite index in $O(\Lambda)$. If this condition is satisfied
then the faces of $\ch$ are of two types:
\begin{enumerate}
\item Type II rays $\bR_{\ge0}v$ generated by vectors with $v^2=0$
  on the boundary of $\ocC_\bQ$. These are in bijection with the
  maximal parabolic subdiagrams of $G$.
\item Type III cones. These are in bijection with elliptic
  subdiagrams of $G$.
\end{enumerate}

\medskip

By \cite[Sec. 3B]{alexeev2022mirror-symmetric} the moduli space
$\fhyp$ admits a toroidal compactification $\ofhyp\ucox$ defined by
the collection of fans
$\fF\cox^\dP = \{\fF_r(18,0,0), \fF_r(18,2,0) \}$, one for each
$0$-cusp. These fans are Coxeter fans for the hyperbolic lattices
$(18,0,0)_1$, $(18,2,0)_1$ for the full reflection groups $W_r$,
generated by reflections in the $(-2)$-roots and in the $(-4)$-roots of
divisibility~$2$. The Coxeter diagrams $G_r(18,0,0)$ and $G_r(18,2,0)$
are given in Figure~\ref{fig-k3-cusps}.

By Lemma~\ref{lem:fen-closed-in-fhyp} there is an immersion
$j\colon \fen\to\fhyp$ whose image is a Noether-Lefschetz locus in
$\fhyp$. The normalization of the closure of $j(\fen)$
in $\ofhyp\ucox$ is then a toroidal
compactification $\ofen\ucox$ for the 
fans $\{\fF^k\cox\}_{k=1,2,3,4,5}$, one for each of $0$-cusp of $\fen$.
The fans
$\fF^k\cox$ are the intersections of the above fans $\fF_r$ in the
lattices $\oT_\dP=(18,0,0)_1$ and $(18,2,0)_1$ with the sublattices
$\oT_\enr = (10,10,0)_1$ and $(10,8,0)_1$, as in
Section~\ref{sec:cusps}. By Lemma~\ref{lem:all-diagrams-are-folded}
and Corollary~\ref{cor:folded-diagrams} these five fans are the
Coxeter fans for the folded Coxeter diagrams $G_r^k$ of
Figures~\ref{fig-cusp2}, \ref{fig-cusps1345}.
By \cite{sterk1991compactifications-enriques1} the
induced groups acting on $e^\perp/e$ are of 
the form $\Gamma_k=\Aut(G_r^k)\ltimes W(G_r^k)$



\begin{lemma}\label{lem:rays-toroidal-comp}
  For $k=1,2,3,4, 5$, the numbers of Type II + Type III rays in
  $\fF^k\cox/\Gamma_k$ are $4+4$, $2+8$, $3+15$, $4+12$,
  $5+17$.
  The toroidal compactification $\ofen\ucox$ has $9+56=65$ Type II +
  Type III divisors. 
\end{lemma}
\begin{proof}
  Direct enumeration of maximal parabolic and elliptic subdiagram of
  rank~$9$ in the Coxeter diagrams $G^k_r$.  Type II divisors
  correspond to curves in $\ofen\ubb$ passing through several
  $0$-cusps, so each of them corresponds to several rays in
  $\fF^k/\Gamma_k$.
\end{proof}

\subsection{Semitoroidal compactification for the generalized Coxeter fans}
\label{sec:semitoroidal}

Looijenga's semitoric, or semitoroidal compactifications of Type IV
domains \cite{looijenga2003compactifications-defined2} generalize
toroidal compactifications in several ways. By
\cite[Thm. 7.18]{alexeev2023compact} these are the normal
compactifications dominating the Baily-Borel compactification and
dominated by some toroidal compactification. They are defined by
collections of compatible semifans, one for each Baily-Borel
$0$-cusp. The data for the $1$-cusps is then uniquely determined. The
cones in semifans have rational generators but, unlike in fans, there
could be infinitely many generators, and the stabilizer groups of the
Type III cones may be infinite.

The generalized Coxeter semifans were defined in
\cite[Section~4D]{alexeev2023stable-pair} using the Wythoff
construction \cite{coxeter1935wythoffs-construction}, as follows. As 
above, let $W$ be a reflection group with a fundamental chamber $\ch$
and $G=\{\alpha_i\}$ be the corresponding Coxeter diagram. Divide the
vertices of $G$ into two complementary sets $I\sqcup J$ of
\emph{relevant} and \emph{irrelevant} roots. Let $W\irr$ be the
subgroup of $W$ generated by the irrelevant roots and let
\begin{math}
  \ch\gend = \cup_{h\in W\irr} h.\ch
\end{math}
The maximal dimensional cones in the semifan $\fF\gend$ are the
chamber $\ch\gend$ and its images under $W$. Another way to describe
$\fF\gend$ is that it is the coarsening of the Coxeter fan $\fF\cox$
obtained by removing the faces of the form
$\cap\alpha_j^\perp \cap\ch$ in which $\{\alpha_j,\ j\in J'\subset J\}$
is a collection of irrelevant roots.


In \cite[Sec. 9A]{alexeev2022mirror-symmetric} the authors defined a
specific semitoroidal compactification of the moduli space $\fhyp$ by
the collection $\fF\ram = \{ \fF\ram(18,0,0), \fF\ram(18,2,0) \}$ of
two semifans.  (Here, \emph{ram} stands for the ramification divisor.)
These are the generalized Coxeter semifans for the Coxeter diagrams of
Figure~~\ref{fig-k3-cusps} in which the irrelevant roots are those
that do not lie on the boundary of the square, resp. the triangle,
numbered respectively $0$--$15$ and $2$--$18$.
The main theorem of \cite{alexeev2022mirror-symmetric} for the moduli
space $\fhyp$ says that the normalization of the KSBA moduli
compactification $\ofhyp$ for the pairs $(X,\epsilon R)$ is this
semitoroidal compactification.

\begin{definition}
  The collection of semifans $\fF = \{\fF^k\}_{k=1,2,3,4,5}$,
  one for each $0$-cusp of $\fen$ is defined by intersecting the
  semifans $\fF\ram(\oT_\dP)$ for $\oT_\dP=(18,0,0)_1$,
  $(18,2,0)_1$ with the subspace $\oT_\enr=(10,10,0)_1$ and
  $(10,8,0)_1$ as in Section~\ref{sec:cusps}. 
\end{definition}

\begin{definition}\label{def:sterk-irrelevant-roots}
  In each of the folded Coxeter diagrams of Figures~\ref{fig-cusp2},
  \ref{fig-cusps1345}, call a root \emph{irrelevant} if it is obtained
  by folding of an irrelevant root in Figure~\ref{fig-k3-cusps},
  i.e. a root which does not lie on the boundary of the square,
  resp. the triangle.
\end{definition}

\begin{lemma}
  The semifans $\{\fF^k\}_{k=1,2,3,4,5}$ are the generalized Coxeter
  fans for the folded Coxeter diagrams of Figures~\ref{fig-cusp2},
  \ref{fig-cusps1345} with the irrelevant roots of
  Definition~\ref{def:sterk-irrelevant-roots}.
\end{lemma}
\begin{proof}
  By Lemma~\ref{lem:all-diagrams-are-folded}, for a root $\alpha$ of
  $\oT_\dP$, if $\alpha^\perp$ intersects the interior of the positive
  cone $\cC$ in $\oT_\enr$ then
  $\alpha^\perp\cap\cC = \alpha_J^\perp \cap\cC$ for the folded
  root $\alpha_J$. By definition, irrelevant roots fold to irrelevant
  roots. Thus, the fans $\fF^k$ are obtained from the Coxeter fans
  $\fF^k\cox$ by removing the faces of the form
  $\cap\alpha_j^\perp \cap\ch$ in which
  $\{\alpha_j,\ j\in J'\subset J\}$ is a collection of irrelevant
  folded roots. So these are the generalized Coxeter semifans as stated.
\end{proof}

\begin{lemma}\label{lem:fans24}
  The semifans $\fF^k$ are fans for $k=2,4$ and are not fans for $k=1,3,5$.
\end{lemma}
\begin{proof}
  Indeed, for $k=2$, resp. $k=4$, the irrelevant subgroup $W\irr=S_2$,
  resp. $S_2^2$, is finite.
  For the other $0$-cusps the
  groups $W\irr$ are infinite, the cones $\ch_k$ have infinitely many
  generators, and the corresponding polyhedra have infinite volumes.
\end{proof}

\begin{lemma}\label{lem:semitoroidal}
  The semitoroidal compactification of $\fen$ defined by the collection of
  semifans $\{\fF^k\}_{k=1,2,3,4,5}$ is toroidal over the $0$-cusps
  $2$ and $4$ and the $1$-cusps which are adjacent to them, and
  over $1$-cusp $35$. It is strictly semitoroidal over the remaining cusps.
\end{lemma}
\begin{proof}
  By Lemma~\ref{lem:fans24}, this semitoroidal compactification
  is toroidal over the cusps 2 and
  4 and so also over the $1$-cusps adjacent to it. In general, the
  definition of the generalized Coxeter semifan above implies that the
  semitoroidal compactification is toroidal over a $1$-cusp exactly
  when the corresponding maximal parabolic diagram does not have a
  connected component consisting entirely of irrelevant
  vertices. Examining Figure~\ref{fig-pars} shows that in addition to
  the $1$-cusps adjacent to the $0$-cusps 2 and 4 there is just one
  more, for the $1$-cusp $35$. This completes the proof.   
\end{proof}

\begin{lemma}\label{lem:rays-semitoroidal-comp}
  For $k=1,2,3,4, 5$, the numbers of Type II + Type III
  divisors at the cusps of the semitoroidal compactification
  $\ofen^\fF$ are $2+0$, $2+7$, $2+7$, $4+7$, $3+0$, for a total of
  $6+21=27$ divisors.
\end{lemma}
\begin{proof}
  This is obtained by removing from the list of subgraphs in
  Lemma~\ref{lem:rays-toroidal-comp} the graphs containing a connected
  component consisting of irrelevant vertices.
\end{proof}

\subsection{The main theorem} 
\label{sec:main-thm}

By Section~\ref{sec:ksba-recall} there exists a compact moduli space
$\ofen$ whose points correspond to the pairs $(Z,\epsilon R_Z)$ of
Enriques surfaces with numerical polarization of degree~$2$ and their
KSBA stable limits, for any $0<\epsilon\ll 1$. This is the
closure of $\fen$ in the KSBA moduli space of stable pairs. 

\begin{theorem}\label{thm:main}
  The normalization of $\ofen$ is semitoroidal for the collection of
  semifans $\{\fF^k\}_{k=1,2,3,4,5}$ of
  Section~\ref{sec:semitoroidal}. It is toroidal over the $0$-cusps
  $2$ and $4$, the $1$-cusps which are adjacent to them, and
  over $1$-cusp $35$. It is strictly semitoroidal over the remaining cusps.
\end{theorem}
\begin{proof}

  The main theorem of \cite{alexeev2023compact} is that the
  normalization of the KSBA compactification of K3 pairs
  $(X,\epsilon R)$ for a \emph{recognizable divisor} $R$ is
  semitoroidal and by \cite{alexeev2021nonsymplectic} the ramification
  divisor is recognizable. The main theorem of
  \cite{alexeev2022mirror-symmetric} for $\fhyp$ is that this semifan
  is the ramification semifan $\fF\ram$ of
  Section~\ref{sec:semitoroidal}.  
  
  Consider the universal family $(\cX,\epsilon \cR)\to \ofhyp$ of
  KSBA-stable pairs over the compactified moduli stack. Denote the
  closure of the image of $\fen$ in $\ofhyp$ by~$B$. Then, the
  pullback of the universal family $(\cX_B,\epsilon \cR_B)\to B$ is a
  family whose general fiber is a pair $(X,\epsilon R)$ of an Enriques
  K3 surface with the ramification divisor $R$ of the del Pezzo
  involution. By uniqueness of KSBA-stable limits, the Enriques
  involution on the general fiber extends to an involution on the
  universal family $(\cX_B,\epsilon \cR_B)$. Taking the quotient gives
  a family $(\cZ,\epsilon\cR_{\cZ})\to B$ over a compact base,
  extending the universal family of Enriques surfaces
  $(Z,\epsilon R_Z)$ with divisor.  
  
  By Lemma~\ref{lem:fen-closed-in-fhyp}, the normalization $B^\nu$ of $B$ is a
  compactification of $\fen$ admitting a universal family of pairs. So
  we have a classifying morphism $B^\nu\to \ofen$.  Furthermore,
  $B^\nu$ is simply the semitoroidal compactification of the
  Noether-Lefschetz locus $B$, induced by the semifan $\fF\ram$ which
  gives the normalization $\ofhyp^\nu$.
  This gives a family of KSBA stable pairs over the induced
  compactification $\ofen$, whose normalization by
  Section~\ref{sec:semitoroidal} is the compactification
  $\ofen^\fF$ for the collection of semifans $\fF =
  \{\fF^k\}_{k=1,2,3,4,5}$. 

To prove the first statement, it remains to show that $B^\nu\to \ofen$
is a finite map. 
Equivalently, we do not lose moduli when we quotient a stable K3 pair $(X,\epsilon R)$ by $\iota_\enr$.
By \cite[Thm. 7.18]{alexeev2023compact}, the normalization of $\ofen$ is given by some semifan coarsening
$\fF^k$ and so it suffices to prove that the maximal cones of this semifan are the same
as the maximal cones $\fF^k$.

The explicit description of Kulikov and stable models from
Proposition \ref{half-kulikov-prop} and Corollary \ref{hk-cor} imply the following fact:
a degeneration of $(Z,\epsilon R_Z)$ has a maximal
number of double curves if and only if $(X,\epsilon R)$ does. 
But if the normalization of $\ofen$ were 
given by any strict coarsening of $\{\fF^k\}$, 
there would be some codimension one cone of some $\fF^k$ 
that parameterized a $1$-dimensional family of non-maximal pairs $(X,\epsilon R)$,
whose Enriques quotients $(Z,\epsilon R_Z)$ had the maximal number of double
curves. This is impossible, so we conclude the first statement.

The last statement follows by Lemma~\ref{lem:semitoroidal}.
\end{proof}

\section{ABCDE surfaces}
\label{sec:abcde}

The paper \cite{alexeev17ade-surfaces} classified the surfaces which may
appear as irreducible components of KSBA stable degenerations of K3
surfaces with a non-symplectic involution $(X,\iota)$ for the pairs
$(X,\epsilon R)$, where $R$ is a component of genus $g\ge2$ of the ramification divisor of the double
cover $X\to X/\iota$. In particular, the irreducible components of
stable pairs $(X,\epsilon R)$ in \cite{alexeev2023stable-pair,
  alexeev2022mirror-symmetric} are all of these types.  
 The surfaces appearing in Type III degenerations naturally
correspond to Dynkin diagrams $A_n$, $D_n$, $E_n$, and those appearing
in Type II degenerations to the affine $\wA_n$, $\wD_n$, $\wE_n$
diagrams. Both come with decorations addressing parity and some extra
data, as in Section~\ref{sec:decorations} below.

On the other hand, it is well known that the non simply laced Dynkin
diagrams of BCFGH types can be naturally described by ``folding''
ADE diagrams by automorphisms.
After recalling the ADE
surfaces relevant to this paper, we define new B and C type
surfaces obtained from them as quotients by involutions.

\smallskip

The surfaces in \cite{alexeev17ade-surfaces} come in pairs
$\pi\colon (X,D+\epsilon R) \to (Y,C+\fraceps B)$, fully
analogous to Diagram~\eqref{eq:basic} in the introduction. Here:
\begin{enumerate}
\item $(Y,C)$ is a log del Pezzo pair of index $2$ with reduced
  boundary $C$ and a nonempty nonklt locus.
  The divisor $B\in |-2(K_Y+C)|$ is ample Cartier,
  and the pair $(Y, C+\fraceps B)$ is KSBA stable; in
  particular it is log canonical.
\item $\pi\colon X\to Y$ is the index-$1$ cover for
  $K_Y+C$. Explicitly, $X = \Spec \cA$, where $\cA = \cO_Y\oplus
  \cO_Y(K_Y+C)$ is an $\cO_Y$-algebra with the multiplication defined
  by an equation of $B$. One has $K_X+D\sim 0$, $R = \frac12\pi^*(B)$
  is ample, the pair $(X,D+\epsilon R)$ is KSBA stable and it has a
  nonempty nonklt locus.
\end{enumerate}
By the Riemann-Hurwitz formula, one has
\begin{equation}\label{eq:hurwitz}
  K_X + D + \epsilon R = \pi^*\left(K_Y+C + \fraceps B\right)
\end{equation}

By \cite[Lemma 2.3]{alexeev17ade-surfaces}, the pairs
$(Y, C+\fraceps B)$ and $(X,D+\epsilon R)$ are in a
one-to-one correspondence. To distinguish them we will call the former
\emph{del Pezzo ADE surfaces} and the latter \emph{anticanonical
  ADE surfaces}. 

\subsection{Type III ADE surfaces}
\label{sec:ade-type3}

The only ADE surfaces needed in this paper are the ones that appear
on the boundary of the KSBA compactification $\ofhyp$. They are described in
detail in the last section of \cite{alexeev2022mirror-symmetric}. Most
of them are easy: they are hypersurfaces in projective
toric varieties in a way very similar to the construction 
in Section~\ref{sec:main-diagram}.

\begin{figure}[htbp]
  \includegraphics[width=360pt]{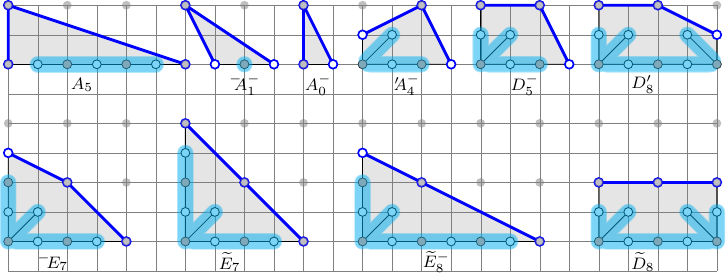}
  \smallskip
  \caption{Some ADE surfaces}
  \label{fig-ade}
\end{figure}

As an example, consider one of the lattice polytopes $Q$ in Figure~\ref{fig-ade} with
an ADE Dynkin diagrams fitted into it. The polytopes are in $\bZ^2$,
and the gray dots indicate the sublattice $2\bZ^2$. The Type III
polytopes for the ordinary elliptic ADE diagrams have a distinguished
vertex with two bold blue sides emanating from it. In the Type II
polytopes for the extended $\wD\wE$ diagram there is a distinguished point
in the interior of the bold blue segment. Together with the ends of
this segment, it makes three special vertices.

By the standard construction, one associates to $Q$ a toric
variety $V_Q$ with an ample line bundle $L_Q$.
Let us define a section of $L_Q$ as the
following sum of monomials in $Q\cap\bZ^2$. For Type III, each of the
three special vertices above gets coefficient $1$. The coefficients of
the vertices in the highlighted Dynkin diagrams are arbitrary numbers
$a_i\in \bC$. The other coefficients are zero. Concretely:

\begin{enumerate}
\item[($A_5$)] $f= (1 + y^2 + x^6) + \sum_{i=1}^5 a_ix^i$
\item[($\mA_1^-$)] $f=(x + y^2 + x^3) + a_1x^2$
\item[($A_0^-$)] $f=1 + y^2 + x$
\item[($\pA_4^-$)] $f=(y + x^2y^2 + x^3) + a_1xy+a_2 + a_3x+a_4x^2$
\item[($D_5^-$)] $f=(y^2 + x^2y^2 + x^3) + a_1xy + a_2y + a_3+a_4x+a_5x^2$
\item[($D'_8$)] $f=(y^2 + x^2y^2 + x^4y) + a_1xy + a_2y + \sum_{i=0}^4
  a_{i+3}x^i + a_8x^3y$
\item[($\mE_7$)] $f=(y^3 + x^2y^2 + x^4) + a_1xy + a_2y^2 + a_3y+ a_4 + a_5x +
  a_6x^2 + a_7x^3$
\end{enumerate}

The corresponding del Pezzo ADE surface is the toric variety $Y=V_Q$
together with the boundary $C=C_1+C_2$ for the two blue sides, and the
divisor $B$ is $(f)$. Combinatorially the condition $B \sim -2(K_Y+C)$
is equivalent to the condition that the other sides of $Q$ have
lattice distance~$2$ from the distinguished point.

We now put these polytopes in $\bZ^2 \times 0 \subset \bZ^3$.
Let $p_0$ be the position of the distinguished vertex, we then add another vertex at the point $p_0 + (0,0,2)$ to which we associate the monomial $z^2$.
Let $P$ be the pyramid with the
apex at the new vertex and with base $Q$. Associated to it we have a
$3$-dimensional polarized toric variety $(V_P,L_P)$ and a section
$z^2 + f$ of $L_P$. An anticanonical ADE surface $X$ is the zero set
of this section, so it is a hypersurface in $V_P$. It comes with a del
Pezzo involution
\begin{math}
  \idp\colon (x,y,z) \to (x,y,-z),
\end{math}
the quotient map is $\pi\colon X\to X/\idp = Y$, the boundary is
$D=\pi\inv(C)$, and the ramification divisor is $R=\pi\inv(B)$. 

Varying the free coefficients $a_i$ we get a family over $\bC^n$,
where $n$ is the rank of the Dynkin diagram. This $\bC^n$ is the
quotient of the algebraic torus $\Hom(\Lambda^*, \bC^*)\simeq (\bC^*)^n$ by the
Weyl group $W(\Lambda)$ for the ADE root lattice $\Lambda$ with the
weight lattice $\Lambda^*$. So it naturally comes from a family over a torus.

\subsection{Decorations}
\label{sec:decorations}

Because of $z^2$ and the double cover, the
vertices in $2\bZ^2\subset\bZ^2$ are clearly distinguished; let's call them even.
When the end of a bold blue edge is even, this edge is {\it long}, of lattice
length~$2$. Then we use no decorations. When this end is odd, the edge
is {\it short}, of lattice length~$1$. To indicate that it is short, we use a minus
or a prime sign. We also use primes to distinguish shapes where the
long leg pokes into the interior of~$Q$.

The classification of del Pezzo ADE surfaces $(Y,C+\fraceps B)$ in
\cite{alexeev17ade-surfaces} is divided into \emph{pure} and
\emph{primed} shapes. The surfaces for the pure shapes are all
toric. The surfaces for \emph{some} of the primed shapes are toric,
but not in general. They are obtained from pure shapes by making a
blow up at a point $x\in D\cap R$ on $X$, resp. a weighted blowup at a
point $y\in C\cap B$ in $Y$.
For each side $D_1, D_2$, the set $D_i\cap R$ is
either a single point (if the side is short) or two points (if it is
long). E.g. priming $A_n$ on a long side once gives $A'_n$ and twice
gives $A''_n$. Priming $A^-_n$ on a short side is denoted by $A^+_n$.

The blow up disconnects $D_i$ from $R$ at that point. If all points in
$D_i\cap R$ are blown up, for the strict preimages we have $D'_i
\cdot R'=0$. In this case the linear system $|mR'|$ for $m\gg0$
contracts $D'_i$ and the corresponding ADE surface has fewer
boundary components. Thus, the surfaces e.g.~for the shapes $A''_n$
and $A^+$ have only one boundary component, and for the shapes
$\ppA_n''$, $\ppA_n^+$, $\plA_n^+$ have zero boundary components. 

\subsection{Type II ADE surfaces}
\label{sec:ade-type2}

The construction for the Type II polytopes is similar. The ends of the
bold blue edge have coefficients $1$ in $f$, and the distinguished interior
point has coefficient $\lambda\in\bC$. For clarity, in
Figure~\ref{fig-ade} one has
\begin{enumerate}
\item[($\wE_7$)] $f= (y^4 + \lambda x^2y^2 + x^4) +
  a_1 xy + a_2y^3+a_3y^2 + a_4y+ a_5 + a_6x + a_7x^2 + a_8x^3$
\item [($\wE_8^-$)] $f = (y^3+\lambda x^2y^2 + x^6) +
  a_1xy + a_2y^2+a_3y + \sum_{i=0}^5 a_{4+i}x^i$
\item [($\wD_8^-$)] $f = (y^2+\lambda x^2y^2 + x^4y^2) +
  a_1xy + a_2x^3y + a_4x^4y + \sum_{i=0}^4 a_{i+5}x^i$
\end{enumerate}
The coefficients for the nodes of the extended Dynkin diagram are
arbitrary numbers $a_i\in\bC$, not all of them zero, and they are now
treated as homogeneous coordinates of weight equal to the lattice
distance from the bold blue edge. Thus, for a fixed $\lambda$ one gets
a family of sections $z^2+f$ of $L_P$ and a family of anticanonical
KSBA stable pairs $(X,D+\epsilon R)$ parameterized by a weighted
projective space. For $\wE_7$ it is $\bP(1^2,2^3,3^2,4)$, for $\wE_8$
it is $\bP(1,2^2,3^2,4^2,5,6)$, and for $\wD_{2n}$ it is
$\bP(1^4, 2^{2n-3})$. The restriction of $z^2+f$ to the divisor
corresponding to the bold blue line gives a double cover of $\bP^1$
which is an elliptic curve. Varying $\lambda$ we get a family of
$\wA\wD\wE$ surfaces parameterized by a bundle of weighted projective
spaces over the $j$-line.  This is the same bundle of weighted projective
spaces that appeared in \cite{looijenga1976root,
  pinkham1977simple-ellitic, looijenga1978semi-universal}.

\smallskip

The $\wA_{2n-1}$ surfaces do not directly correspond to polytopes.
These
surfaces are double covers of cones over elliptic curves branched in a
bisection. The easiest description, closest to toric is to use the Tate
curve. For each $i\in \bZ_{2n}$ define the theta function $\theta_i$
as the formal power series
\begin{displaymath}
  \theta_i = \sum_{k\equiv i\, ({\rm mod\ }{2n})} q^{k(k-1)/2} x^k
\end{displaymath}
It converges for any $q\in \bC^*$ with $|q|<1$ and defines a section
of $L^2$, where $L$ is an ample line bundle of degree $n$ on the
elliptic curve $E_q = \bC^*/q^{\bZ}$. For any $c_i\in \bC$ not all
zero, $g(x) = \sum c_i\theta_i$ is a nonzero section of $L^2$ and
$f(x,y) = y^2 + g(x)$ is a section on the square of the tautological
line bundle on $\wY = \bP(\cO\oplus L)$. It also defines a section of
a line bundle on the surface $Y$ that is a cone over $E$, obtained by
contracting an exceptional section of $\wY$. Finally,
$z^2 + f(x,y)$ defines a double cover $X\to Y$ and the covering
involution $\idp$ is
\begin{math}
  (x,y,z)\to (x,y,-z).
\end{math}

\subsection{Anticanonical ADE surfaces with two commuting involutions}
\label{sec:ade-two-involutions}

Let $(X,D)$ be a log canonical pair with $K_X+D\sim 0$. Pick a
generator $\omega$ of the space $H^0(\cO(K_X+D))=\bC$. Just as for K3 surfaces,
an involution $\iota$ is called symplectic if $\iota^*\omega=\omega$
and nonsymplectic if $\iota^*\omega=-\omega$. By looking at a local
equation $dx\wedge\frac{dy}y$ of $\omega$ near the boundary, it is
easy to see that for a nonsymplectic involution the quotient map
$X\to X/\iota$ is not ramified along any irreducible component of $D$.

\begin{proposition}\label{prop:two-covers}
  Let $\pi\colon (X,D+\epsilon R)\to (Y,C+\fraceps B)$ be the
  anticanonical and del Pezzo ADE surfaces, and $\idp$ be the
  anticanonical involution such that $Y=X/\idp$. Suppose that
  $\ien\colon X\to X$ is another nonsymplectic involution commuting
  with $\idp$ such that $\idp$ and the induced involution
  $\tau\colon Y\to Y$ both have finite fixed sets. Then
  there exists a diagram of log canonical pairs 
  \begin{equation}\label{eq:basic2}
    \begin{tikzcd}
      (X,D+\epsilon R) \arrow{d}{\pi}
      \arrow[swap]{r}{\psi}
      \arrow[bend left=15,swap]{rr}{\psi'}
      & (Z,D_Z + \epsilon R_Z) \arrow{d}{\rho} 
      & (Z',D_{Z'} + \epsilon R_{Z'}) \arrow{dl}{\rho'} 
      \\
      (Y,C +\fraceps B) \arrow[swap]{r}{\varphi} 
      & (W,C_W + \fraceps B_W)
    \end{tikzcd}    
  \end{equation}
  in which
  \begin{enumerate}
  \item $\psi\colon X\to Z$ is the quotient by $\ien$ and
    $\psi'\colon X\to Z'$ is the quotient by the symplectic involution
    $\inik=\idp\circ\ien$.
  \item $R_Z=\frac12\rho^*(B_W)$ and
    $R_{Z'}=\frac12\rho'{}^*(B_W)$ are reduced divisors and one has
    $R = \psi^*(R_Z) = \psi'{}^*(R_{Z'})$.
  \item $D_Z=\rho^*(C_W)$ and $D_{Z'}=\rho'{}^*(C_W)$ are
    reduced divisors and one has $D = \psi^*(D_Z) = \psi'{}^*(D_{Z'})$.
  \item $(W,C_W+\fraceps B_W)$ is a del Pezzo ADE surface, and
    $(Z',D_{Z'}+\epsilon R_{Z'})$ is an anticanonical ADE surface
    which is its index-$1$ cover.
  \item     For any $\epsilon$ one has
    \begin{eqnarray*}
      &K_X + D + \epsilon R 
        = \psi^*(K_Z + D_Z + \epsilon R_Z) 
        = \psi'{}^*(K_{Z'} + D_{Z'} + \epsilon R_{Z'})  \\
      &K_Z + D_Z + \epsilon R_Z = \rho^*(K_W + C_W + \fraceps B_W)\\
      &K_{Z'} + D_{Z'} + \epsilon R_{Z'} = \rho'{}^*(K_W + C_W + \fraceps B_W)
    \end{eqnarray*}

  \item $2(K_Z+D_Z)\sim 0$ but $K_Z+D_Z\not\sim 0$.
  \item $\rho'$ is branched in $C_W$ and a finite subset of $\Branch(\varphi)$.
  \item $\rho$ is branched in $C_W$, a finite subset of
    $\Branch(\varphi)$, and the irreducible components of $C_W$ which
    are part of the branch locus of $Y\to W$.
  \item For any $p\in\Branch(\varphi)\setminus C_W$, one has
    $p\in \Branch(\rho)$ iff $p\notin \Branch(\rho')$.
  \end{enumerate}

\end{proposition}

\begin{proof}
  (1)--(3) are straightforward. Since $\inik$ is symplectic,
  $\cO(K_X+D)\simeq \cO_X$, and taking the $\inik$-invariants gives
  $\cO(K_{Z'}+D_{Z'})\simeq \cO_{Z'}$.  (4) and (7) follow from this by
  \cite[Lem.~2.3]{alexeev17ade-surfaces}. (5) holds by the Riemann-Hurwitz formula.

  The following argument applies to both $T=Z$ or $Z'$, $\iota=\ien$
  or $\inik$. The image of $K_X+D$ under the norm map between Cartier
  divisors is $2(K_T+D_T)$, thus $2(K_T+D_T)\sim 0$. One has
  $\cO_X = \cO_T \oplus \cA$ for a divisorial sheaf $\cA$ on $T$. The
  sheaves $\cO_T$, $\cA$ are the $(\pm 1)$-eigenspaces for the action of
  $\iota^*$ on $\cO_T$. Also, $\cA\not\simeq\cO_T$ since $X$ is
  connected. Since $\cO_X(K_X + D)=\cO_X$, we get
  $\cO_{Z'}(K_{Z'} + D_{Z'}) = \cO_{Z'}$ and
  $\cO_Z(K_Z+D_Z)= \cA\not\simeq \cO_Z$. This proves (6).

  For (8) and (9), consider $p\in\Branch(\varphi)$, $p\notin C_W$ and
  let $q=\varphi\inv(p)$. Then the preimage $\rho\inv(q)$ consists of
  two points $r_1,r_2$ interchanged by $\idp$. One has $p\in
  \Branch(\rho)$ iff $\ien(r_1)=\ien(r_2)$ iff
  $\inik(r_1)\neq\inik(r_2)$ iff $p\not\in\Branch(\rho')$. 
\end{proof}

One could say that the ADE surfaces $Z'\to W$ are obtained by
folding the ADE surfaces $X\to Y$ by the symplectic involution
$\inik$, and $Z\to W$ are obtained from $X\to Y$ by folding by the
nonsymplectic involution $\ien$.  The index-$1$ cover
$\rho'\colon Z'\to W$ and the index~$2$ cover
$\rho\colon Z\to W$ are dual in a similar way to
Remark~\ref{rem:two-dual-covers}.

\smallskip

In the next two sections we find several examples of such foldings,
naturally corresponding to foldings of ADE Dynkin diagrams,
producing some non simply laced Dynkin diagrams of $B$ and $C$
types. The smaller versions of these examples can be found in
Figures~\ref{fig-par2}, \ref{fig-pars}, \ref{fig-maxconn-ell2},
\ref{fig-maxconn-ell}. The involutions appearing at Cusp 5 are
described in Section~\ref{sec:toric-involutions}, and those appearing
at other cusps in Section~\ref{sec:polytope-involutions} below.

For the parabolic diagrams, we
follow Vinberg's conventions \cite{vinberg1972-groups-of-units}: the
$\wD_n$ diagram has two forks, $\wB_n$ has one, and $\wC_n$ is a chain
without forks.

\subsection{Quotients by $\pm1$ in the torus}
\label{sec:toric-involutions}

We first consider the Enriques involution on an ADE anticanonical
surface $X = \{z^2 + f(x,y)=0\}$ that is given by the same formula
$\ien\colon (x,y,z)\to (-x,-y,-z)$ as in
Section~\ref{sec:main-diagram}. The pairs $(X,D)$ of this type appear very
naturally in Horikawa's construction,  when $\bP^1\times\bP^1$
degenerates to a stable surface $Y=\cup (Y_i,D_i)$. As in
Section~\ref{sec:main-diagram}, let
$\bZ^2_\even = \{(a,b) \mid a+b\in 2\bZ \}$. We have
$2\bZ^2 \subsetneq \bZ^2_\even \subsetneq \bZ^2$.

Let $Q$ be one of
the ADE polytopes of Sections~\ref{sec:ade-type3},
\ref{sec:ade-type2} above and assume that the monomials of $f(x,y)$
lie in $\bZ^2_\even$. This means that the bold blue edges are long,
the Dynkin diagram ends in odd vertices on the boundary, and there are
no minus or prime decorations.

We then have four surfaces as in Diagram~\eqref{eq:basic2}. Our
notation for the covers will be $\alpha:2={}_2\beta\subset\gamma$, where
$\alpha$ is the ADE type of $X\to Y$, $\gamma$ is the ADE type of
$Z'\to W$, and ${}_2\beta$ is the ABCDE type of the index-$2$ cover
$Z\to W$; or simply $\alpha:2={}_2\beta$ if $\beta=\gamma$.

\begin{lemma}
  There exist diagrams of the following types:
  \begin{enumerate}
  \item $A_{4n-1} : 2 = {}_2A_{2n-1}$ and $A_{4n+1} : 2 = \mAtwo_{2n-1}^-$
  \item $A_{4n+1} : 2 = {}_2A^-_{2n}$ and $A_{4n+1} : 2 = {}_2A_{2n}^-$
  \item $D_{4n} : 2 =  \pB_{2n}^- \subset \pD^-_{2n+1}$ and
    $D_{4n+2} : 2 =  \pB_{2n+1}^- \subset \pD_{2n+2}$
  \item $\wD_{4n} : 2 =  {}_2\wC_{2n} \subset \wD_{2n+2}$
  \item $\wE_7 : 2 =  {}_2\wB_3 \subset \wD_4$
  \item $\wA_{4n-1} : 2 =  {}_2\wA_{2n-1}$
  \end{enumerate}
\end{lemma}
\begin{proof}
  The conditions of Proposition~\ref{prop:two-covers} are immediate to
  check.  Let $Q$ be the polytope corresponding to the toric surface
  $Y$.  The surface $(W,C_W)$ is toric for the same polytope $Q$ and
  the lattice $\bZ^2_\even$, so its ADE type is easy to find. In
  case (1) we get $A_n$ and $\mA_n$. In case (2) it is the $\pD$ type,
  as can be seen in \cite[Fig.~9]{alexeev17ade-surfaces}. The other
  three cases are checked similarly, with the aid of
  \cite[Tables 2, 3]{alexeev17ade-surfaces}. 
\end{proof}

Thus, we describe the index-$2$ anticanonical surface $(Z,D_Z)$ in
two ways:
\begin{enumerate}
\item as the quotient of $(X,D)$ by $\ien$, and
\item as an index-$2$ cover of a del Pezzo ADE surface
  $(W,C_W)$. 
\end{enumerate}
The first way presents $Z$ as a hypersurface in the toric variety
$V_P$ for the same polytope $P$ as $X$ but for a new lattice
$\bZ^3_\even = \{(a,b,c) \mid a+b+c \in 2\bZ \}$.

The branch locus of $\varphi\colon Y\to W$ consists of:
\begin{enumerate}
\item The torus-fixed points corresponding to the vertices of $Q$. Let
  us denote the distinguished vertex of $Q$ by $c$ and the adjacent corners of $Q$ 
  by $v_i$.
\item The boundary divisors corresponding to the sides $(c,v_i)$ of
  $Q$ which are long with respect to the lattice $\bZ^2_\even$.
  We number them by $i$ with $i\equiv 0\pmod 4$.
\end{enumerate}
By Proposition~\ref{prop:two-covers}, $\rho\colon Z\to W$ is
branched at the point for the distinguished vertex $c$ and 
in the boundary divisors for the sides $C_i = (c,v_i)$ with
$i\equiv 0\pmod 4$. There are two $A_1$ singularities over each point
in $C_i\cap B_W$. Also, $\rho$ is unramified over the points for $v_i$
with $i\equiv 2\pmod 4$.

\begin{example}
  In Figure~\ref{fig-square-triangle}, consider the square $Q$ with the
  vertices $(0,0)$, $(2,0)$, $(0,2)$, $(2,2)$.  Let $Y$ be the
  corresponding toric variety, and and let $C_1$, $C_2$ be the
  boundary curves for the two sides passing through the central point
  $(2,2)$.  Then $Y=\bP^1\times\bP^1$, $C=C_1+C_2$ are the fibers of
  two $\bP^1$-fibrations, and $B\in |\cO(2,2)|$.  Both $Y$ and
  $W=Y/\tau$ are toric varieties corresponding to the square $Q$
  but for different lattices: $\bZ^2$ and $\bZ^2_\even$, as in
  Section~\ref{sec:main-diagram}.  $W$ has four $A_1$ singularities at the
  torus-fixed points corresponding to the corners of $Q$.

  The surface $(Y, C + \fraceps B)$ is a del Pezzo ADE surface of
  type $D_4$, and $(W, C_W)$ is a del Pezzo ADE surface of type
  $\pD_3^-$. The index-$2$ cover corresponds to the $B_2$ diagram and
  we denote it $\pB_2^-$. 

  The index-$2$ cover $Z\to W$ is branched in $B_W$ and at the two
  torus-fixed points where $K_X+C_W$ is Cartier. The corresponding
  index-$1$ is branched at $B_W$ and at the \emph{other} two
  torus-fixed points.
\end{example}

\begin{example}
  In Figure~\ref{fig-square-triangle}, let $Q$ be the triangle with vertices $(0,0)$,
  $(2,0)$, $(2,2)$ and $C=C_1+C_2$ be the boundary curves passing
  through the sides through $(2,2)$. Then $Y=\bP^2$ and
  $B\in \cO(2)$.
  The surface $W$ is the quadratic cone $\bP(1,1,2)$. 
  The ADE-type of $(Y,C)$ is $A_1$ and the ADE-type
  of $(W,C_W)$ is $A_0^-$.  

  The index-$2$ cover $Z\to W$ is branched in $B_W$ and the long side
  of $C_{1,W}$ of $Q$ in $\bZ^2_\even$.  It has two
  $A_1$ singularities above $C_{1,W}\cap B_W$ and two more
  above the apex of $\bP(1,1,2)$.
  The corresponding index-$1$ cover of $W$ instead is branched in
  $B_W$ and at the apex of $\bP(1,1,2)$, and is isomorphic to
  $\bP^2$.
\end{example}

\subsection{Quotients by polytope involutions}
\label{sec:polytope-involutions}

Now consider an ADE polytope $Q$ which has an involution that in
some coordinates can be written as $\tau\colon (x,y) \to (x\inv,
-y)$. For the anticanonical ADE surface $X = \{z^2 + f(x,y)=0\}$ we
choose the involution $\ien\colon (x,y,z)\to (x\inv, -y,-z)$. 

In the $\wA_{2n-1}$ case, the involution that is centered at
$0\in \bZ_{4n}$ is
\begin{math}
  \ien\colon (x,y,z) \to (qx\inv, -y, -z).
\end{math}
This sends $\theta_i$ to $\theta_{4n-i}$, and $z^2+f$ is
$\ien$-invariant iff $c_i = c_{-i}$. Similarly, one can define
involutions centered at any $i$ with $4|i$.

\begin{lemma}
  There exist diagrams of the following types:
\begin{enumerate}
\item $A_{4n-1} : 2 =  \ppB_{2n}\subset \ppD_{2n+2}$ and  
  $\mA^-_{4n-3} : 2 =  \ppB^-_{2n-1}\subset \ppD_{2n+1}^-$
\item $\pA'_{4n-1} : 2 =  \ppB'_{2n} \subset \ppD'_{2n+2}$ 
\item $\wD_{4n} : 2 = {}_2\wB_{2n} \subset \wD_{2n+2}''$ 
\item $\wA_{4n-1} : 2 = {}_2\wC_{2n} \subset \wD_{2n+4}''''$ 
\end{enumerate}
\end{lemma}
\begin{proof}
  The conditions of Proposition~\ref{prop:two-covers} are immediate to
  check. The ADE types are easily found by locating the
  singularities in the nonklt locus of $(W,C_W)$ in
  \cite[Tables~2, 3]{alexeev17ade-surfaces}.
\end{proof}

\begin{example}
  Consider the case $A_{4n-1} : 2 = \pB_{2n}$. Then $Y=\bP(1,1,2n)$
  with the minimal resolution $\wY=\bF_{2n}$. The induced involution on
  $\wY$ has four fixed points, two on the $(-2n)$-section and two on a
  $(+2n)$-section. On the quotient of $\wY$ by the induced involution
  this gives four $A_1$ singularities. It follows that $W$ has three
  singularities, one on $C_W$ whose resolution graph is a chain of
  curves with $-E_i^2$ equal $(2,n+1,2)$ and two outside of
  $C_W$. From \cite[Table~2]{alexeev17ade-surfaces} we read off that
  the ADE type of $(W,C_W)$ is $\ppD'_{2n+2}$. 

  The simplest form of the equation of $X$ is $z^2+f$, where
\begin{displaymath}
  f= y^2 + x^{-2n} + \sum_{i=1}^{2n-1} a_ix^{-2n+i} + a_{2n} +
  \sum_{i=1}^{2n-1} a_{4n-i}x^{2n-i}
\end{displaymath}
with the involution 
\begin{math}
  \ien\colon (x,y,z) \to (x\inv, -y, -z).
\end{math}
This equation is symmetric iff $a_i = a_{4n-i}$ for
$i=1,\dotsc, 2n-1$, giving $2n$ free parameters. In alternative
coordinates $u=y+\sqrt{-1} z$, $v=y-\sqrt{-1} z$, the equation is
$uv+f=0$ so that the variable $v=-fu\inv$ can be eliminated, and the
involution is
\begin{math}
  (x,u)\to (x\inv,-u).
\end{math}
We note that for $A_{4n-3}$ a similar involution has a curve in the
fixed locus, so it is not of Enriques type. 
\end{example}

\smallskip

\section{KSBA stable degenerations of Enriques surfaces}
\label{sec:stable-models}

\subsection{Type III stable models of K3 surfaces}
\label{sec:type3-stable-k3} 

The Type III and Type II degenerations of K3 surfaces in $\ofhyp$,
i.e. of degree~$4$ K3 surfaces with a del Pezzo involution are described in detail
in the last section of \cite{alexeev2022mirror-symmetric}. We briefly
recall it, beginning with the Type III degenerations. There are two
$0$-cusps with the lattices $e^\perp/e=(18,2,0)_1$ and $(18,0,0)_1$
which were shown in Figure~\ref{fig-k3-cusps}.  At each of these cusps there
is a unique maximal degeneration. These are shown in
Figure~\ref{fig-maxdegs}. Note the uncanny resemblance the
Coxeter diagrams, shown in Figure~\ref{fig-k3-cusps}. The
similarities between the two figures become even more pronounced when
describing the non-maximal degenerations. 

\begin{figure}[htbp]
  \includegraphics[width=300pt]{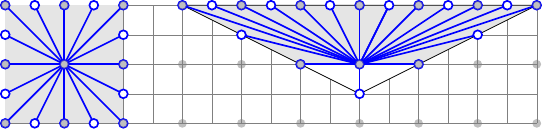} 
  \smallskip
  \caption{Maximal degenerations of K3 surfaces for $(18,2,0)_1$
    and $(18,0,0)_1$ cusps of $\fhyp$}
  \label{fig-maxdegs} 
\end{figure}

For the $(18,2,0)_1$-cusp, the maximal degeneration is a union of $16$
surfaces of $A_0^-$ type, which is $\bP^2$ with a del Pezzo involution such
that the quotient is the quadratic cone $\bP(1,1,2)$.  We may
symbolically write it as $(A_0^-\mA_0)^8$. This degeneration corresponds
to the empty subdiagram of $G_r(18,2,0)$.

For an elliptic subdiagram $G\subset G_r(18,2,0)$, each
\emph{relevant} component (i.e. not lying entirely in the interior of
the square) gives an ADE surface. Then the corresponding KSBA
degeneration is their union glued along double curves. The ADE
surfaces are obtained by smoothing some of the double curves in the
maximal degeneration; these edges correspond to the vertices in
$G$. All of the degenerations are of the ``pumpkin type'', see
\cite[Fig.~2]{alexeev2022mirror-symmetric}.

There is however a caveat: the $C_3$ diagram in the third row of
Figure~\ref{fig-maxconn-ell} should be treated instead as an $\pA_3'$
diagram. This is because the diagrams $G$ are supposed to be
subdiagrams of $G_2$, for the reflection group generated by the
$(-2)$-roots, and $G_r$ is the Coxeter diagram for the full
reflection group, which includes both $(-2)$ and $(-4)$-roots. There is a simple
dictionary to translate from one to another, see
\cite[Fig.~15]{alexeev2022mirror-symmetric}.

\medskip

At the $(18,0,0)_1$-cusp, the degenerations are of the ``smashed
pumpkin type'', as in  \cite[Fig.~2]{alexeev2022mirror-symmetric}. It
can be understood in the following way. Begin with a union of $18$ surfaces of
$A_0^-$ type, $\cup_i^{18} (V_i,D_i)$, where $V_i\simeq \bP^2$ with an
involution $(x,y,z)\to (x,y,-z)$. The ramification divisor on $V_i$ is
a line, and the boundary curves $D_1$, $D_2$ are a line and a
conic. Consider two neighboring $V_i$, $V_{i+1}$ that are glued along
a line $D_1$, so that $R\cap D_1=p$ is a point. 
Blow up this point in each of the surfaces to get $V'_i$ and
$V'_{i+1}$, both isomorphic to $\bF_1$. The strict preimage of $R'\cap
V'_i$ is now a fiber $f$ of $\bF_1$, and same for $V'_{i+1}$. Contract by
the linear system $|f|$. Then $V_i\cup V_{i+1}$ collapses to 
$\bP^1\cup \bP^1$ and the entire surface $\cup_{i=1}^{18} V_i$ which
previously was represented by a ``pumpkin'' is partially collapsed,
with the north and south poles colliding. 

For the non-maximal degenerations we begin with a Coxeter diagram
$G_r(19,1,1)$ as in \cite[Fig.~4.1]{alexeev2023stable-pair}. An
elliptic subdiagram $G$ of this Coxeter diagram, as in the case above, corresponds to a union of
ADE surfaces. We then perform the move described above
to partially collapse it. The edge between $V_i$ and
$V_{i+1}$ is always contracted, bringing the north and south poles of
the pumpkin together. The components $V_i$ and $V_{i+1}$ are collapsed
only if they are of the $A_0^-$ type, i.e. the conic $D_2$ in $\bP^2$
was not smoothed out. 

\subsection{Type II stable models of K3 surfaces}
\label{sec:type2-stable-k3}

The stable models in this case are very similar
to the Type III models described above. They correspond to
maximal parabolic subdiagrams $G\subset G_r$. After throwing away
irrelevant connected components of $G$, each of the remaining
components is a $\wA\wD\wE$ subdiagram, giving an $\wA\wD\wE$ surface.

\subsection{Type III stable models of Enriques surfaces}
\label{sec:type3-stable-en}

\begin{figure}[htpb]
  \centering
  \includegraphics[width=330pt]{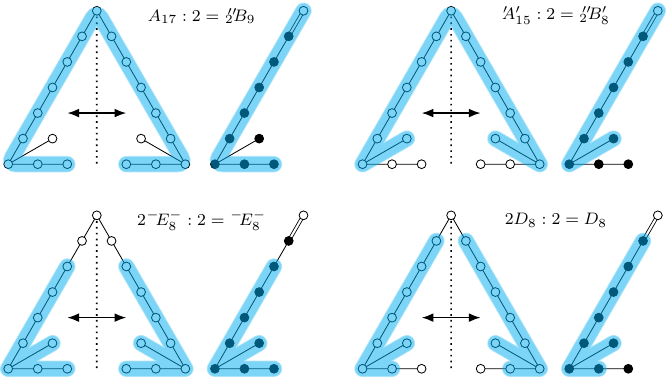}
  \caption{Max connected elliptic diagrams for $0$-cusp 2}
  \label{fig-maxconn-ell2}
\end{figure}

\begin{figure}[htpb]
  \centering
  \includegraphics[width=360pt]{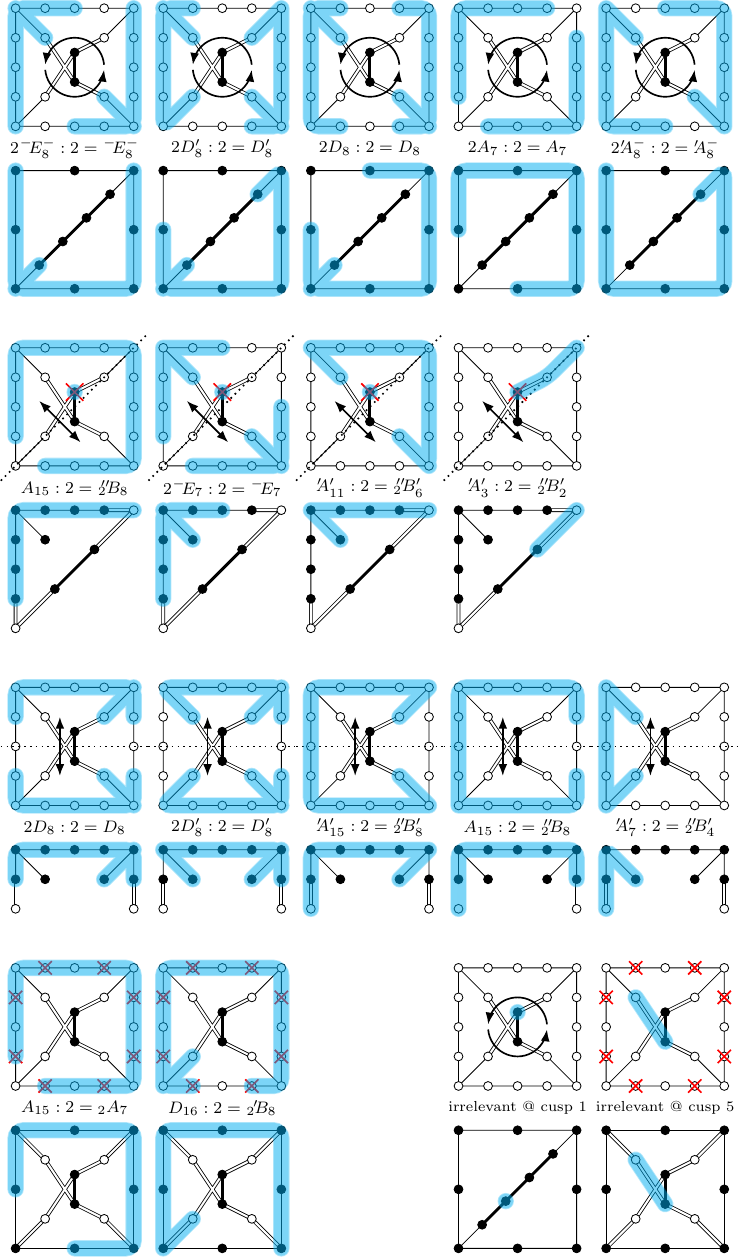}
  \caption{Max connected elliptic diagrams for $0$-cusps 1, 3, 4, 5}
  \label{fig-maxconn-ell}
\end{figure}

By Corollary~\ref{hk-cor} and the proof of Theorem~\ref{thm:main}, the description of the
KSBA stable limit of Enriques pairs $(Z,\epsilon R_Z)$ are now
straightforward: these are simply quotients of KSBA stable limits of
K3 pairs $(X,\epsilon R)$ by an Enriques involution. The latter acts
in different ways, depending on the $0$-cusp of $\fen$. The action is
determined by the folding of the Coxeter diagram, as in
Figures~\ref{fig-cusp2} and \ref{fig-cusps1345}. Let us spell them
out, representing the surface $(X,\epsilon R)$ by a sphere $S^2$.

\smallskip (1) At the cusp 1, the action on $S^2$ is antipodal, with
the quotient $\bR\bP^2$. So all irreducible components of
$X=\cup V_i$ are interchanged in pairs $V_i\simeq V_{\sigma(i)}$. Then
the normalization of $V_i\cup V_{\sigma(i)} / \ien$ is isomorphic to
the normalization of $V_i$.

\smallskip (2, 3, 4) At these cusps the action on $S^2$ is a reflection
which is different from the equatorial reflection defined by
$\idp$. Some of the components $V_i$ of $X$ come in pairs, and some
are fixed by $\ien$. The latter ones are the $B_n$ surfaces of
Section~\ref{sec:polytope-involutions}.

\smallskip (5) Here the action of $\ien$ on $S^2$ is the same as the
action of $\idp$. Each component $V_i$ is fixed by $\ien$ and the
quotients are the surfaces described in
Section~\ref{sec:toric-involutions}.

\smallskip

In Figures~\ref{fig-maxconn-ell2} and \ref{fig-maxconn-ell} we list
the maximal connected elliptic subdiagrams in the Coxeter diagrams,
for each of the five $0$-cusps of $\fen$. These then describe the
largest possible irreducible components in $X$ and $Z=X/\ien$. All
other irreducible components correspond to the subdiagrams 
of these maximal ones, which are preserved by the folding symmetry.

The surfaces are glued according to the Coxeter diagram.

\begin{example}
Consider the first surface $2\mE^-_8 : 2 = \mE^-_8$ in 
Figure~\ref{fig-maxconn-ell}. The degenerate Enriques surface is
irreducible and its normalization is an ADE surface of type
$\mE_8^-$. It is then glued to itself by an isomorphism $D_1\to
D_2$ between the two sides. 
\end{example}

The Coxeter diagrams in Figures~\ref{fig-cusp2} and
\ref{fig-cusps1345} also describe the ramification divisor $R_Z$ on
the Type III degenerations $Z$. The boundary of each Coxeter diagram
for the Cusps 1, 2, 3, 4, 5, i.e.~the image of the boundary of the
square or the triangle, represents the ramification divisor
$R_Z$. Thus, in Cusps 1 and 5, $R_Z$ is a cycle, and in 
the other three cusps it is a chain. 

\subsection{Type II stable models of Enriques surfaces}
\label{sec:type2-stable-en}

Similarly, the irreducible components of Type II degenerations are
described by the relevant components of the maximal parabolic subdiagrams
in the Coxeter diagrams. We listed them in Figures~\ref{fig-par2} and
\ref{fig-pars}. The folded Type II surfaces are described in
Sections \ref{sec:toric-involutions} (cusp 5) and
\ref{sec:polytope-involutions} (cusps 2, 3, 4).

\bibliographystyle{amsalpha}

\begin{thebibliography}{GHK15b}

\bibitem[ABE22]{alexeev2022compactifications-moduli}
Valery Alexeev, Adrian Brunyate, and Philip Engel, \emph{Compactifications of
  moduli of elliptic {K}3 surfaces: {S}table pair and toroidal}, Geom. Topol.
  \textbf{26} (2022), no.~8, 3525--3588. \MR{4562567}

\bibitem[AE22]{alexeev2022mirror-symmetric}
Valery Alexeev and Philip Engel, \emph{Compactifications of moduli spaces of
  {K3} surfaces with a nonsymplectic involution}, arXiv:2208.10383 (2022).

\bibitem[AE23]{alexeev2023compact}
\bysame, \emph{Compact moduli of {K}3 surfaces}, Ann. of Math. (2) \textbf{198}
  (2023), no.~2, 727--789. \MR{4635303}

\bibitem[AEH21]{alexeev2021nonsymplectic}
Valery Alexeev, Philip Engel, and Changho Han, \emph{Complete moduli of {K}3
  surfaces with a nonsymplectic automorphism}, Trans. Amer. Math. Soc., to
  appear (2021), arXiv:2110.13834.

\bibitem[AET23]{alexeev2023stable-pair}
Valery Alexeev, Philip Engel, and Alan Thompson, \emph{Stable pair
  compactification of moduli of {K}3 surfaces of degree 2}, J. Reine Angew.
  Math. \textbf{799} (2023), 1--56. \MR{4595306}

\bibitem[AT21]{alexeev17ade-surfaces}
Valery Alexeev and Alan Thompson, \emph{{ADE} surfaces and their moduli}, J.
  Algebraic Geometry \textbf{30} (2021), 331--405, arXiv:1712.07932.

\bibitem[BB66]{baily1966compactification-of-arithmetic}
W.~L. Baily, Jr. and A.~Borel, \emph{Compactification of arithmetic quotients
  of bounded symmetric domains}, Ann. of Math. (2) \textbf{84} (1966),
  442--528. \MR{0216035 (35 \#6870)}

\bibitem[Bir23]{birkar2023geometry-polarized}
Caucher Birkar, \emph{Geometry of polarised varieties}, Publ. Math. Inst.
  Hautes \'{E}tudes Sci. \textbf{137} (2023), 47--105. \MR{4588595}

\bibitem[CD89]{cossec1989enriques-surfaces}
Fran\c{c}ois~R. Cossec and Igor~V. Dolgachev, \emph{Enriques surfaces. {I}},
  Progress in Mathematics, vol.~76, Birkh\"{a}user Boston, Inc., Boston, MA,
  1989. \MR{986969}

\bibitem[CDL24]{cossec2024enriques-surfaces1}
Fran\c{c}ois~R. Cossec, Igor Dolgachev, and Christian Liedtke, \emph{Enriques
  surfaces {I}}, Springer Nature, 2024.

\bibitem[Cos83]{cossec1983projective-models}
Fran\c{c}ois~R. Cossec, \emph{Projective models of {E}nriques surfaces}, Math.
  Ann. \textbf{265} (1983), no.~3, 283--334. \MR{721398}

\bibitem[Cox35]{coxeter1935wythoffs-construction}
H.~S.~M. Coxeter, \emph{Wythoff's {C}onstruction for {U}niform {P}olytopes},
  Proc. London Math. Soc. (2) \textbf{38} (1935), 327--339. \MR{1576319}

\bibitem[dFKX17]{defernex2017dual-complex}
Tommaso de~Fernex, J\'{a}nos Koll\'{a}r, and Chenyang Xu, \emph{The dual
  complex of singularities}, Higher dimensional algebraic geometry---in honour
  of {P}rofessor {Y}ujiro {K}awamata's sixtieth birthday, Adv. Stud. Pure
  Math., vol.~74, Math. Soc. Japan, Tokyo, 2017, pp.~103--129. \MR{3791210}

\bibitem[EF21]{engel2021smoothings}
Philip Engel and Robert Friedman, \emph{Smoothings and rational double point
  adjacencies for cusp singularities}, J. Differential Geom. \textbf{118}
  (2021), no.~1, 23--100. \MR{4255071}

\bibitem[Eng18]{engel2018looijenga}
Philip Engel, \emph{Looijenga's conjecture via integral-affine geometry}, J.
  Differential Geom. \textbf{109} (2018), no.~3, 467--495. \MR{3825608}

\bibitem[Enr06]{enriques1906sopra-le-superficie}
Federigo Enriques, \emph{Sopra le superficie algebriche di bigenere uno}, Mem.
  Soc. Ital. Sci. \textbf{XIV} (1906), no.~3a, 327--352.

\bibitem[Fri15]{friedman2015on-the-geometry}
Robert Friedman, \emph{On the geometry of anticanonical pairs}, Preprint
  (2015), arXiv:1502.02560.

\bibitem[FS86]{friedman1986type-III}
Robert Friedman and Francesco Scattone, \emph{Type {${\rm III}$} degenerations
  of {$K3$} surfaces}, Invent. Math. \textbf{83} (1986), no.~1, 1--39.
  \MR{813580 (87k:14044)}

\bibitem[GH16]{gritsenko2016moduli-polarized}
V.~Gritsenko and K.~Hulek, \emph{Moduli of polarized {E}nriques surfaces}, K3
  surfaces and their moduli, Progr. Math., vol. 315, Birkh\"{a}user/Springer,
  [Cham], 2016, pp.~55--72. \MR{3524164}

\bibitem[GHK15a]{gross2015mirror-symmetry-for-log}
Mark Gross, Paul Hacking, and Sean Keel, \emph{Mirror symmetry for log
  {C}alabi-{Y}au surfaces {I}}, Publ. Math. Inst. Hautes \'{E}tudes Sci.
  \textbf{122} (2015), 65--168. \MR{3415066}

\bibitem[GHK15b]{gross2015moduli-of-surfaces}
\bysame, \emph{Moduli of surfaces with an anti-canonical cycle}, Compos. Math.
  \textbf{151} (2015), no.~2, 265--291. \MR{3314827}

\bibitem[Hor78a]{horikawa1978periods-enriques1}
Eiji Horikawa, \emph{On the periods of {E}nriques surfaces. {I}}, Math. Ann.
  \textbf{234} (1978), no.~1, 73--88. \MR{491725}

\bibitem[Hor78b]{horikawa1978periods-enriques2}
\bysame, \emph{On the periods of {E}nriques surfaces. {II}}, Math. Ann.
  \textbf{235} (1978), no.~3, 217--246. \MR{491726}

\bibitem[KK72]{kiernan1972satake-compactification}
Peter Kiernan and Shoshichi Kobayashi, \emph{Satake compactification and
  extension of holomorphic mappings}, Invent. Math. \textbf{16} (1972),
  237--248. \MR{310297}

\bibitem[KLSV18]{kollar2018remarks}
J\'{a}nos Koll\'{a}r, Radu Laza, Giulia Sacc\`a, and Claire Voisin,
  \emph{Remarks on degenerations of hyper-{K}\"{a}hler manifolds}, Ann. Inst.
  Fourier (Grenoble) \textbf{68} (2018), no.~7, 2837--2882. \MR{3959097}

\bibitem[KM98]{kollar1998birational-geometry}
J\'{a}nos Koll\'{a}r and Shigefumi Mori, \emph{Birational geometry of algebraic
  varieties}, Cambridge Tracts in Mathematics, vol. 134, Cambridge University
  Press, Cambridge, 1998, With the collaboration of C. H. Clemens and A. Corti,
  Translated from the 1998 Japanese original. \MR{1658959}

\bibitem[Kol23]{kollar2023families-of-varieties}
J\'{a}nos Koll\'{a}r, \emph{Families of varieties of general type}, Cambridge
  Tracts in Mathematics, vol. 231, Cambridge University Press, Cambridge, 2023.
  \MR{4566297}

\bibitem[Kul77]{kulikov1977degenerations-of-k3-surfaces}
Vik.~S. Kulikov, \emph{Degenerations of {$K3$} surfaces and {E}nriques
  surfaces}, Izv. Akad. Nauk SSSR Ser. Mat. \textbf{41} (1977), no.~5,
  1008--1042, 1199. \MR{0506296}

\bibitem[KX20]{kollar2019moduli-of-polarized}
J\'{a}nos Koll\'{a}r and Chen~Yang Xu, \emph{Moduli of polarized {C}alabi-{Y}au
  pairs}, Acta Math. Sin. (Engl. Ser.) \textbf{36} (2020), no.~6, 631--637.
  \MR{4110324}

\bibitem[Laz16]{laza2016ksba-compactification}
Radu Laza, \emph{The {KSBA} compactification for the moduli space of degree two
  {$K3$} pairs}, J. Eur. Math. Soc. (JEMS) \textbf{18} (2016), no.~2, 225--279.
  \MR{3459951}

\bibitem[LO21]{laza2021git-versus}
Radu Laza and Kieran O'Grady, \emph{G{IT} versus {B}aily-{B}orel
  compactification for {$K3$}'s which are double covers of {$\Bbb P^1\times\Bbb
  P^1$}}, Adv. Math. \textbf{383} (2021), Paper No. 107680, 63. \MR{4233275}

\bibitem[Loo76]{looijenga1976root}
Eduard Looijenga, \emph{Root systems and elliptic curves}, Inventiones
  mathematicae \textbf{38} (1976), no.~1, 17--32.

\bibitem[Loo78]{looijenga1978semi-universal}
\bysame, \emph{On the semi-universal deformation of a simple-elliptic
  hypersurface singularity. {II}. {T}he discriminant}, Topology \textbf{17}
  (1978), no.~1, 23--40. \MR{0492380}

\bibitem[Loo86]{looijenga1986new-compactifications}
\bysame, \emph{New compactifications of locally symmetric varieties},
  Proceedings of the 1984 {V}ancouver conference in algebraic geometry, CMS
  Conf. Proc., vol.~6, Amer. Math. Soc., Providence, RI, 1986, pp.~341--364.
  \MR{846027 (87m:32072)}

\bibitem[Loo03]{looijenga2003compactifications-defined2}
\bysame, \emph{Compactifications defined by arrangements. {II}. {L}ocally
  symmetric varieties of type {IV}}, Duke Math. J. \textbf{119} (2003), no.~3,
  527--588. \MR{2003125 (2004i:14042b)}

\bibitem[Mor81]{morrison1981semistable}
David~R. Morrison, \emph{Semistable degenerations of {E}nriques' and
  hyperelliptic surfaces}, Duke Math. J. \textbf{48} (1981), no.~1, 197--249.
  \MR{610184}

\bibitem[Nam85]{namikawa1985periods-of-enriques}
Yukihiko Namikawa, \emph{Periods of {E}nriques surfaces}, Math. Ann.
  \textbf{270} (1985), no.~2, 201--222. \MR{771979}

\bibitem[Nik79]{nikulin1979integer-symmetric}
V.~V. Nikulin, \emph{Integer symmetric bilinear forms and some of their
  geometric applications}, Izv. Akad. Nauk SSSR Ser. Mat. \textbf{43} (1979),
  no.~1, 111--177, 238. \MR{525944 (80j:10031)}

\bibitem[Pin77]{pinkham1977simple-ellitic}
H.~C. Pinkham, \emph{Simple elliptic singularities, {D}el {P}ezzo surfaces and
  {C}remona transformations}, Several complex variables ({P}roc. {S}ympos.
  {P}ure {M}ath., {V}ol. {XXX}, {P}art 1, {W}illiams {C}oll., {W}illiamstown,
  {M}ass., 1975), Proc. Sympos. Pure Math., vol. Vol. XXX, Part 1, Amer. Math.
  Soc., Providence, RI, 1977, pp.~69--71. \MR{441969}

\bibitem[PP81]{persson1981degeneration-of-surfaces}
Ulf Persson and Henry Pinkham, \emph{Degeneration of surfaces with trivial
  canonical bundle}, Ann. of Math. (2) \textbf{113} (1981), no.~1, 45--66.
  \MR{604042}

\bibitem[PS20]{peters2020on-k3-double}
Chris Peters and Hans Sterk, \emph{On {K}3 double planes covering {E}nriques
  surfaces}, Math. Ann. \textbf{376} (2020), no.~3-4, 1599--1628. \MR{4081124}

\bibitem[Sha81]{shah1981degenerations-of-k3}
Jayant Shah, \emph{Degenerations of {$K3$} surfaces of degree {$4$}}, Trans.
  Amer. Math. Soc. \textbf{263} (1981), no.~2, 271--308. \MR{594410
  (82g:14039)}

\bibitem[Ste91]{sterk1991compactifications-enriques1}
Hans Sterk, \emph{Compactifications of the period space of {E}nriques surfaces.
  {I}}, Math. Z. \textbf{207} (1991), no.~1, 1--36. \MR{1106810}

\bibitem[Sym03]{symington2003four-dimensions}
Margaret Symington, \emph{Four dimensions from two in symplectic topology},
  Topology and geometry of manifolds ({A}thens, {GA}, 2001), Proc. Sympos. Pure
  Math., vol.~71, Amer. Math. Soc., Providence, RI, 2003, pp.~153--208.
  \MR{2024634}

\bibitem[Vin72]{vinberg1972-groups-of-units}
\`E.~B. Vinberg, \emph{The groups of units of certain quadratic forms}, Mat.
  Sb. (N.S.) \textbf{87(129)} (1972), 18--36. \MR{0295193}

\bibitem[Vin75]{vinberg1973some-arithmetic}
\bysame, \emph{Some arithmetical discrete groups in {L}oba\v cevski\u\i\
  spaces}, Discrete subgroups of {L}ie groups and applications to moduli
  ({I}nternat. {C}olloq., {B}ombay, 1973), Oxford Univ. Press, Bombay, 1975,
  pp.~323--348. \MR{0422505 (54 \#10492)}

\end{thebibliography}

\def\cprime{$'$}
\providecommand{\bysame}{\leavevmode\hbox to3em{\hrulefill}\thinspace}
\providecommand{\MR}{\relax\ifhmode\unskip\space\fi MR }
\providecommand{\MRhref}[2]{%
  \href{http://www.ams.org/mathscinet-getitem?mr=#1}{#2}
}
\providecommand{\href}[2]{#2}

\end{document}